\documentclass[10pt,a4paper]{amsart}
\usepackage[utf8]{inputenc}
\usepackage{amsmath}
\usepackage{mathtools}
\usepackage{amsfonts}
\usepackage{amssymb}
\usepackage{amsmath}
\usepackage{paralist}
\usepackage{amsthm}
\usepackage{amscd}
\usepackage{tikz}
\usepackage{graphicx}
\usepackage{caption}
\usepackage{comment}

\newtheorem{theorem}{Theorem}[section]
\newtheorem{corollary}[theorem]{Corollary}

\newtheorem{lemma}[theorem]{Lemma}
\newtheorem{proposition}[theorem]{Proposition}
\newtheorem{conjecture}{Conjecture}

\theoremstyle{definition}
\newtheorem{definition}[theorem]{Definition}
\newtheorem{remark}[theorem]{Remark}

\providecommand{\skp}[2]{\langle#1,#2\rangle}

\newcommand{\vecIII}[3]{ 
\ensuremath{
\begin{pmatrix}
#1 \\ #2 \\ #3 \\
\end{pmatrix}}}

\newcommand{\matIII}[9]{
\ensuremath{ 
\begin{pmatrix}  
#1 & #2 & #3\\
#4 & #5 & #6\\
#7 & #8 & #9 
\end{pmatrix}}}

\providecommand{\sm}{\setminus}
\providecommand{\N}{\mathbb{N}}
\providecommand{\R}{\mathbb{R}}
\providecommand{\Z}{\mathbb{Z}}
\providecommand{\C}{\mathbb{C}}
\providecommand{\cN}{\mathcal{N}}
\providecommand{\eps}{\varepsilon}

\providecommand{\skp}[2]{\langle#1,#2\rangle}
\providecommand{\les}{\lesssim}

\DeclareMathOperator{\dist}{dist}

\renewcommand{\qed}{\hfill $\Box$}

\makeatletter
\@namedef{subjclassname@2020}{%
  \textup{2020} Mathematics Subject Classification}
\makeatother

\author{Rainer Mandel}
\email{rainer.mandel@kit.edu}
\author[Robert Schippa]{Robert Schippa*}
\email{robert.schippa@kit.edu}
\address{Department of Mathematics, Karlsruhe Institute of Technology, Englerstrasse 2, 76131 Karlsruhe, Germany}

\begin{document}
\title[Solutions for anisotropic Maxwell's equations]{Time-harmonic Solutions for Maxwell's equations in anisotropic media and Bochner--Riesz estimates with negative index for non-elliptic surfaces}

\begin{abstract}
  We solve time-harmonic Maxwell's equations in anisotropic, spatially homogeneous media in intersections of
$L^p$-spaces. The material laws are time-independent. The analysis requires Fourier restriction--extension estimates for perturbations of Fresnel's wave surface. 
This surface can be decomposed into finitely many components of the following three types: smooth surfaces with non-vanishing Gaussian curvature, smooth surfaces with
Gaussian curvature vanishing along one-dimensional submanifolds but without flat points, and surfaces with conical singularities. Our estimates are based on new Bochner--Riesz estimates
with negative index for non-elliptic surfaces.
\end{abstract}

\thanks{*Corresponding author}

\maketitle

\section{Introduction}

The purpose of this article is to prove the existence of solutions to the time-harmonic Maxwell's equations
and estimating the solutions (electromagnetic fields) in terms of the input data (currents) in $L^p$-spaces. Let
$(\mathcal{E},\mathcal{H}): \R \times \R^3 \to \R^3 \times \R^3$ denote the \emph{electric and magnetic
field}, $(\mathcal{D},\mathcal{B}): \R \times \R^3 \to \R^3 \times \R^3$ the \emph{displacement field} and \emph{magnetic induction}, and $(\mathcal{J}_e,\mathcal{J}_m): \R \times \R^3 \to \R^3 \times \R^3$ the \emph{electric and magnetic current}. Maxwell's equations in the absence of charges are given by
\begin{equation}
\label{eq:3dMaxwell}
\left\{ \begin{array}{cl}
\partial_t \mathcal{D} &= \nabla \times \mathcal{H} - \mathcal{J}_e, \quad \nabla \cdot \mathcal{D} = \nabla \cdot \mathcal{B} = \nabla \cdot \mathcal{J}_e = \nabla \cdot \mathcal{J}_m = 0, \\
\partial_t \mathcal{B} &= - \nabla \times \mathcal{E} + \mathcal{J}_m, \quad (t,x) \in \R \times \R^3.
\end{array} \right.
\end{equation}

We suppose that displacement and magnetic field are related with electric field and magnetic induction through time-independent and spatially homogeneous material laws. This leads to supplementing \eqref{eq:3dMaxwell} with
\begin{equation}
\label{eq:MaterialLaws}
\mathcal{D}(t,x) = \varepsilon \mathcal{E}(t,x), \quad \mathcal{B}(t,x) = \mu \mathcal{H}(t,x) , \quad \varepsilon \in \R^{3 \times 3}, \; \mu \in \R^{3 \times 3}.
\end{equation}
$\varepsilon$ is referred to as \emph{permittivity}, and $\mu$ is referred to as \emph{permeability}. Permittivity and permeability are positive-definite in classical physical applications. We suppose in the following that $\varepsilon$ and $\mu$ are diagonal matrices and write
\begin{equation}
\label{eq:DiagonalPermittivityPermeability}
\varepsilon = \text{diag}(\varepsilon_1,\varepsilon_2,\varepsilon_3), \quad \mu =
\text{diag}(\mu_1,\mu_2,\mu_3), \quad \varepsilon_i,\mu_j > 0.
\end{equation}
Maxwell's equations are invariant under change of basis, i.e., the transformations $X'(t,x) = MX(t,M^t x)$
for the involved vector fields with $M \in SO(3)$, and time-parity symmetry $(t,x) \to (-t,-x)$. Hence, the
more general case when $\varepsilon$ and $\mu$ are commuting positive-definite matrices,  or
equivalently: simultaneously orthogonally diagonalizable, reduces to
\eqref{eq:DiagonalPermittivityPermeability}. For physical explanations, we refer to
\cite{FeynmanLeightonSands1964,LandauLifschitz1990}.  The assumption $\nabla \cdot \mathcal{D} = 0$
corresponds to the absence of electrical charges and
$\nabla \cdot \mathcal{B} = 0$ translates to the absence of magnetic monopoles.
 Due to conservation of charges, the currents are likewise divergence-free. Since magnetic monopoles are hypothetical, $\mathcal{J}_m$ is vanishing for most applications. Here, we consider the more general case, which will highlight symmetry between $\mathcal{E}$ and $\mathcal{H}$. In this paper we
focus on the fully anisotropic case
\begin{equation}\label{eq:FullAnisotropy}
  \frac{\eps_1}{\mu_1}\neq \frac{\eps_2}{\mu_2}\neq \frac{\eps_3}{\mu_3} \neq \frac{\eps_1}{\mu_1}.
\end{equation}
%given that the easier isotropic and partially anisotropic cases have already been analyzed by the authors~

Upon considering the time-harmonic, monochromatic ansatz
\begin{equation}
\label{eq:TimeHarmonicAnsatz}
\left\{ \begin{array}{cl}
\mathcal{D}(t,x) &= e^{i \omega t} D(x), \quad \mathcal{B}(t,x) = e^{i \omega t} B(x), \\
\mathcal{J}_e(t,x) &= e^{i \omega t} J_{e}(x), \quad \mathcal{J}_m(t,x) = e^{i \omega t} J_{m}(x)
\end{array} \right.
\end{equation}
with $(D,B): \R^{3} \to \R^3 \times \R^3$, $(J_{e},J_{m}) : \R^3 \to \R^3 \times \R^3$ divergence-free, \eqref{eq:3dMaxwell} becomes
\begin{equation*}
%\label{eq:TimeHarmonicEquations}
\left\{ \begin{array}{cl}
i \omega D &= \nabla \times H - J_{e}, \quad \nabla \cdot J_{e} = \nabla \cdot J_{m} = 0, \\
i \omega B &= - \nabla \times E + J_{m}.
\end{array} \right.
\end{equation*}
With \eqref{eq:MaterialLaws} we arrive at the equations
\begin{equation}
\label{eq:TimeHarmonicEquationsII}
\left\{ \begin{array}{cl}
\nabla \times E + i \omega \mu H &= J_m, \quad \nabla \cdot J_m = \nabla \cdot J_e = 0, \\
\nabla \times H - i \omega \varepsilon E &= J_e.
\end{array} \right.
\end{equation}
Below $W^{m,p}(\R^d)$ denotes the $L^p$-based Sobolev space defined by
\begin{equation*}
W^{m,p}(\R^d) = \{ f \in L^p(\R^d) \, : \; \partial^\alpha f \in L^p(\R^d) \text{ for all } \alpha \in \N_0^d, \, |\alpha| \leq m \}.
\end{equation*}
 We prove the following:
\begin{theorem}
\label{thm:LpLqEstimatesTimeHarmonicMaxwell}
Let $1 \leq p_1, p_2, q \leq \infty$, $\eps,\mu\in\R^3$ as in 
\eqref{eq:DiagonalPermittivityPermeability},\eqref{eq:FullAnisotropy}
and $(J_e,J_m) \in L^{p_1}(\R^3) \cap L^{p_2}(\R^3)$ divergence-free. If
\begin{align}
\label{eq:ConditionsOnpq}
  \begin{aligned}
&\qquad \frac{1}{p_1} > \frac{3}{4}, \quad \frac{1}{q} < \frac{1}{4}, \quad \frac{1}{p_1} - \frac{1}{q} \geq \frac{2}{3}, \\
 &\text{ and } 0 \leq \frac{1}{p_2} - \frac{1}{q} \leq \frac{1}{3}, \quad (p_2,q) \notin \{(1,1), (3,\infty), (\infty,\infty) \},
 \end{aligned}
\end{align}
then, for any given $\omega\in\R\sm\{0\}$, there exists a distributional time-harmonic solution to fully
anisotropic Maxwell's equations \eqref{eq:TimeHarmonicEquationsII} that satisfies
\begin{equation}
\label{eq:ResolventEstimates}
\| (E,H) \|_{L^q(\R^3)} \lesssim_{p,q,\omega} \| (J_e,J_m) \|_{L^{p_1}(\R^3) \cap L^{p_2}(\R^3)}
\end{equation}
with locally uniform dependence with respect to $\omega \in \R \backslash \{0\}$.

If additionally $J_e,J_m\in L^q(\R^3)$, $q<\infty$, then $E,H\in W^{1,q}(\R^3)$
is a weak solution satisfying 
$$
 \|(E,H)\|_{W^{1,q}(\R^3)} \lesssim_{p,q,\omega}  \| (J_e,J_m) \|_{L^{p_1}(\R^3) \cap L^q(\R^3)}.
$$
\end{theorem}

We shall see that the Fourier multiplier derived by inverting \eqref{eq:TimeHarmonicEquationsII} for
$\omega\in \R$ is not well-defined in the sense of distributions. A common regularization is to consider
$\omega \in \C \backslash \R$ and derive estimates independent of $\text{dist}(\omega, \R)$. This program
was carried out in our previous works \cite{CossettiMandel2020,ResolventEstimates2d}, which were concerned with
isotropic, possibly inhomogeneous, respectively, partially anisotropic, but homogeneous media. The necessity of considering $(J_e,J_m)$
within intersections of $L^p$-spaces and the connection with resolvent estimates for the Half-Laplacian was
discussed in~\cite{ResolventEstimates2d}. In the present work we need to regularize
differently due to a more complicated behaviour of the involved Fourier symbols with respect to the change
$\omega\mapsto \omega+i\eps$. In other words, we do not prove a Limiting Absorption Principle in the
classical sense.

In the proof we will reduce the analysis to the case $\mu_1 = \mu_2 = \mu_3=1$ as in~\cite{Liess1991} in order
to simplify the notation. We will justify this step in Section~\ref{section:FresnelSurface}.
In the partially anisotropic case $\varepsilon_1 = \varepsilon_2 \neq \varepsilon_3$ the
matrix-valued Fourier multiplier associated with Maxwell's equations can be diagonalized easily and a
combination of Riesz transform estimates and resolvent estimates for the Half-Laplacian are used to prove
uniform bounds. In our fully anisotropic case~\eqref{eq:FullAnisotropy} 
this does not work at all. Instead of diagonalizing the symbol, we take the more direct approach of
inverting the matrix Fourier multiplier associated with~\eqref{eq:TimeHarmonicEquationsII}. Taking the Fourier
transform in $\R^3$, denoting with $\xi \in \R^3$ the dual variable of $x \in \R^3$, and the vector-valued
Fourier transform of $E$ with $\hat{E}$, likewise for the other vector-valued quantities, we find
that~\eqref{eq:TimeHarmonicEquationsII} is equivalent to
\begin{equation}
\label{eq:TimeHarmonicEquationsFourier}
\left\{ \begin{array}{cl}
i b(\xi) \hat{E}(\xi) + i \omega \mu \hat{H}(\xi) &= \hat{J}_m, \quad \xi \cdot \hat{J}_m = \xi \cdot \hat{J}_e = 0, \\
i b(\xi) \hat{H}(\xi) - i \omega \varepsilon \hat{E}(\xi) &= \hat{J}_e.
\end{array} \right.
\end{equation}
In the above display, we denote
\begin{equation*}
%\label{eq:RotMatrix}
(\nabla \times f) \widehat (\xi) = i b(\xi) \hat{f}(\xi), \quad b(\xi) = 
\begin{pmatrix}
0 & -\xi_3 & \xi_2 \\
\xi_3 & 0 & -\xi_1 \\
-\xi_2 & \xi_1 & 0
\end{pmatrix}
.
\end{equation*}
In the first step, we use the block structure to show that solutions to
\eqref{eq:TimeHarmonicEquationsFourier} solve the following two $3 \times 3$-systems of second order:
\begin{proposition}
\label{prop:ReductionTimeHarmonicMaxwell}
If $(E,H) \in \mathcal{S}'(\R^3)^2$ solve \eqref{eq:TimeHarmonicEquationsFourier}, then the following holds true: \begin{equation}
\label{eq:ReductionFourierMaxwell}
\left\{ \begin{array}{cl}
(M_E(\xi) - \omega^2) \hat{E} &= - i \omega \varepsilon^{-1} \hat{J}_e + i \varepsilon^{-1} b(\xi) \mu^{-1} \hat{J}_m, \\
(M_H(\xi) - \omega^2) \hat{H} &= i \mu^{-1} b(\xi) \varepsilon^{-1} \hat{J}_e + i \omega \mu^{-1} \hat{J}_m.
\end{array} \right.
\end{equation}
Here,
\begin{align*}
M_E(\xi) = - \varepsilon^{-1} b(\xi) \mu^{-1} b(\xi), \qquad M_H(\xi) = - \mu^{-1} b(\xi) \varepsilon^{-1} b(\xi).
\end{align*}
\end{proposition}

The proof of the proposition follows from rewriting~\eqref{eq:TimeHarmonicEquationsFourier} as
\begin{equation*}
%\label{eq:FirstOrderSystem}
\begin{pmatrix}
-i \omega \varepsilon & i b(\xi) \\
i b(\xi) & i \omega \mu
\end{pmatrix}
\begin{pmatrix}
\hat{E} \\ \hat{H}
\end{pmatrix}
= 
\begin{pmatrix}
\hat{J}_{e} \\ \hat{J}_{m}
\end{pmatrix}
\end{equation*}
and multiplying this equation with  
\begin{equation}
\label{eq:Diagonalizer}
\begin{pmatrix}
-i \omega \varepsilon^{-1} & i \varepsilon^{-1} b(\xi) \mu^{-1} \\
i \mu^{-1} b(\xi) \varepsilon^{-1} & i \omega \mu^{-1}
\end{pmatrix}
. 
\end{equation}
Notice, however, that~\eqref{eq:TimeHarmonicEquationsFourier} and~\eqref{eq:ReductionFourierMaxwell}
are not equivalent because the symmetrizer~\eqref{eq:Diagonalizer} has a nontrivial kernel. A lengthy, but
straight-forward computation reveals
\begin{equation}
\label{eq:Determinant}
  p(\omega,\xi)
  := \det (M_E(\xi) - \omega^2) = \det(M_H(\xi) - \omega^2) = - \omega^2 (\omega^4 - \omega^2 q_0(\xi) + q_1(\xi)),
\end{equation}
where
\begin{align*}
q_0(\xi) &= \xi_1^2 \big( \frac{1}{\varepsilon_2 \mu_3} + \frac{1}{\mu_2 \varepsilon_3} \big) + \xi_2^2 \big( \frac{1}{\varepsilon_1 \mu_3} + \frac{1}{\mu_1 \varepsilon_3} \big) + \xi_3^2 \big( \frac{1}{\varepsilon_1 \mu_2} + \frac{1}{\varepsilon_2 \mu_1} \big), \\
q_1(\xi) &= \frac{1}{\varepsilon_1 \varepsilon_2 \varepsilon_3 \mu_1 \mu_2 \mu_3} (\varepsilon_1 \xi_1^2 + \varepsilon_2 \xi_2^2 + \varepsilon_3 \xi_3^2) (\mu_1 \xi_1^2 + \mu_2 \xi_2^2 + \mu_3 \xi_3^2).
\end{align*}
In the case $\mu_1=\mu_2=\mu_3>0$ this corresponds to~\cite[Eq.~(1.4)]{Liess1991} by Liess.

\medskip

From Proposition~\ref{prop:ReductionTimeHarmonicMaxwell} we infer that solutions to anisotropic Maxwell's
equations can be found provided that the mapping properties of the Fourier multiplier with symbol
$p^{-1}(\omega,\xi)$ or, actually, an adequate regularization of this, can be controlled. The first step of
this analysis is to develop a sound understanding of the geometry of 
$S:=\{\xi\in\R^3: p(\omega,\xi)=0\}$, with an emphasis on its principal curvatures. This has essentially been
carried out by Darboux~\cite{Darboux1993} and Liess~\cite[Appendix]{Liess1991}. We devote
Section~\ref{section:FresnelSurface} to recapitulate these facts along with some computational
details that were omitted in~\cite{Liess1991}.  
$S$ is known as Fresnel's wave surface, which was previously described, e.g., in \cite{Darboux1993,Liess1991,Knoerrer1986,FladtBaur1975}. We refer
to Figure~\ref{fig:Fresnel} for visualizations. Despite its seemingly complicated structure, 
this surface can be perceived as non-smooth deformation of the doubly covered sphere in $\R^3$. For the involved algebraic computations we provide a MAPLE\textsuperscript{TM} sheet for verification.

\medskip

We turn to a discussion of the regularization of $p(\omega,\xi)^{-1}$. Motivated by Cramer's rule, we
multiply~\eqref{eq:ReductionFourierMaxwell} with the adjugate matrices and divide by $p(\omega,\xi) + i
\delta$. This leads us to approximate solutions $(E_\delta,H_\delta)$. We postpone the precise
definition to Section~\ref{section:Reductions}. 
The main part of the proof of Theorem~\ref{thm:LpLqEstimatesTimeHarmonicMaxwell} is then to show uniform
bounds in $\delta \neq 0$:
\begin{equation*}
\| (E_\delta,H_\delta) \|_{L^q(\R^3)} \lesssim \| (J_e,J_m) \|_{L^{p_1}(\R^3) \cap L^{p_2}(\R^3)}
\end{equation*}
for $q$, $p_1$, $p_2$ as in Theorem~\ref{thm:LpLqEstimatesTimeHarmonicMaxwell}. In
Section~\ref{section:Reductions} we shall see how this allows us to infer the existence of distributional solutions to~\eqref{eq:TimeHarmonicEquationsII} and how the limits can be understood as
principal value distribution and delta distribution for Fresnel's wave surface in Fourier space. Moreover, the distributional solutions are weak solutions provided that the currents have sufficiently high integrability.

\medskip

We point out the connection to Bochner-Riesz operators of negative index and seemingly digress for a moment to explain key points for these operators. For $0 < \alpha <1$, consider the Bochner-Riesz operator with negative index given by
\begin{equation}
\label{eq:BochnerRieszNegative}
S^\alpha f (x) =  \frac{C_d}{\Gamma(1-\alpha)} \int_{\R^d} e^{ix.\xi} (1-|\xi|^2)^{-\alpha}_+ \hat{f}(\xi) d\xi.
\end{equation}
$C_d$ denotes a dimensional constant, $\Gamma$ denotes the Gamma function, and $x_+ = \max(x,0)$. For $1 \leq \alpha \leq (d+1)/2$, $S^\alpha$ is explained by analytic continuation. The body of literature concerned with Bochner-Riesz estimates with negative index is huge, see, e.g., \cite{Sogge1986,Boerjeson1986,Gutierrez2000,ChoKimLeeShim2005,KwonLee2020}. In Section \ref{section:BochnerRieszEstimates} we give a more exhaustive overview.
For $\alpha = 1$, we find
\begin{equation*}
%\label{eq:RestrictionExtensionOperator}
S^\alpha f (x ) = C_d \int_{\mathbb{S}^{d-1}} e^{i x. \xi} \hat{f}(\xi) d\sigma(\xi) = C_d \int_{\R^d} e^{ix.\xi} \delta(|\xi|^2 - 1) \hat{f}(\xi) d\xi
\end{equation*}
because the distribution in \eqref{eq:BochnerRieszNegative} for $\alpha = 1$ coincides with the delta
distribution up to a factor. Estimates for such Fourier restriction-extension operators are the
backbone of the Limiting Absorption Principle for the Helmholtz equation (cf. \cite{Gutierrez2004}). It
turns out that we need more general Fourier restriction-extension estimates than the ones associated with
elliptic surfaces because the Gaussian curvature of the Fresnel surface $S$ changes sign, as we shall see
in Section~\ref{section:FresnelSurface}. We take the opportunity to prove estimates for generalized
Bochner-Riesz operators of negative index for non-elliptic surfaces as the associated Fourier restriction-extension operators will be important in the proof of Theorem \ref{thm:LpLqEstimatesTimeHarmonicMaxwell}. 

\medskip

To describe our results in this direction, let $d \geq 3$ and $S = \{ (\xi',\psi(\xi')) : \, \xi' \in
[-1,1]^{d-1} \}$ be a smooth surface with $k \in \{1,\ldots,d-1 \}$ principal curvatures bounded from below. The case $d=2$ was disclosed by Bak \cite{Bak1997} and Guti\'errez \cite{Gutierrez2000}. Let
\begin{equation*}
%\label{eq:GeneralizedBochnerRiesz}
(T^\alpha f) \widehat (\xi) = \frac{1}{\Gamma(1-\alpha)} \frac{\chi(\xi')}{(\xi_d - \psi(\xi'))_+^\alpha}
\hat{f}(\xi), \; \chi \in C^\infty_c([-1,1]^{d-1}), \; 0 < \alpha < \frac{k+2}{2}.
\end{equation*}
In the following theorem, we show $L^p$-$L^q$-bounds
\begin{equation}
\label{eq:GeneralizedBochnerRieszEstimate}
\| T^\alpha f \|_{L^q(\R^d)} \lesssim \| f \|_{L^p(\R^d)}.
\end{equation}
 within a pentagonal region (see Figure~\ref{fig:Riesz1})
\begin{equation*}
( \frac{1}{p}, \frac{1}{q} ) \in \text{conv}^0 (C_{\alpha,k}, B_{\alpha,k},B'_{\alpha,k},C_{\alpha,k}',A ), \qquad A:=(1,0).
\end{equation*}
For $0<\alpha< \frac{k+2}{2}$, let
\begin{equation}
\label{eq:RangeGeneralizedBochnerRiesz}
\mathcal{P}_{\alpha}(k) = 
\left\{ (x,y) \in [0,1]^2 : \, x > \frac{k+2\alpha}{2(k+1)}, \; y < \frac{k+2-2\alpha}{2(k+1)}, \; x - y \geq
\frac{2\alpha}{k+2} \right\}.
\end{equation}
For two points $X$, $Y \in [0,1]^2$, let 
\begin{align*}
[X,Y] &= \{ Z \in [0,1]^2 \, : \, Z = \lambda X + (1-\lambda)Y \text{ for some } \lambda \in [0,1] \}, \\
\text{ and } (X,Y] &= [X,Y] \backslash \{X\}, \; [X,Y) = [X,Y] \backslash \{Y \}, \; (X,Y) = [X,Y] \backslash \{X,Y\}.
\end{align*}

At its inner endpoints $B_{\alpha,k}$, $B'_{\alpha,k}$, we show restricted weak bounds
\begin{equation}
\label{eq:WeakTypeBoundIII}
\| T^\alpha f \|_{L^{q,\infty}(\R^d)} \lesssim \| f \|_{L^{p,1}(\R^d)},
\end{equation}
and on part of its boundary, we show weak bounds
\begin{align}
\label{eq:WeakTypeBound}
\| T^\alpha f \|_{L^q(\R^d)} &\lesssim \| f \|_{L^{p,1}(\R^d)}, \\
\label{eq:WeakTypeBoundII}
\| T^\alpha f \|_{L^{q,\infty}(\R^d)} &\lesssim \| f \|_{L^p(\R^d)}.
\end{align}
\begin{figure}
\centering
\begin{tikzpicture}[scale=0.6] 
\draw[->] (0,0) -- (13,0); \draw[->] (0,0) -- (0,13);
\draw (0,0) --(12,12); \draw(0,12) -- (12,12); \draw (12,12) -- (12,0);
\coordinate (E) at (6,3);
\coordinate [label=left:$\frac{k}{2(k+2)}$] (EX) at (0,3);
\coordinate [label=above:$B_{\alpha_2,k}$] (B) at (7,2.5);
\coordinate [label=above:$B'_{\alpha_2,k}$] (B') at (9.5,5);
\coordinate [label=below:$C_{\alpha_2,k}$] (C) at (7,0);
\coordinate [label=right:$C'_{\alpha_2,k}$] (C') at (12,5);
\coordinate [label=above:$B_{\alpha_1,k}$] (U) at (5,3);
\coordinate [label=above:$B'_{\alpha_1,k}$] (U') at (9,7);
\coordinate [label=below:$C_{\alpha_1,k}$] (V) at (5,0);
\coordinate [label=right:$C'_{\alpha_1,k}$] (V') at (12,7);
\coordinate [label=below:$A$] (XX) at (12,0);
\coordinate [label=left:$\frac{1}{2}$] (Y) at (0,6);
\coordinate [label=left:$1$] (YY) at (0,12);
\coordinate [label=below:$\frac{1}{2}$] (X) at (6,0);
\coordinate [label=left:$\frac{1}{q}$] (YC) at (0,9);
\coordinate [label=below:$\frac{1}{p}$] (XC) at (9,0);
\draw [dotted] (EX) -- (E); \draw [dotted] (E) -- (X); \draw [dotted] (E) -- (XX); \draw [dotted] (XX) -- (YY);
\draw [help lines] (C) -- (B); \draw (B) -- (B');
\draw [help lines] (B') -- (C'); \draw [help lines] (U) -- (V); \draw [help lines] (U') -- (V');
\draw (U) -- (U');
\foreach \point in {(C),(C'),(B),(B'),(U),(U'),(V),(V'),(XX)}
\fill [black, opacity = 1] \point circle (3pt);
\end{tikzpicture}
 \caption{Riesz diagram for Theorem~\ref{thm:GeneralizedBochnerRieszEstimates} with 
  $\alpha_1 < \frac{1}{2} < \alpha_2$.
 }
  \label{fig:Riesz1}
\end{figure}
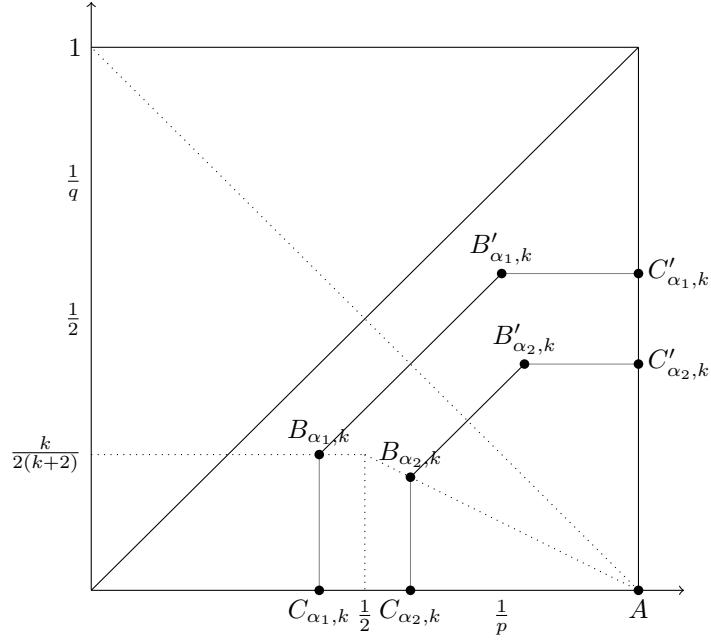

\begin{theorem}
\label{thm:GeneralizedBochnerRieszEstimates}
Let $1 \leq p, q \leq \infty$ and $d\in\N,d\geq 3$. 
%$T^\alpha: \mathcal{S}' \to \mathcal{S}'$ is well-defined in the sense of generalized functions, and 
\begin{itemize}
\item[(i)] For $\frac{1}{2} \leq \alpha < \frac{k+2}{2}$ let
\begin{align*}
B_{\alpha,k} &= \big( \frac{k+2\alpha}{2(k+1)}, \frac{k(k+2-2\alpha)}{2(k+1)(k+2)} \big), &C_{\alpha,k} = \big( \frac{k+2\alpha}{2(k+1)}, 0 \big), \\
B_{\alpha,k}' &= \left( \frac{ k^2 + 2(2+\alpha)k +4}{2(k+1)(k+2)}, \frac{k+2-2\alpha}{2(k+1)} \right), &C_{\alpha,k}' = (1, \frac{k+2-2\alpha}{2(k+1)}).
\end{align*}
\eqref{eq:GeneralizedBochnerRieszEstimate} holds true for $(\frac{1}{p},\frac{1}{q}) \in \mathcal{P}_{\alpha}(k)$ defined in
\eqref{eq:RangeGeneralizedBochnerRiesz}.\\
For $\alpha > \frac{1}{2}$, we find estimates \eqref{eq:WeakTypeBound} to hold for $(\frac{1}{p},\frac{1}{q})
\in (B_{\alpha,k}, C_{\alpha,k}]$; \eqref{eq:WeakTypeBoundII} for $(\frac{1}{p}, \frac{1}{q}) \in
(B'_{\alpha,k},C_{\alpha,k}']$, and \eqref{eq:WeakTypeBoundIII} for $(\frac{1}{p},\frac{1}{q}) \in \{
B_{\alpha,k}, B'_{\alpha,k} \}$.
\item[(ii)] For $0<\alpha < \frac{1}{2}$ let
\begin{align*}
B_{\alpha,k} &= \big( \frac{d-1+2\alpha}{2d}, \frac{k}{2(2+k)} \big), &C_{\alpha,k} = \big( \frac{d-1+2\alpha}{2d}, 0 \big), \\ \quad B_{\alpha,k}' &= \big( \frac{4+k}{2(2+k)}, \frac{d+1-2\alpha}{2d} \big), &C_{\alpha,k}' = \big(1, \frac{d+1-2\alpha}{2d} \big).
\end{align*}
\eqref{eq:GeneralizedBochnerRieszEstimate} holds true for
\begin{equation*}
\frac{1}{p} > \frac{d-1+2\alpha}{2d}, \quad \frac{1}{q} < \frac{d+1-2\alpha}{2d}, \quad \frac{1}{p} - \frac{1}{q} \geq \frac{2(d-1+2\alpha) + k (2\alpha-1)}{2d(2+k)}.
\end{equation*}
Furthermore, we find estimates \eqref{eq:WeakTypeBound} to hold for $(\frac{1}{p},\frac{1}{q}) \in
(B_{\alpha,k}, C_{\alpha,k}]$; \eqref{eq:WeakTypeBoundII} for $(\frac{1}{p}, \frac{1}{q}) \in
(B'_{\alpha,k},C_{\alpha,k}']$, and \eqref{eq:WeakTypeBoundIII} for $(\frac{1}{p},\frac{1}{q}) \in \{
B_{\alpha,k}, B'_{\alpha,k} \}$.
\end{itemize}
For any $\alpha$ the constant in \eqref{eq:GeneralizedBochnerRieszEstimate}-\eqref{eq:WeakTypeBoundII} depends on the lower bounds of
the principal curvatures and $\| \chi \|_{C^N}$ and $\| \psi \|_{C^N}$ for $N=N(p,q,d)$. In particular it is
stable under smooth perturbations of $\chi$ and $\psi$.
\end{theorem}
The proof is based on the decay of the Fourier transform of the surface measure on $S$ (cf.
\cite{Littman1963}, \cite[Section~VIII.5.8]{Stein1993}) and convenient decompositions of the distribution
$\frac{1}{\Gamma(1-\alpha)} x^{-\alpha}_+$ (cf. \cite[Section~3.2]{Hoermander2003},
\cite[Lemma~2.1]{ChoKimLeeShim2005}). We also show that the strong bounds are sharp for $\alpha \geq
\frac{1}{2}$. In the elliptic case the currently best results were shown by Kwon--Lee
\cite[Section~2.6]{KwonLee2020}.
This also shows that our strong bounds are not sharp for $\alpha < \frac{1}{2}$. We refer to Section
\ref{section:BochnerRieszEstimates} for further discussion.
 
To describe the remainder of our analysis, we recall important properties of the Fresnel surface. Up to arbitrary neighbourhoods of four singular points, the surface is a smooth compact manifold with two connected components. The Gaussian curvature vanishes precisely along the
so-called Hamiltonian circles on the outer sheet. However, the surface is never flat, i.e., there is always a principal section
away from zero. Around the singular points, the surface looks conical and ceases to be a smooth
manifold.

We briefly explain how this leads to an analysis of the Fourier multiplier
$(p(\omega,\xi) + i\delta)^{-1}$, $\omega \in \R \backslash \{ 0 \}$, $0 < |\delta| \ll 1$. We recall that solutions
to time-harmonic Maxwell's equations are constructed by considering $\delta \to 0$ with bounds independent of
$\delta$. The non-resonant contribution of $\{\xi\in\R^3 \, : \, |p(\omega,\xi)| \geq t_0 \}$, $t_0>0$ away from
Fresnel's wave surface is estimated by Mikhlin's theorem and standard estimates for Bessel potentials. This
high-frequency part of the solutions is responsible for the condition $ 0 \leq \frac{1}{p_2} -
\frac{1}{q} \leq \frac{1}{3}$ in~\eqref{eq:ConditionsOnpq}. This contribution was called global in
\cite{ResolventEstimates2d}. We refer to \cite[Section~3]{ResolventEstimates2d} for further explanation how
this contribution does not allow for an estimate $\| (E,H) \|_{L^q(\R^3)} \lesssim \| (J_e,J_m)
\|_{L^p(\R^3)}$. 

After smoothly cutting away the global contribution, we focus on estimates for the multiplier $(p(\omega,\xi)
+ i \delta)^{-1}$ in a neighbourhood $\{ |p(\omega,\xi)| \leq t_0 \}$ near the surface. It
turns out that around the smooth elliptic part with Gaussian curvature bounded away from zero, we can use the
estimates for the Bochner-Riesz operator from Theorem~\ref{thm:GeneralizedBochnerRieszEstimates} for
$d=3,k=2,\alpha=1$. However, there is also a smooth non-elliptic part where the modulus of the Gaussian
curvature is small and vanishes precisely along the Hamiltonian circles. Here,
Theorem~\ref{thm:GeneralizedBochnerRieszEstimates} applies for $d=3,k=1,\alpha=1$. In the corresponding
analyis of the multiplier $(p(\omega,\xi)+i\delta)^{-1}$ we foliate the neighbourhoods of the Fresnel surface
by level sets of $p(\omega,\xi)$.
The contributions of the single layers are estimated with the Fourier restriction-extension theorem. In the analysis we use decompositions in Fourier space generalizing arguments of Kwon--Lee \cite[Section~4]{KwonLee2020}, where the decompositions were adapted to the sphere.

 For the contribution coming from neighbourhoods of the four isolated conical singularities, we will apply
Theorem~\ref{thm:GeneralizedBochnerRieszEstimates} once more for $d=3,k=1,\alpha=1$. On a technical level,
a major difference compared to the other regions comes from the fact that the cone is not a smooth manifold: we use an additional Littlewood-Paley decomposition and scaling to uncover its mapping properties. Jeong--Kwon--Lee \cite{JeongKwonLee2016} previously applied related arguments to analyze Sobolev inequalities for second degree non-elliptic operators.

%Previously, Bonheure--Casteras and the first author \cite{BonheureCasterasMandel..} analyzed Fourier multipliers $(Q(\xi) + i \varepsilon)^{-1}$ for fourth order polynomials $Q$ by partial fraction decomposition, which, for a restricted range of parameters, gave similar mapping properties as due to Guti\'errez \cite{Gutierrez2004}. 
We further mention the very recent preprint by Cast\'eras--F\"oldes \cite{CasterasFoldes2021} (see also \cite{BonheureCasterasMandel2019}). In \cite{CasterasFoldes2021} $L^p$-mapping properties of Fourier multipliers $(Q(\xi) + i \varepsilon)^{-1}$ for fourth order polynomials $Q$ were analyzed in the context of traveling waves for nonlinear equations. The analysis in \cite{CasterasFoldes2021} does not cover surfaces $\{Q(\xi) = 0 \}$ containing singular points, and the $L^p$-$L^q$-boundedness range stated in \cite[Theorem~3.3]{CasterasFoldes2021} is strictly smaller than in the corresponding results given in Theorem \ref{thm:GeneralizedBochnerRieszEstimates}.
\vspace*{0.5cm}

\emph{Outline of the paper.} In Section \ref{section:Reductions} we carry out reductions to prove
Theorem~\ref{thm:LpLqEstimatesTimeHarmonicMaxwell}.  We anticipate the uniform estimates of the regularized
solutions that we will prove in Sections \ref{section:CurvatureBoundedBelow} -
\ref{section:SingularPoints}, by which we finish the proof of Theorem~\ref{thm:LpLqEstimatesTimeHarmonicMaxwell}.  
In Section~\ref{section:FresnelSurface} we recall the relevant geometric properties of the Fresnel surface
and reduce our analysis to the case $\omega=\mu_1 = \mu_2 = \mu_3 = 1$. In
Section~\ref{section:BochnerRieszEstimates} we recall results on Bochner-Riesz estimates with negative index
for elliptic surface and extend those to estimates for a class of more general nondegenerate surfaces by
proving Theorem~\ref{thm:GeneralizedBochnerRieszEstimates}.
In Section~\ref{section:CurvatureBoundedBelow} we use these estimates to uniformly bound
solutions to \eqref{eq:TimeHarmonicAnsatz} corresponding to the smooth part of the Fresnel
surface. In Section~\ref{section:SingularPoints} we finally estimate the contribution with Fourier support
close to the four singular points.

\section{Reduction to multiplier estimates related to the Fresnel surface}
\label{section:Reductions}

The purpose of this section is to carry out the reductions indicated in the Introduction. We first define
suitable approximate solutions $(E_\delta,H_\delta)$ and present estimates for those related to 
the different parts of the Fresnel surface and away from the Fresnel surface. With these estimates
at hand, to be shown in the upcoming sections, we finish the
proof of Theorem~\ref{thm:LpLqEstimatesTimeHarmonicMaxwell}. At the end of the section we give explicit
formulae for the solution.

We work with the following convention for the Fourier transform: For $f \in \mathcal{S}(\R^d)$ the Fourier transform is defined by
\begin{equation*}
\hat{f}(\xi) = \int_{\R^d} e^{-ix.\xi} f(x) dx
\end{equation*}
and as usually extended by duality to $\mathcal{S}'(\R^d)$. The Fourier inversion formula reads for $f \in
\mathcal{S}(\R^d)$
\begin{equation*}
f(x) = (2 \pi)^{-d} \int_{\R^d} e^{ix.\xi} \hat{f}(\xi) d\xi.
\end{equation*}

\subsection{Approximate solutions}

By Proposition~\ref{prop:ReductionTimeHarmonicMaxwell} the original anisotropic Maxwell system leads to the
following second order $3 \times 3$-system for $E$ and $H$
 \begin{equation} \label{eq:MaxwellSystem}
\left\{ \begin{array}{cl}
(M_E(\xi) - \omega^2) \hat{E} &= - i \omega \varepsilon^{-1} \hat{J}_e + i \varepsilon^{-1} b(\xi) \mu^{-1} \hat{J}_m, \\
(M_H(\xi) - \omega^2) \hat{H} &= i \mu^{-1} b(\xi) \varepsilon^{-1} \hat{J}_e + i \omega \mu^{-1} \hat{J}_m
\end{array} \right.
\end{equation}
 where $M_E(\xi) = - \varepsilon^{-1} b(\xi) \mu^{-1} b(\xi)$ and $M_H(\xi) = - \mu^{-1} b(\xi)
 \varepsilon^{-1} b(\xi)$.  From~\eqref{eq:Determinant} we recall 
\begin{equation*}
  p(\omega,\xi)
  = \det (M_E(\xi) - \omega^2) = \det(M_H(\xi) - \omega^2) = - \omega^2 (\omega^4 - \omega^2 q_0(\xi) + q_1(\xi)),
\end{equation*}
for the polynomials $q_0,q_1$ as defined there. Inverting $M_E(\xi)-\omega^2$ using Cramer's rule, we
find for all $\xi\in\R^3$ such that $p(\omega,\xi)\neq 0$:
\begin{align} \label{eq:Inverses}
  \begin{aligned}
  (M_E(\xi)-\omega^2)^{-1}
  &= \frac{1}{p(\omega,\xi)} \text{adj}(M_E(\xi)-\omega^2) 
  = \frac{1}{\eps_1\eps_2\eps_3 p(\omega,\xi)} Z_{\eps,\mu}(\xi) \eps, \\
  (M_H(\xi)-\omega^2)^{-1}
  &= \frac{1}{p(\omega,\xi)} \text{adj}(M_H(\xi)-\omega^2) 
  = \frac{1}{\mu_1\mu_2\mu_3 p(\omega,\xi)} Z_{\mu,\eps}(\xi) \mu.
\end{aligned}
\end{align}  
Here, $\text{adj}(M)$ denotes the adjugate matrix of $M$. Sarrus's rule and lengthy computations yield that the components of $Z=Z_{\varepsilon,\mu}$ are given as
follows:  
\begin{align} \label{eq:ZMatrix}
  \begin{aligned}
  Z_{11}(\xi)
  %&= X_{22}X_{33}-X_{23}X_{32} \\
  &= \xi_1^2( \frac{\xi_1^2}{\mu_2\mu_3}+\frac{\xi_2^2}{\mu_1\mu_3}+\frac{\xi_3^2}{\mu_1\mu_2}) 
  - \omega^2(
  \frac{\eps_2}{\mu_2}\xi_1^2+\frac{\eps_3}{\mu_3}\xi_1^2
  +\frac{\eps_2}{\mu_1}\xi_2^2+\frac{\eps_3}{\mu_1}\xi_3^2) +\omega^4\eps_2\eps_3, \\
  Z_{12}(\xi)
  &= Z_{21}(\xi) %= X_{13}X_{32}-X_{12}X_{33} \\
  =  \xi_1\xi_2(\frac{\xi_1^2}{\mu_2\mu_3}+\frac{\xi_2^2}{\mu_1\mu_3}+\frac{\xi_3^2}{\mu_1\mu_2}-\omega^2\frac{\eps_3}{\mu_3}),
  \\
  Z_{13}(\xi)
  &= Z_{31}(\xi)% = X_{12}X_{23}-X_{13}X_{22} \\
  =
  \xi_1\xi_3(\frac{\xi_1^2}{\mu_2\mu_3}+\frac{\xi_2^2}{\mu_1\mu_3}+\frac{\xi_3^2}{\mu_1\mu_2}-\omega^2\frac{\eps_2}{\mu_2}),
  \\
  Z_{22}(\xi)
  %&= X_{11}X_{33}-X_{13}X_{31} \\
  &= \xi_2^2( \frac{\xi_1^2}{\mu_2\mu_3}+\frac{\xi_2^2}{\mu_1\mu_3}+\frac{\xi_3^2}{\mu_1\mu_2}) 
  - \omega^2(
  \frac{\eps_1}{\mu_2}\xi_1^2+\frac{\eps_3}{\mu_3}\xi_2^2
  +\frac{\eps_1}{\mu_1}\xi_2^2+\frac{\eps_3}{\mu_2}\xi_3^2) +\omega^4\eps_1\eps_3, \\
  Z_{23}(\xi)
  &= Z_{32}(\xi)% =X_{13}X_{21}-X_{11}X_{23} \\  
   =\xi_2\xi_3(\frac{\xi_1^2}{\mu_2\mu_3}+\frac{\xi_2^2}{\mu_1\mu_3}+\frac{\xi_3^2}{\mu_1\mu_2}-\omega^2\frac{\eps_1}{\mu_1}),\\ 
  Z_{33}(\xi) 
 % &= X_{11}X_{22}-X_{12}X_{21} \\
  &= \xi_3^2( \frac{\xi_1^2}{\mu_2\mu_3}+\frac{\xi_2^2}{\mu_1\mu_3}+\frac{\xi_3^2}{\mu_1\mu_2}) 
  - \omega^2(
  \frac{\eps_1}{\mu_3}\xi_1^2+\frac{\eps_2}{\mu_3}\xi_2^2
  +\frac{\eps_1}{\mu_1}\xi_3^2+\frac{\eps_2}{\mu_2}\xi_3^2) +\omega^4\eps_1\eps_2.
\end{aligned}
\end{align}
A crucial observation is that the associated matrix-valued Fourier multiplier will be applied to
divergence-free functions. This is a
consequence of~\eqref{eq:MaxwellSystem} and~\eqref{eq:Inverses}. For that reason the fourth order terms in the
entries can be ignored (if convenient), which becomes important when estimating the large frequency parts of
our approximate solutions.
Let $Z^{\text{eff}}(\xi)=Z_{\varepsilon,\mu}^{\text{eff}}(\xi)$ denote the unique matrix-valued polynomial of
degree~$2$ such that
\begin{align*}
Z_{\varepsilon,\mu}(\xi) &= O(|\xi|^4) + Z_{\varepsilon,\mu}^{\text{eff}}(\xi), \\ \qquad \forall \xi \in \R^3: \,
Z_{\varepsilon,\mu}(\xi)v &= Z_{\varepsilon,\mu}^{\text{eff}}(\xi)v \text{ for all } v \in \R^3 \text{ with } v \cdot \xi = 0.
\end{align*}
In view of~\eqref{eq:MaxwellSystem} and~\eqref{eq:Inverses} it is natural to define the approximate solutions
$(E_\delta,H_\delta)$ for $|\delta| \neq 0$ as  follows:
\begin{equation}
\label{eq:RegularizedSolutions}
\left\{ \begin{array}{cl}
\hat{E}_\delta(\xi) &= \frac{i}{\eps_1\eps_2\eps_3(p(\omega,\xi)+i\delta)} Z_{\eps,\mu}(\xi)( - \omega  \hat{J}_e(\xi) +  b(\xi) \mu^{-1} \hat{J}_m(\xi)), \\
\hat{H}_\delta(\xi) &= \frac{i}{\mu_1\mu_2\mu_3(p(\omega,\xi)+i\delta)} Z_{\mu,\eps}(\xi)( 
b(\xi) \varepsilon^{-1} \hat{J}_e(\xi)+  \omega \hat{J}_m(\xi)).
\end{array} \right.
\end{equation}
To prove Theorem~\ref{thm:LpLqEstimatesTimeHarmonicMaxwell}, we show estimates for these
functions that  are uniform with respect to $\delta$. The global part away from the Fresnel surface is
considered in the next subsection, the remaining estimates will be done later. 
Then, taking these estimates for granted, we
show how to conclude the argument. 
%We emphasize once more that our regularization is not the expectable one coming
%from replacing $\omega$ by $\omega+i\delta$. In a nutshell, the reason is that it seems difficult to control
%the level sets of $\{\Real(\tilde{p}_{\mu,\varepsilon,\omega+i\delta}(\xi))=t\}$ even for small $|t|$.

\subsection{Local and global contributions}

We turn to the description of the different contributions of $(E_\delta,H_\delta)$. We split the local and
global contribution. Let $\beta_1,\beta_2 \in C^\infty(\R^3)$ satisfy $\beta_1(\xi) + \beta_2(\xi)=1$ with
\begin{equation*}
\beta_1(\xi) = 1 \text{ if } |p(\omega,\xi)| \leq t_0  
\quad\text{and}\quad \text{supp}(\beta_1) \subseteq \{ \xi \in \R^3 \,: \, |p(\omega,\xi)| \leq 2 t_0 \}
\end{equation*}
where $t_0>0$ denotes a small constant. $t_0$ will be chosen later when carrying out the estimates close to the surface. Also, for $m \in C^\infty(\R^d)$ we write
\begin{equation*}
(m(D) f) \widehat (\xi) = m(\xi) \hat{f}(\xi).
\end{equation*}
%Below $W^{m,p}(\R^d)$ denote the $L^p$-based Sobolev spaces:

% We decompose
% \begin{align*}
% \hat{J}_{x}(\xi) = \beta_1(\xi) \hat{J}_x(\xi) + \beta_2(\xi) \hat{J}_x(\xi) =: \hat{J}_{x1}(\xi) + \hat{J}_{x2}(\xi), \quad x \in \{ e,m \},
% \end{align*}
% $J_{x1}$ gives rise to the local contribution, whose estimate hinges on the regularity of the Fourier multiplier. We start with the estimate for $J_{x2}$, the global contribution.

\begin{proposition}
\label{prop:GlobalEstimate}
Let $E_\delta,H_\delta$ be given by~\eqref{eq:RegularizedSolutions}. Then, we find the following
estimate to hold uniformly in $|\delta| >0$:
\begin{equation}
\label{eq:GlobalEstimate}
\| \beta_2(D) (E_\delta,H_\delta) \|_{L^q(\R^3)} \lesssim \| \beta_2(D)(J_e,J_m) \|_{L^p(\R^3)},
\end{equation}
provided that $1 \leq p,q \leq \infty$ with $0 \leq \frac{1}{p} - \frac{1}{q} \leq \frac{1}{3}$ and $(p,q) \notin \{(1,1), (3,\infty), (\infty, \infty) \}$.
 If additionally $J_e,J_m\in L^q(\R^3)$, $q<\infty$, then $E,H\in W^{1,q}(\R^3)$ with
$$
 \| \beta_2(D)(E,H)\|_{W^{1,q}(\R^3)} \lesssim_{p,q,\omega}  \|  \beta_2(D)(J_e,J_m) \|_{L^{p_1}(\R^3) \cap
 L^q(\R^3)}. 
$$
\end{proposition}
\begin{proof}
Choose $\chi\in C_c^\infty(\R^3)$ with $|p(\omega,\xi) | \geq c>0$ on $\text{supp}(\chi)$ and $\chi(\xi)=1$ on
$\text{supp}(\beta_2)$. We first consider the case $q\neq \infty$. Then
\begin{align*}
  \beta_2(\xi)\hat E_\delta(\xi) 
  &=  \frac{i\beta_2(\xi)}{\eps_1\eps_2\eps_3(p(\omega,\xi)+i\delta)} 
  Z_{\eps,\mu}(\xi)( - \omega  \hat{J}_e(\xi) +  b(\xi) \mu^{-1} \hat{J}_m(\xi)), \\
   &=  -\frac{i\omega \chi(\xi)\langle \xi \rangle^2
   Z^{\text{eff}}_{\eps,\mu}(\xi)}{\mu_1\mu_2\mu_3(p(\omega,\xi)+i\delta)}
   \langle \xi \rangle^{-2}\beta_2(\xi) \hat{J}_e(\xi) \\
   &\; + \frac{i \chi(\xi)\langle \xi \rangle
   Z^{\text{eff}}_{\eps,\mu}(\xi)b(\xi)\mu^{-1}}{\tilde{p}_{\varepsilon,\mu,\omega}(\xi) + i \delta}
     \langle \xi \rangle^{-1}\beta_2(\xi) \hat{J}_m(\xi).
\end{align*}
By the choice of $\chi$ we have the following uniform estimates with respect to~$\delta$:  
\begin{align*}
\left| \partial^\alpha 
\left(\frac{\omega\chi(\xi)\langle \xi \rangle^2
   Z^{\text{eff}}_{\eps,\mu}(\xi)}{\eps_1\eps_2\eps_3(p(\omega,\xi)+i\delta)}\right) \right| + \left| \partial^\alpha 
\left(\frac{\chi(\xi)\langle \xi \rangle
   Z^{\text{eff}}_{\eps,\mu}(\xi)b(\xi)\mu^{-1}}{\mu_1\mu_2\mu_3(p(\omega,\xi)+i\delta)}\right)
   \right| 
\lesssim_\alpha |\xi|^{-\alpha}
\text{ for } \alpha \in \mathbb{N}_0^3.
\end{align*}
Since $1 < q < \infty$, Mikhlin's theorem (cf. \cite[Chapter~6]{Grafakos2014}) applies and Bessel
potential estimates (see for instance~\cite[Theorem~30]{CossettiMandel2020}) yield
\begin{align*}
  \|\beta_2(D) (E_\delta,H_\delta)\|_{L^q(\R^3)}
  &\les \|\langle D \rangle^{-2}\beta_2(D) J_e\|_{L^q(\R^3)} 
  + \|\langle D \rangle^{-1}\beta_2(D) J_m\|_{L^q(\R^3)} \\
  & \quad + \|\langle D \rangle^{-2}\beta_2(D) J_m\|_{L^q(\R^3)} 
  + \|\langle D \rangle^{-1}\beta_2(D) J_e\|_{L^q(\R^3)} \\
  &\les \|\beta_2(D)(J_e,J_m)\|_{L^p(\R^3)}
\end{align*}
for the claimed range of exponents. If $q=\infty$, we first use Sobolev embedding to lower $q< \infty$,
and applying the previous argument gives \eqref{eq:GlobalEstimate} for $0 < \frac{1}{p}  <
\frac{1}{3}$, which is all we had to show in this case. This gives the claim concerning
$L^q(\R^3)$-integrability.
For the Sobolev regularity we obtain in a similar fashion

\begin{align*}
  \|\beta_2(D) (E_\delta,H_\delta)\|_{W^{1,q}(\R^3)}
  &\les \|\beta_2(D) (E_\delta,H_\delta)\|_{L^q(\R^3)}
   + \|\langle D\rangle \beta_2(D) (E_\delta,H_\delta)\|_{L^q(\R^3)} \\
  &\les \|\langle D \rangle^{-1}\beta_2(D) J_e\|_{L^q(\R^3)} 
  + \| \beta_2(D) J_m\|_{L^q(\R^3)} \\
  & \quad + \|\langle D \rangle^{-1}\beta_2(D) J_m\|_{L^q(\R^3)} 
  + \|\beta_2(D) J_e\|_{L^q(\R^3)} \\
  &\les \|\beta_2(D)(J_e,J_m)\|_{L^q(\R^3)}
\end{align*} 
The proof is complete.
\end{proof}

The paper is mainly devoted to estimate the local contribution close to the Fresnel surface $S= \{p(\omega,\xi) = 0\}$. In  
Section~\ref{section:FresnelSurface}  we shall see that the Fresnel surface has components of the following type:
\begin{itemize}
\item smooth components with non-vanishing Gaussian curvature,
\item smooth components with curvature vanishing along a one-dimensional submanifold (Hamiltonian
circles), but without flat points,
\item neighbourhoods of conical singularities.
\end{itemize}
This fact is established in Corollary~\ref{cor:Fresnel}. Precisely, it suffices to consider six components of the first kind, and four components of the second and third type.

Corresponding to the three types listed above, we split
\begin{equation*}
\beta_1(\xi) = \beta_{11}(\xi) + \beta_{12}(\xi) + \beta_{13}(\xi)
\end{equation*}
with smooth compactly supported functions localizing to neighbourhoods of the components of the above types.  
% and $\beta^\varepsilon_{13}$ additionally taking into account the $\varepsilon$-regularization, cutting off smoothly $\varepsilon$-neighbourhoods of the singular points.
 The estimate for the smooth components with non-vanishing Gaussian curvature is a consequence of estimates for
 Bochner-Riesz operators with negative index that we will prove in
 Section~\ref{section:BochnerRieszEstimates}:
\begin{proposition}
\label{prop:PositiveCurvatureEstimate}
Let $1 \leq p,q \leq \infty$ and $(E_\delta,H_\delta)$ as in~\eqref{eq:RegularizedSolutions}. We find the
following estimate to hold uniformly in $|\delta| \neq 0$:
\begin{equation*}
%\label{eq:NondegenerateCurvatureSolutions}
\| \beta_{11}(D) (E_\delta,H_\delta) \|_{L^q(\R^3)} \lesssim \| \beta_{11}(D) (J_e,J_m) \|_{L^{p}(\R^3)}
\end{equation*} 
provided that $\frac{1}{p} > \frac{2}{3}$, $\frac{1}{q} < \frac{1}{3}$, and $\frac{1}{p} - \frac{1}{q} \geq \frac{1}{2}$.
\end{proposition}
By similar means, we show the inferior estimate for components with vanishing Gaussian curvature along the
Hamiltonian circles:
\begin{proposition}
\label{prop:HamiltonEstimates}
Let $1 \leq p, q \leq \infty$ and $(E_\delta,H_\delta)$ as in~\eqref{eq:RegularizedSolutions}. We find the
following estimate to hold uniformly in $|\delta| \neq 0$:
\begin{equation*}
%\label{eq:EstimatesHamiltonCircles}
\| \beta_{12}(D) (E_\delta,H_\delta) \|_{L^q(\R^3)} \lesssim \| \beta_{12}(D) (J_e,J_m) \|_{L^p(\R^3)}
\end{equation*}
provided that $\frac{1}{p} > \frac{3}{4}$, $\frac{1}{q} < \frac{1}{4}$, and $\frac{1}{p} - \frac{1}{q} \geq \frac{2}{3}$.
\end{proposition}
At last, the estimate around the singular points is shown in Section \ref{section:SingularPoints}:
\begin{proposition}
\label{prop:SingularEstimate}
Let $1 \leq p, q \leq \infty$ and $(E_\delta,H_\delta)$ as in~\eqref{eq:RegularizedSolutions}. We find the
following estimates to hold uniformly in $|\delta| \neq 0$:
\begin{equation*}
%\label{eq:EstimateSingularPoints}
\| \beta_{13}(D) (E_\delta,H_\delta) \|_{L^q(\R^3)} \lesssim \| \beta_{13}(D) (J_e,J_m) \|_{L^p(\R^3)}
\end{equation*}
provided that $\frac{1}{p} > \frac{3}{4}$, $\frac{1}{q} < \frac{1}{4}$, and $\frac{1}{p} - \frac{1}{q} \geq \frac{2}{3}$.
\end{proposition}

\begin{remark}
\label{rem:MultiplierEstimate}
For these estimates, due to bounded frequencies, the precise form of $Z_{\varepsilon,\mu}$ (or 
$Z^{\text{eff}}_{\varepsilon,\mu}$) is not important. It suffices to show the above estimates for the
multiplier
\begin{equation*}
%\label{eq:Adelta}
A_\delta f(x) = \int_{\R^3} e^{ix.\xi} \frac{\beta_{1i}(\xi)}{p(\omega,\xi) + i\delta} \hat{f}(\xi) d\xi.
\end{equation*}
Again due to bounded frequencies, the $W^{1,q}(\R^3)$-estimates result from
$$
\| \beta_{1i}(D) (E_\delta,H_\delta) \|_{W^{m,q}(\R^3)} \lesssim_{m,q} \| \beta_{1i}(D) (E_\delta,H_\delta)
\|_{L^q(\R^3)} \qquad (i=1,2,3)
$$ 
as a consquence of Young's inequality.
\end{remark}

\subsection{Proof of Theorem \ref{thm:LpLqEstimatesTimeHarmonicMaxwell}} 

By Propositions \ref{prop:GlobalEstimate} - \ref{prop:SingularEstimate} we have uniform bounds in $\delta \neq 0$:
\begin{equation*}
\| (E_\delta,H_\delta) \|_{L^q(\R^3)} \lesssim \| (J_e,J_m) \|_{L^{p_1}(\R^3) \cap L^{p_2}(\R^3)}
\end{equation*}
for $q$, $p_1$, $p_2$ as in the assumptions of Theorem~\ref{thm:LpLqEstimatesTimeHarmonicMaxwell}.
Hence, there is a weak limit $(E,H) \in L^q(\R^3;\C^6)$, which satisfies the same bound by the
Banach--Alaoglu--Bourbaki theorem. 
We have to show that the approximate solutions weakly converge to distributional solutions of
\begin{equation}
\label{eq:OriginalFirstOrderSystem}
\left\{ \begin{array}{cl}
ib(\xi) \hat{E}(\xi) + i \omega \mu \hat{H}(\xi) &= \hat{J}_m(\xi), \\
i b(\xi) \hat{H}(\xi) - i \omega \varepsilon \hat{E}(\xi) &= \hat{J}_e(\xi).
\end{array} \right.
\end{equation}
Indeed,~\eqref{eq:RegularizedSolutions} gives
\begin{equation*}
%\label{eq:JmLimit}
\begin{split}
&\quad ib(\xi) \hat{E}_{\delta}(\xi) + i \omega\mu \hat{H}_{\delta}(\xi) \\
&=
\frac{b(\xi)Z_{\varepsilon,\mu}(\xi)}{\eps_1\eps_2\eps_3(p(\omega,\xi)+i\delta)} 
\big( \omega   \hat{J}_e(\xi) - b(\xi) \mu^{-1} \hat{J}_m(\xi) \big) \\
&\quad - \frac{\omega\mu Z_{\mu,\varepsilon}(\xi)}{\mu_1\mu_2\mu_3(p(\omega,\xi)+i\delta)} \big( b(\xi)
\varepsilon^{-1} \hat{J}_e(\xi) + \omega \hat{J}_m(\xi) \big) \\
&= \frac{\omega}{p(\omega,\xi) + i \delta} \left( \frac{b(\xi) Z_{\varepsilon,\mu}(\xi)}{
\varepsilon_1 \varepsilon_2 \varepsilon_3} - \frac{\mu Z_{\mu,\varepsilon}(\xi) b(\xi)
\varepsilon^{-1}}{\mu_1 \mu_2 \mu_3} \right) \hat{J}_e(\xi) \\
&\quad - \frac{1}{p(\omega,\xi) + i\delta} \left( \frac{ b(\xi) Z_{\varepsilon,\mu}(\xi) b(\xi)
\mu^{-1}}{\varepsilon_1 \varepsilon_2 \varepsilon_3} + \frac{\omega^2 \mu
Z_{\mu,\varepsilon}(\xi)}{\mu_1 \mu_2 \mu_3} \right) \hat{J}_m(\xi).
\end{split} 
\end{equation*}
 From~\eqref{eq:ZMatrix} one infers after lengthy computations 
\begin{equation*}
\begin{split}
\frac{b(\xi) Z_{\varepsilon,\mu}(\xi)}{\varepsilon_1 \varepsilon_2 \varepsilon_3} -
 \frac{\mu Z_{\mu,\varepsilon}(\xi) b(\xi) \varepsilon^{-1}}{\mu_1 \mu_2 \mu_3} &= 0, \\
\frac{b(\xi) Z_{\varepsilon,\mu}(\xi) b(\xi) \mu^{-1}}{\varepsilon_1
\varepsilon_2 \varepsilon_3} + \frac{\omega^2 \mu Z_{\mu,\varepsilon}(\xi)}{\mu_1 \mu_2 \mu_3} 
&=  -p(\omega,\xi) I_3.
\end{split}
\end{equation*}
As a consequence we obtain
\begin{equation*}
ib(\xi) \hat{E}_{\delta}(\xi) + i \omega\mu \hat{H}_{\delta}(\xi) 
= \frac{p(\omega,\xi)}{p(\omega,\xi) + i \delta} \hat{J}_m(\xi) 
= \hat{J}_m(\xi) - \frac{i \delta}{p(\omega,\xi) + i \delta} \hat{J}_m(\xi).
\end{equation*}
By Proposition \ref{prop:GlobalEstimate} - \ref{prop:SingularEstimate}, and
Remark~\ref{rem:MultiplierEstimate} we have
\begin{equation*}
\| (p(\omega,D) + i \delta)^{-1} J_m \|_{L^q(\R^3)} \lesssim \| J_m \|_{L^{p_1}(\R^3) \cap L^{p_2}(\R^3)}
\end{equation*}
and, when assuming $J_m\in L^q(\R^3)$,
\begin{equation*}
\| (p(\omega,D) + i \delta)^{-1} J_m \|_{W^{1,q}(\R^3)} \lesssim \| J_m \|_{L^{p_1}(\R^3) \cap L^q(\R^3)}
\end{equation*}
so that the only $\delta$-dependent term vanishes as $\delta\to 0$. This implies
\begin{equation*}
\nabla \times E + i \omega \mu H = J_m \quad\text{in }\R^3
\end{equation*}
in the distributional sense and even in the weak sense for $J_m\in L^q(\R^3)$. Similarly, one proves the
validity of the second equation in~\eqref{eq:OriginalFirstOrderSystem}, and the proof is complete.\\ \hfill $\Box$

\subsection{Explicit representations of solutions}
At last, we give explicit representations of the constructed solutions. By Sokhotsky's formula (cf. Sections~3.2 and 6.1 in \cite{Hoermander2003}):
\begin{proposition}
\label{prop:PrincipalValue}
Let $H: \R^d \to \R$ such that $| \nabla H(\xi) | \neq 0$ at any point where $H(\xi) = 0$, then we can define the distributional limit
\begin{equation*}
(H(\xi) \pm i0)^{-1} = \lim_{\varepsilon \to 0} (H(\xi) \pm i \varepsilon)^{-1}.
\end{equation*}
Furthermore,
\begin{equation*}
(H(\xi) \pm i0)^{-1} = v.p. \frac{1}{H(\xi)} \mp i \pi \delta(H)
\end{equation*}
in the sense of distributions.
\end{proposition}

In the context of the easier Helmholtz equation
\begin{equation*}
%\label{eq:HelmholtzEquation}
(\Delta +1) u = -f,
\end{equation*}
this allows to write for so-called outgoing solutions
\begin{equation*}
\begin{split}
u(x) &= \frac{1}{(2 \pi)^d} \int_{\R^d} \frac{e^{ix.\xi}}{|\xi|^2-1 - i 0} \hat{f}(\xi) d\xi \\
 &= \frac{1}{(2 \pi)^d} v.p. \int_{\R^d} \frac{e^{ix.\xi}}{|\xi|^2-1} \hat{f}(\xi) d\xi +
  \frac{i \pi}{(2 \pi)^d} \int_{\mathbb{S}^{d-1}} \hat{f}(\xi) e^{i x. \xi} d\sigma(\xi).
\end{split}
\end{equation*}

Proposition \ref{prop:PrincipalValue} suggests that the solutions to anisotropic Maxwell's equations can
again be written as principal value and delta distribution in Fourier space. However, Proposition
\ref{prop:PrincipalValue} only allows to make sense of the principal value and delta distribution if $S =
\{\xi \in \R^3 \, : \, p(\omega,\xi) = 0 \}$ is a smooth manifold. But there are four isolated singular
points $\zeta_1,\ldots,\zeta_4\in S$ as we will prove in Proposition~\ref{prop:singularpoints}. Still, we
shall see how $v.p. \frac{1}{p(\omega,\xi)}$ and $\delta_S(\xi)$ can be understood as Fourier multipliers with certain $L^p$-mapping properties. For a dense set, e.g., $J \in \mathcal{S}(\R^3)$, $\zeta_i \notin \text{supp} (\hat{J})$, we can explain $\delta_S$ as a Fourier multiplier
\begin{equation*}
\int_{\R^3} \delta_S(\xi) e^{ix.\xi} \hat{J}(\xi) d\xi = \int_S e^{ix.\xi} \hat{J}(\xi) d\sigma(\xi).
\end{equation*}
The density follows by Littlewood-Paley theory. As a consequence of Sections
\ref{section:CurvatureBoundedBelow} and \ref{section:SingularPoints}, we have
\begin{equation*}
\left\| \int_S e^{ix.\xi} \hat{J}(\xi) d\sigma(\xi) \right\|_{L^q(\R^3)} \lesssim \| J \|_{L^p(\R^3)}
\end{equation*}
for $p$ and $q$ as in Proposition \ref{prop:SingularEstimate} with a bound independent of the support of $\hat{J}$. This allows to extend $\mathcal{F}^{-1} \delta_S \mathcal{F}: L^p(\R^3) \to L^q(\R^3)$ by density. Likewise, we can explain
\begin{equation*}
v.p. \int_{\R^3} \frac{e^{ix.\xi} \beta(\xi)}{p(\omega,\xi)} \hat{J}(\xi) d\xi
\end{equation*}
with $\beta \in C^\infty_c(\R^3)$ for $J \in \mathcal{S}(\R^3)$ and $\zeta_1,\ldots,\zeta_4 \notin \text{supp
} (\hat{J})$.
This explains the formula  
\begin{equation*}
%\label{eq:ExplicitSolutions}
\begin{split}
E &= \frac{-i \pi}{(2 \pi)^3\eps_1\eps_2\eps_3} \int_{\R^3} \delta_{S}(\xi) e^{ix.\xi} 
  Z_{\eps,\mu}(\xi)\big(-i \omega \hat{J}_e(\xi) + i b(\xi) \mu^{-1} \hat{J}_m(\xi) \big) d\xi \\
&\quad + \frac{1}{(2 \pi)^3\eps_1\eps_2\eps_3} v.p. \int_{\R^3}
\frac{e^{ix.\xi}}{p(\omega,\xi)} Z_{\eps,\mu}(\xi)  
\big( - i \omega \hat{J}_e(\xi) + ib(\xi) \mu^{-1} \hat{J}_m(\xi) \big) d\xi, \\
H&= \frac{-i \pi}{(2 \pi)^3\mu_1\mu_2\mu_3} \int_{\R^3} \delta_{S}(\xi) e^{ix.\xi}Z_{\mu,\eps}(\xi)
\big( i  b(\xi) \varepsilon^{-1} \hat{J}_e(\xi) + i \omega \hat{J}_m(\xi) \big) d\xi \\
&\quad + \frac{1}{(2 \pi)^3\mu_1\mu_2\mu_3 } v.p. \int_{\R^3}
\frac{e^{ix.\xi}}{p(\omega,\xi)} Z_{\mu,\eps}(\xi) \big( i
 b(\xi) \varepsilon^{-1} \hat{J}_e(\xi) + i \omega  \hat{J}_m(\xi) \big) d\xi.
\end{split}
\end{equation*}
for solutions to anisotropic Maxwell's equations. Notice that in these formulae we may replace 
the matrices $Z_{\eps,\mu}(\xi),Z_{\mu,\eps}(\xi)$ by the corresponding effective matrices. 
 
\section{Properties of the Fresnel surface}
\label{section:FresnelSurface}

As explained above, the set $\{\xi\in\R^3: p(\omega,\xi)=0\}$ plays a decisive role for our
analysis. This classical quartic surface is known as Fresnel's surface initially discovered by Augustin-Jean
Fresnel in 1822 to describe the phenomenon of double refraction. In an optically anisotropic medium, e.g., a
biaxial crystal, Fresnel's surface corresponds to Huygen's elementary spherical wave surfaces in isotropic media.
This surface was already studied in the 19th century by Darboux \cite{Darboux1993}. For an
account on classical references we refer to the survey by Kn\"orrer \cite{Knoerrer1986}. In the present context the curvature
properties will be most important, which were collected by Liess~\cite[Appendix]{Liess1991}. We think it is
worthwhile to elaborate on Liess's presentation, as we shall also discuss first and second fundamental form in suitable coordinates.

We recall the key properties of Fresnel's wave surface 
$$
  S=\{\xi\in\R^3: p(\omega,\xi)=0\},\quad  
  p(\omega,\xi)= - \omega^2 (\omega^4 - \omega^2 q_0(\xi) + q_1(\xi))
$$
and 
\begin{align*}
q_0(\xi) &= \xi_1^2 \big( \frac{1}{\varepsilon_2 \mu_3} + \frac{1}{\mu_2 \varepsilon_3} \big) + \xi_2^2 \big( \frac{1}{\varepsilon_1 \mu_3} + \frac{1}{\mu_1 \varepsilon_3} \big) + \xi_3^2 \big( \frac{1}{\varepsilon_1 \mu_2} + \frac{1}{\varepsilon_2 \mu_1} \big), \\
q_1(\xi) &= \frac{1}{\varepsilon_1 \varepsilon_2 \varepsilon_3 \mu_1 \mu_2 \mu_3} (\varepsilon_1 \xi_1^2 + \varepsilon_2 \xi_2^2 + \varepsilon_3 \xi_3^2) (\mu_1 \xi_1^2 + \mu_2 \xi_2^2 + \mu_3 \xi_3^2).
\end{align*}  
Recall that we assume full anisotropy~\eqref{eq:FullAnisotropy}.  
We first notice that we can reduce our analysis to the case  $\mu_1=\mu_2=\mu_3=\omega=1$. This results from
the change of coordinates $\xi\to \eta$ given by 
$$
  \eta_i  = \frac{\xi_i}{\omega\sqrt{\mu_{i+1}\mu_{i+2}}} \qquad (i=1,2,3) 
$$
Notice that this change of coordinates results from a suitable dilation of the coordinates, which corresponds
to an appropriate dilation in physical space. To see the equivalence, let us introduce the corresponding
quantities for $\mu_1=\mu_2=\mu_3=\omega=1$, namely $\mathcal N(\eta):= 1-q_0^*(\eta)+q_1^*(\eta)$ where 
\begin{align*}
  q_0^*(\eta)
  &=  \eta_1^2( \frac{1}{\varepsilon_2}+\frac{1}{\varepsilon_3})
  + \eta_2^2( \frac{1}{\varepsilon_1}+\frac{1}{\varepsilon_3})
  + \eta_3^2( \frac{1}{\varepsilon_1}+\frac{1}{\varepsilon_2}),  \\
  q_1^*(\eta)
  &= \frac{1}{\varepsilon_1\varepsilon_2\varepsilon_3}(\varepsilon_1\eta_1^2+\varepsilon_2\eta_2^2+\varepsilon_3\eta_3^2)(\eta_1^2+\eta_2^2+\eta_3^2).
\end{align*} 
Then one observes $\omega^4 \mathcal N(\eta)=p(\omega,\xi)$, hence the qualitative properties of Fresnel's
surface in the special case $\mu_1=\mu_2=\mu_3=\omega=1$ carry over to the general case. 
For this reason we focus on the analysis of 
$$
  S^* = \{ \eta\in\R^3: \mathcal N(\eta) = 1 -q_0^*(\eta)+q_1^*(\eta)=0\}. 
$$
Notice that~\eqref{eq:FullAnisotropy} then reads
$$
  \eps_1\neq \eps_2\neq \eps_3\neq \eps_1.
$$
In the following we write
\begin{equation*}
\eps_{i+1} \in \langle\eps_i,\eps_{i+2}\rangle \quad\text{if}\quad \eps_i<\eps_{i+1}<\eps_{i+2} \text{ or }
\eps_{i+2}<\eps_{i+1}<\eps_i.
\end{equation*}

\medskip

We first show that $S^*$ is a smooth manifold away from four singular points. To see this, we compute
$$
  \nabla \cN(\eta) = \vecIII{t_1(\eta)\eta_1}{t_2(\eta)\eta_2}{t_3(\eta)\eta_3},\quad
  t_i(\eta) = - \frac{2}{\varepsilon_{i+1}}-\frac{2}{\varepsilon_{i+2}} +
  \frac{2\varepsilon_i|\eta|^2+2(\varepsilon_1\eta_1^2+\varepsilon_2\eta_2^2+\varepsilon_3\eta_3^2)}{\varepsilon_1\varepsilon_2\varepsilon_3}.
$$

\begin{definition}
  A point $\eta\in S^*$ is called singular if $\nabla \cN(\eta)=0$.
   The set of singular points is denoted by $\Sigma$.
\end{definition}

The reason for this definition is that $S^* \sm\Sigma$ is a smooth manifold, whereas the neighbourhood of the  
singular points require a separate analysis. It turns out that there are precisely four singular
points. This is a consequence of the following result.

\begin{proposition} \label{prop:singularpoints}
  The set of singular points consists of all $\eta\in S^*$ such that 
  $$
    \eta_i^2 = \frac{\varepsilon_{i+2} (\varepsilon_i-\varepsilon_{i+1})}{\varepsilon_i - \varepsilon_{i+2}},
    \qquad \eta_{i+1}=0,
    \qquad 
    \eta_{i+2}^2 = \frac{\varepsilon_i(\varepsilon_{i+2} -\varepsilon_{i+1})}{\varepsilon_{i+2} -\varepsilon_i},
  $$
  where $i\in\{1,2,3\}$ is uniquely determined by $\varepsilon_{i+1}\in
  \langle\varepsilon_i,\varepsilon_{i+2}\rangle$.
\end{proposition}
\begin{proof}
  We have to prove that each solution of $\nabla \cN(\eta)=(t_1(\eta)\eta_1,t_2(\eta)\eta_2,t_3(\eta)\eta_3) =
  (0,0,0)$ satisfies the above conditions. We first show $\eta_j=0$ for some $j\in\{1,2,3\}$. Otherwise, we
  would have $t_1(\eta)=t_2(\eta)=t_3(\eta)=0$, and thus for $j\in\{1,2,3\}$,
  $$
    2\varepsilon_j\eta_j^2+(\varepsilon_j+\varepsilon_{j+1})\eta_{j+1}^2+(\varepsilon_j+\varepsilon_{j+2})\eta_{j+2}^2
    =  \varepsilon_j(\varepsilon_{j+1}+\varepsilon_{j+2}). 
  $$
  Hence,
  $$
    \underbrace{\matIII{2\varepsilon_1}{\varepsilon_1+\varepsilon_2}{\varepsilon_1+\varepsilon_3}{
    \varepsilon_1+\varepsilon_2}{2\varepsilon_2}{\varepsilon_2+\varepsilon_3}{
    \varepsilon_1+\varepsilon_3}{\varepsilon_2+\varepsilon_3}{2\varepsilon_3}}_{=:M} \vecIII{\eta_1^2}{\eta_2^2}{\eta_3^2} 
    =  \vecIII{\varepsilon_1(\varepsilon_2+\varepsilon_3)}{\varepsilon_2 
    (\varepsilon_1+\varepsilon_3)}{\varepsilon_3(\varepsilon_1+\varepsilon_2)}. 
  $$
  The adjugate matrix of $M$ is given by 
  $$
    adj(M)
    = \matIII{-(\varepsilon_2-\varepsilon_3)^2}{(\varepsilon_3-\varepsilon_1)(\varepsilon_3-\varepsilon_2)}{
    (\varepsilon_2-\varepsilon_1)(\varepsilon_2-\varepsilon_3)}{
    (\varepsilon_3-\varepsilon_1)(\varepsilon_3-\varepsilon_2)}{-(\varepsilon_1-\varepsilon_3)^2}{(\varepsilon_1-\varepsilon_2)
    (\varepsilon_1-\varepsilon_3)}{ (\varepsilon_2-\varepsilon_1)(\varepsilon_2-\varepsilon_3)}{
    (\varepsilon_1-\varepsilon_2)(\varepsilon_1-\varepsilon_3)}{-(\varepsilon_1-\varepsilon_2)^2}. 
  $$
  Multiplying this equation with $adj(M)$ and using $adj(M)M=\det(M)I_3=0$ we get
  \begin{align*}
    0
    &= adj(M)M \vecIII{\eta_1^2}{\eta_2^2}{\eta_3^2}  
    =  adj(M) 
    \vecIII{\varepsilon_1(\varepsilon_2+\varepsilon_3)}{\varepsilon_2(\varepsilon_1+\varepsilon_3)}{ \varepsilon_3(\varepsilon_1+\varepsilon_2)} \\
    &=  \vecIII{(\varepsilon_2-\varepsilon_3)^2(\varepsilon_1-\varepsilon_2)(\varepsilon_1-\varepsilon_3)}{
    (\varepsilon_1-\varepsilon_3)^2(\varepsilon_2-\varepsilon_1)(\varepsilon_2-\varepsilon_3)}{
    (\varepsilon_1-\varepsilon_2)^2(\varepsilon_3-\varepsilon_1)(\varepsilon_3-\varepsilon_2)}.
  \end{align*}
  Since this is impossible due to the full anisotropy, we conclude $\eta_j=0$ for some $j\in\{1,2,3\}$.
  
  \medskip
  
  Next we show that only one coordinate of $\eta$ vanishes. First, $\eta_1=\eta_2=\eta_3=0$ is impossible in
  view of $\eta\in S^*=\{\cN(\eta)=0\}$ and $\cN(0,0,0)=1\neq 0$. So we argue by contradiction and suppose that
  $\eta_{j+1}=\eta_{j+2}=0$ and $\eta_j\neq 0,t_j(\eta)=0$ for some $j\in\{1,2,3\}$. In view of the formula
  for $t_j$ this implies $2\eta_j^2=\varepsilon_{j+1}+\varepsilon_{j+2}$. Inserting this into $\cN(\eta)=0$, we obtain $\varepsilon_{j+1}=\varepsilon_{j+2}$ as a necessary
  condition, which contradicts our assumption of full anisotropy. Hence, precisely one
  coordinate vanishes, say $\eta_{j+1}=0,t_j(\eta)=t_{j+2}(\eta)=0, \eta_{j},\eta_{j+2}\neq 0$ for $j\in\{1,2,3\}$. Elementary Linear
  Algebra shows that these conditions are equivalent to 
  $$
    \eta_j^2 = \frac{\varepsilon_{j+2} (\varepsilon_j-\varepsilon_{j+1})}{\varepsilon_j - \varepsilon_{j+2}},
    \qquad \eta_{j+1}=0,
    \qquad  \eta_{j+2}^2 = \frac{\varepsilon_j(\varepsilon_{j+2} -\varepsilon_{j+1})}{\varepsilon_{j+2} -\varepsilon_j}.
  $$ 
  Since the expressions on the right hand-side are positive if and only if $\varepsilon_{j+1}\in
  \langle\varepsilon_j,\varepsilon_{j+2}\rangle$, we get the claim.
\end{proof}

  In particular, the Gaussian curvature is well-defined and smooth on $S^*\sm \Sigma$, i.e.,  away from
  the four singular points. We now introduce the explicit parametrization of $S^*$ by Darboux and Liess
  (\cite[A3]{Liess1991}). Our parameters $(s,t)$ correspond to $(\beta,\alpha')$ in Liess' work. As
  in~\cite{Liess1991}, this parametrization is given away from the four singular points and the principal sections $S\cap \{\eta_i=0\}$.
  
  \begin{proposition}
    Let $\sigma_1,\sigma_2,\sigma_3\in\{-1,+1\}$. Then
    a smooth parametrization of $(S^* \sm \Sigma)\cap \bigcap_{i=1}^3 \{\sigma_i \eta_i>0\}$ is given by 
    $$
      \Phi_i(s,t) 
      :=  \sigma_i \sqrt{\frac{\varepsilon_1\varepsilon_2\varepsilon_3(\varepsilon_i-s)(t^{-1}-\varepsilon_i^{-1})}{(\varepsilon_i-\varepsilon_{i+1})(\varepsilon_i-\varepsilon_{i+2})}}
      \qquad (i=1,2,3).
    $$
    For $j\in\{1,2,3\}$ such that $\eps_j<\eps_{j+1}<\eps_{j+2}$ we either have
    $\eps_j<s<\eps_{j+1}<t<\eps_{j+2}$ or $\eps_j<t<\eps_{j+1}<s<\eps_{j+2}$.
  \end{proposition}
  \begin{proof}
    If we define $\eta:=\Phi(s,t)\in\R^3$, then one can subsequently verify
    \begin{align*}
      \eta_1^2+\eta_2^2+\eta_3^2
      &= s, \quad
      \varepsilon_1\eta_1^2+\varepsilon_2\eta_2^2+\varepsilon_3\eta_3^2
      =  \varepsilon_1\varepsilon_2\varepsilon_3 t^{-1}, \\
       \quad
      q_1^*(\eta) 
      &=  st^{-1},\quad
      q_0^*(\eta)
      = 1+st^{-1}.
    \end{align*}
    This implies 
    $\cN(\eta)= 1-  (1+st^{-1})+ st^{-1}=0$, 
    which proves $\Phi(s,t)\in S^*$ for all $s,t$ such that the argument of the square root is positive. On the other hand,
    every point of $(S^*\sm \Sigma)\cap \bigcap_{i=1}^3 \{\sigma_i \eta_i>0\}$ can be written in this way.
    To see this, one solves the linear system
    \begin{align*}
      s&=\eta_1^2+\eta_2^2+\eta_3^2,\qquad 
      \varepsilon_1\eta_1^2+\varepsilon_2\eta_2^2+\varepsilon_3\eta_3^2
      = \varepsilon_1\varepsilon_2\varepsilon_3 t^{-1}, \\
      0 
      &= \mathcal N(\eta) 
      = 1- \eta_1^2( \frac{1}{\varepsilon_2}+\frac{1}{\varepsilon_3})
        - \eta_2^2( \frac{1}{\varepsilon_1}+\frac{1}{\varepsilon_3})
       - \eta_3^2( \frac{1}{\varepsilon_1}+\frac{1}{\varepsilon_2})  
       + \frac{s}{t} 
    \end{align*}
    for $\eta_1^2,\eta_2^2,\eta_3^2$. In this way one finds $\eta_i^2=\Phi_i(s,t)^2$, so
    $\Phi$ is a smooth parametrization of the set $(S \sm \Sigma)\cap \bigcap_{i=1}^3 \{\sigma_i \eta_i>0\}$.
    A computation shows that $\Phi=(\Phi_1,\Phi_2,\Phi_3)$ is well-defined (the arguments of all square roots
    are positive) if and only if either $\eps_j<s<\eps_{j+1}<t<\eps_{j+2}$ or 
    $\eps_j<t<\eps_{j+1}<s<\eps_{j+2}$ holds provided that $\eps_j<\eps_{j+1}<\eps_{j+2}$.
  \end{proof}
  
   We note that the two parameter regions $\eps_j<s<\eps_{j+1}<t<\eps_{j+2}$ and
   $\eps_j<t<\eps_{j+1}<s<\eps_{j+2}$ give rise to the inner, respectively, outer sheet of the wave surface,
   cf. Figure~\ref{fig:Fresnel}. Both sheets meet at the singular points that formally correspond to
   $s=t=\eps_{j+1}$ where, in accordance with Proposition~\ref{prop:singularpoints}, one has
   $\eta_{j+1}=\Phi_{j+1}(s,t)=0$.
 We now turn towards the computation of the Gaussian curvature on $S^* \sm \Sigma$. This will first be done
 away from the principal sections, but the formula will prevail also in the principal sections since $S$ is a
 smooth manifold in that region as we showed above. We start with computing the relevant derivatives for the
 first and second fundamental form of $S^*$:
 \begin{align*}
   \partial_s\Phi_i(s,t)
   &= \frac{1}{2(s-\varepsilon_i)}\Phi_i,\quad
   \partial_t\Phi_i(s,t)
   = \frac{\varepsilon_i}{2t(t-\varepsilon_i)}\Phi_i,\quad
   \partial_{ss}\Phi_i(s,t)
   = -\frac{1}{4(s-\varepsilon_i)^2}\Phi_i,\\
   &\partial_{st}\Phi_i(s,t)
   = \frac{\varepsilon_i}{4t(t-\varepsilon_i)(s-\varepsilon_i)}\Phi_i, \qquad
   \partial_{tt}\Phi_i(s,t)
   = \frac{\varepsilon_i(3\varepsilon_i-4t)}{4t^2(t-\varepsilon_i)^2}\Phi_i.
 \end{align*}
 From these formulae one gets the following.
 
 \begin{proposition}
 \label{prop:FirstFundamentalForm}
   The first fundamental form of $S^* \sm \Sigma$ is given by 
   $$
     E(s,t)ds^2+2F(s,t)\,ds\,dt+G(s,t)\,dt^2,
   $$ 
   where
   \begin{align*}
     E(s,t)
     %&= \skp{\partial_s\Phi(s,t)}{\partial_s\Phi(s,t)}\\
     &= \frac{s^2t-(\varepsilon_1+\varepsilon_2+\varepsilon_3)st + (\varepsilon_1\varepsilon_2+\varepsilon_1\varepsilon_3+\varepsilon_2\varepsilon_3)t-\varepsilon_1\varepsilon_2\varepsilon_3}{4t(s-\varepsilon_1)(s-\varepsilon_2)(s-\varepsilon_3)}, 
     \\
     F(s,t)
     %&= \skp{\partial_s\Phi(s,t)}{\partial_t\Phi(s,t)} 
     &= 0, \\
     G(s,t)
    % &= \skp{\partial_t\Phi(s,t)}{\partial_t\Phi(s,t)} 
      &= \frac{\varepsilon_1\varepsilon_2\varepsilon_3(s-t)}{4t^2(t-\varepsilon_1)(t-\varepsilon_2)(t-\varepsilon_3)}.   
   \end{align*}
 \end{proposition}
 \begin{proof}
   This follows from lengthy, but straightforward computations based on 
     \begin{align*}
     E(s,t)
     &= \skp{\partial_s\Phi(s,t)}{\partial_s\Phi(s,t)}
     = \sum_{i=1}^3  \frac{\Phi_i(s,t)^2}{4(s-\varepsilon_i)^2} ,\\
     F(s,t)
     &= \skp{\partial_s\Phi(s,t)}{\partial_t\Phi(s,t)}
     = \sum_{i=1}^3  \frac{\varepsilon_i\Phi_i(s,t)^2}{4t(t-\varepsilon_i)(s-\varepsilon_i)},  \\
     G(s,t)
     &=  \skp{\partial_t\Phi(s,t)}{\partial_t\Phi(s,t)}
      = \sum_{i=1}^3 \frac{\varepsilon_i^2\Phi_i(s,t)^2}{4t^2(t-\varepsilon_i)^2}. 
   \end{align*}
 \end{proof}
 
 To write down the second fundamental form, we introduce the following functions:
 \begin{align*}
     m(s,t)&:=\left(\frac{\varepsilon_1\varepsilon_2\varepsilon_3}{(t-s)(s^2t-(\varepsilon_1+\varepsilon_2+\varepsilon_3)ts+(\varepsilon_1\varepsilon_2+\varepsilon_1\varepsilon_3+\varepsilon_2\varepsilon_3)t-\varepsilon_1\varepsilon_2\varepsilon_3)}\right)^{1/2},  \\
     P_L(s,t) &:= s^2t-2st^2+(\varepsilon_1+\varepsilon_2+\varepsilon_3)t^2-(\varepsilon_1\varepsilon_2+\varepsilon_1\varepsilon_3+\varepsilon_2\varepsilon_3)t+\varepsilon_1\varepsilon_2\varepsilon_3, \\
     P_N(s,t) &:= -s^2t^2+(\varepsilon_1+\varepsilon_2+\varepsilon_3)st^2-(\varepsilon_1\varepsilon_2+\varepsilon_1\varepsilon_3+\varepsilon_2\varepsilon_3)t^2+\varepsilon_1\varepsilon_2\varepsilon_3(2t-s).
   \end{align*}  
  
 \begin{proposition}\label{prop:SecondFundamentalForm}
   The second fundamental form of $S^* \sm \Sigma$ is given by
   $$
     L(s,t)ds^2+2M(s,t)\,ds\,dt+N(s,t)\,dt^2,
   $$ 
   where
   \begin{align*}
     L(s,t)
     = \frac{m(s,t)P_L(s,t)}{4t(s-\varepsilon_1)(s-\varepsilon_2)(s-\varepsilon_3)}, \quad &M(s,t) 
     = \frac{m(s,t)}{4t},\\
     &N(s,t)
     =  \frac{m(s,t)P_N(s,t)}{ 4t^2(t-\varepsilon_1)(t-\varepsilon_2)(t-\varepsilon_3)}. 
   \end{align*}  
 \end{proposition}
 \begin{proof}
   By definition, the functions $L,M,N$ are given by
   \begin{align*}
     L(s,t) = \skp{\nu(s,t)}{\partial_{ss}\Phi(s,t)}, \quad &M(s,t) = \skp{\nu(s,t)}{\partial_{st}\Phi(s,t)}, \\
     &N(s,t) = \skp{\nu(s,t)}{\partial_{tt}\Phi(s,t)}
   \end{align*}
   where $\nu(s,t)$ denotes the outer unit normal on $S \sm \Sigma$ at the point $\Phi(s,t)$.
   In Euclidean coordinates, a normal at $\eta=\Phi(s,t)$ is given by \\
   $\nabla\cN(\eta)=(t_1(\eta)\eta_1,t_2(\eta)\eta_2,t_3(\eta)\eta_3)$. So we define     
   \begin{align*}
     \tilde \nu_i(s,t) 
     &:= 2t_i(\Phi(s,t)) \Phi_i(s,t) 
     =  \left(-\frac{1}{\varepsilon_{i+1}}-\frac{1}{\varepsilon_{i+2}}+\frac{s}{\varepsilon_{i+1}\varepsilon_{i+2}}+\frac{1}{t}\right) \Phi_i(s,t)
   \end{align*}
    and obtain after normalization 
    \begin{align*}
      \nu_i(s,t)
      &=  \frac{m(s,t)t}{2}\tilde\nu_i(s,t).
   \end{align*}
   Using this formula for the unit normal field $\nu$, and plugging in the formulae for
   $\Phi_{ss},\Phi_{st},\Phi_{tt}$, one obtains the above expressions for $L(s,t),M(s,t),N(s,t)$.
 \end{proof}
 
  We continue with the formulae for the Gaussian and mean curvature, which were given in (A.1),(A.2) in
  Liess' work~\cite{Liess1991}.

  \begin{proposition}
    The Gaussian curvature at $\Phi(s,t)\in S^* \sm\Sigma$ is given by
    \begin{align*}
      K(s,t)
      = \frac{ (st-(\varepsilon_1+\varepsilon_2)t+\varepsilon_1\varepsilon_2)  (st-(\varepsilon_1+\varepsilon_3)t+\varepsilon_1\varepsilon_3) (st-(\varepsilon_2+\varepsilon_3)t+\varepsilon_2\varepsilon_3)}{
      (s-t)(s^2t-(\varepsilon_1+\varepsilon_2+\varepsilon_3)st + (\varepsilon_1\varepsilon_2+\varepsilon_1\varepsilon_3+\varepsilon_2\varepsilon_3)t-\varepsilon_1\varepsilon_2\varepsilon_3)^2}. 
    \end{align*}
  \end{proposition}
  \begin{proof}
    The determinant of the first fundamental form is given by
    \begin{align*}
     & (EG-F^2)(s,t) \\
     &=  \frac{s^2t-(\varepsilon_1+\varepsilon_2+\varepsilon_3)st + (\varepsilon_1\varepsilon_2+\varepsilon_1\varepsilon_3+\varepsilon_2\varepsilon_3)t-\varepsilon_1\varepsilon_2\varepsilon_3
     }{4t(s-\varepsilon_1)(s-\varepsilon_2)(s-\varepsilon_3)}   \\
     &\quad \times 
     \frac{\varepsilon_1\varepsilon_2\varepsilon_3(s-t)}{4t^2(t-\varepsilon_1)(t-\varepsilon_2)(t-\varepsilon_3)} \\
     &=  \frac{\varepsilon_1\varepsilon_2\varepsilon_3(s-t)(s^2t-(\varepsilon_1+\varepsilon_2+\varepsilon_3)st + (\varepsilon_1\varepsilon_2+\varepsilon_1\varepsilon_3+\varepsilon_2\varepsilon_3)t-\varepsilon_1\varepsilon_2\varepsilon_3)
     }{16t^3(s-\varepsilon_1)(s-\varepsilon_2)(s-\varepsilon_3)(t-\varepsilon_1)(t-\varepsilon_2)(t-\varepsilon_3)}  
    \end{align*}
    The determinant of the second fundamental form is
    \begin{align*}
      &(LN-M^2)(s,t) \\
      &= \frac{m(s,t)P_L(s,t)}{4t(s-\varepsilon_1)(s-\varepsilon_2)(s-\varepsilon_3)}
      \cdot\frac{m(s,t)P_N(s,t)}{ 4t^2(t-\varepsilon_1)(t-\varepsilon_2)(t-\varepsilon_3)} - \frac{m(s,t)^2}{16t^2}     \\
     &= \frac{m(s,t)^2P_L(s,t)P_N(s,t)}{16t^3(s-\varepsilon_1)(s-\varepsilon_2)(s-\varepsilon_3)(t-\varepsilon_1)(t-\varepsilon_2)(t-\varepsilon_3)} - \frac{m(s,t)^2}{16t^2}     \\
     &= \frac{m(s,t)^2\left[P_L(s,t)P_N(s,t) - t(s-\varepsilon_1)(s-\varepsilon_2)(s-\varepsilon_3)(t-\varepsilon_1)(t-\varepsilon_2)(t-\varepsilon_3)  \right] 
     }{16t^3(s-\varepsilon_1)(s-\varepsilon_2)(s-\varepsilon_3)(t-\varepsilon_1)(t-\varepsilon_2)(t-\varepsilon_3)}   \\
     &= \frac{\varepsilon_1\varepsilon_2\varepsilon_3  (st-(\varepsilon_1+\varepsilon_2)t+\varepsilon_1\varepsilon_2)  }{16t^3(s-\varepsilon_1)(s-\varepsilon_2)(s-\varepsilon_3)(t-\varepsilon_1)(t-\varepsilon_2)(t-\varepsilon_3)
     } \\
     &\quad \times \frac{(st-(\varepsilon_1+\varepsilon_3)t+\varepsilon_1\varepsilon_3)
     (st-(\varepsilon_2+\varepsilon_3)t+\varepsilon_2\varepsilon_3)}{(s^2t-(\varepsilon_1+\varepsilon_2+\varepsilon_3)ts+(\varepsilon_1\varepsilon_2+\varepsilon_1\varepsilon_3+\varepsilon_2\varepsilon_3)t-\varepsilon_1\varepsilon_2\varepsilon_3)}.
    \end{align*}    
    So the Gaussian curvature at the point $\Phi(s,t)$ is
    \begin{align*}
      &\quad K(s,t) \\
      &= \frac{(LN-M^2)(s,t)}{(EG-F^2)(s,t)} \\
      &= \frac{ (st-(\varepsilon_1+\varepsilon_2)t+\varepsilon_1\varepsilon_2)  (st-(\varepsilon_1+\varepsilon_3)t+\varepsilon_1\varepsilon_3) (st-(\varepsilon_2+\varepsilon_3)t+\varepsilon_2\varepsilon_3)}{
      (s-t)(s^2t-(\varepsilon_1+\varepsilon_2+\varepsilon_3)st + (\varepsilon_1\varepsilon_2+\varepsilon_1\varepsilon_3+\varepsilon_2\varepsilon_3)t-\varepsilon_1\varepsilon_2\varepsilon_3)^2} .
    \end{align*}
  \end{proof}

  Following Liess, we define $\alpha(s,t)$ to be the squared distance of the origin to the tangent plane
  through $\Phi(s,t)\in S^* \sm\Sigma$. Then
  \begin{align*}
    \alpha(s,t)
    &:= \left( \frac{\nabla\cN(\eta)\cdot\eta}{|\nabla\cN(\eta)|}\right)^2\Big|_{\eta=\Phi(s,t)} \\
    &= \frac{(\nabla\cN(\Phi(s,t))\cdot\Phi(s,t))^2}{|\nabla\cN(\Phi(s,t))|^2}  \\
    &= \frac{(\sum_{i=1}^3 t_i(\Phi(s,t))\Phi_i(s,t))^2}{\sum_{i=1}^3 t_i(\Phi(s,t))^2\Phi_i(s,t)^2}  \\
    &= 
    \frac{(t-s)\varepsilon_1\varepsilon_2\varepsilon_3}{s^2t-(\varepsilon_1+\varepsilon_2+\varepsilon_3)st  +
    (\varepsilon_1\varepsilon_2+\varepsilon_1\varepsilon_3+\varepsilon_2\varepsilon_3)t-\varepsilon_1\varepsilon_2\varepsilon_3}.
  \end{align*} 
  From this we deduce
  \begin{align*}
    K(s,t) = \frac{(\alpha(s,t)-\varepsilon_1)(\alpha(s,t)-\varepsilon_2)(\alpha(s,t)-\varepsilon_3)}{
    \alpha(s,t)(s-\varepsilon_1)(s-\varepsilon_2)(s-\varepsilon_3)}.
  \end{align*}
  
  \begin{proposition} \label{prop:meancurvature}
      The mean curvature at $\Phi(s,t)\in S^* \sm\Sigma$ is given by $(\alpha=\alpha(s,t))$
    \begin{align*}
     K_m (s,t)
      &= - \frac{1}{2} \big( \frac{s}{\sqrt{\alpha}} K(s,t) \\
      &\quad - \frac{1}{\sqrt{\alpha}} \left( \frac{(\alpha -
      \varepsilon_1)(\alpha-\varepsilon_2)}{(s-\varepsilon_1)(s-\varepsilon_2)} +
      \frac{(\alpha-\varepsilon_2)(\alpha-\varepsilon_3)}{(s-\varepsilon_2)(s-\varepsilon_3)} + \frac{(\alpha
      - \varepsilon_1)(\alpha - \varepsilon_3)}{(s-\varepsilon_1)(s-\varepsilon_3)} \right) \big).
    \end{align*}
  \end{proposition}
  \begin{proof}
    This is a consequence of the formula 
\begin{equation*}
\mathcal K_m(s,t) = \frac{G(s, t)L(s, t)-2F(s, t)M(s, t)+E(s, t)N(s, t))}{2(E(s, t)G(s, t)-F(s, t)^2)},
\end{equation*}    
    and the coefficients of first and second fundamental form computed in Propositions~\ref{prop:FirstFundamentalForm}-\ref{prop:SecondFundamentalForm}.
  \end{proof}
  We remark that our result deviates by the factor $- \frac{1}{2}$ from Liess' formula \cite[(A.2),~p.~91]{Liess1991}. This does not change the curvature properties, which we describe in the following:
   
 The Gaussian curvature $K$ vanishes precisely in those points where
   $\alpha(s,t)$ attains one of the values $\eps_1,\eps_2,\eps_3$. We assume $\eps_1<\eps_2<\eps_3$ for
   simplicity. Then one has $\eps_1<\alpha(s,t)<\eps_3$ so that the Gaussian curvature vanishes
   precisely at those points where $\alpha(s,t)=\eps_2$. Those are given by $t=T(s)$ where
   \begin{equation}
   \label{eq:ParametrizationHamiltonianCircles}
     T(s) = \frac{\varepsilon_1 \varepsilon_3 (\varepsilon_2 -s) }{s^2 - (\varepsilon_1 + \varepsilon_2 +
   \varepsilon_3) s + (\varepsilon_1 \varepsilon_2 + \varepsilon_1 \varepsilon_3 + \varepsilon_2
   \varepsilon_3) - \varepsilon_1 \varepsilon_3} = \frac{\varepsilon_1 \varepsilon_3}{\varepsilon_1 + \varepsilon_3 -s}.
   \end{equation}
   This is the parametrization of a one-dimensional submanifold that is called a Hamiltonian circle.
   Notice that each of the four singular point has its own Hamiltonian circle.
   (They are distinguished by $\sigma_i,\sigma_{i+2}\in\{-1,+1\}$ in Proposition~\ref{prop:singularpoints}).  
   By Proposition~\ref{prop:meancurvature} the mean curvature is non-zero along the Hamiltonian circles. We
   thus conclude that in the smooth regular part of Fresnel's wave surface, there is at least one
   principal curvature bounded away from zero. The Gaussian curvature is positive on the inner sheet and on
   the parts on the outer sheet that lie outside the Hamiltonian circles, while it is negative inside the
   Hamiltonian circles, i.e., close to the singular points on the outer sheet. In Proposition \ref{prop:ParametrizedCone} we show that the Hessian matrix at a singular point $D^2 p(\omega,\zeta)$ is indefinite.
   
    To summarize the geometric properties, we can perceive $S$ as union of two sheets $A$ and $B$, linked
together at the singular points, when $A$ is completely encased by $B$. $A$ is convex, but $B$ is not. Close
to the singular points, $B$ is not convex, and the Gaussian curvature is negative. Increasing geodesic
distance from the singular points on $B$, we reach the Hamiltonian circles: the curvature vanishes. Beyond
the Hamiltonian circles, $B$ is locally convex, too, and has again positive Gaussian curvature.
   
       \begin{figure}
      \includegraphics[scale=0.35]{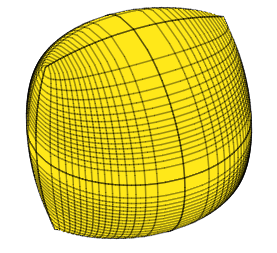}
      \includegraphics[scale=0.3]{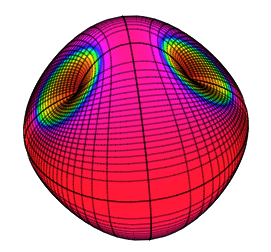}
      \includegraphics[scale=0.3]{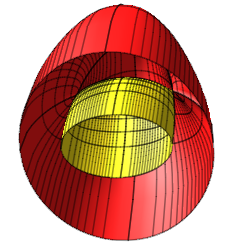}
      \caption{Fresnel's wave surface: inner sheet (top left) and outer
sheet (top right) for $\varepsilon_1 = 1$, $\varepsilon_2 = 5$, $\varepsilon_3 = 15$.
      The colours on the outer sheet highlight regions of identical
Gaussian curvature. The blue Hamiltonian
      circles encase the singular points.
       The contact of inner (yellow) and half of the outer sheet (red)
at two singular points is depicted
      in the figure below.}
      \label{fig:Fresnel}
    \end{figure}   
   
   \begin{corollary} \label{cor:Fresnel}
     The wave surface $S=\{\xi\in\R^3: p(\omega,\xi)=0\}$ admits a decomposition $S=S_1\cup S_2\cup S_3$, where  
     \begin{itemize}
       \item[(i)] $S_1$ is a compact smooth regular manifold with two non-vanishing principal
       curvatures in the interior,
       \item[(ii)] $S_2$ is a compact smooth regular manifold with one non-vanishing principal
       curvature in the interior,
       \item[(iii)] $S_3$ is the union of (small) neighbourhoods of the singular points described in
       Proposition~\ref{prop:singularpoints}.
     \end{itemize}
\end{corollary} 

For later sections, it will be important to have these curvature properties likewise for level sets
$\{p(\omega,\xi) = t\}_{t \in [-t_0,t_0]}$ for some $0< t_0 \ll 1$ with uniform bounds in $t$.\\
For this purpose, recall that for an implicitly defined surface $\{ \xi\in\R^3: F(\xi) = 0 \}$ the Gaussian
curvature is given by (cf. \cite[Corollary~4.2,~p.~643]{Goldman2005})
\begin{equation*}
K = - \begin{vmatrix} D^2 F & \nabla F \\ \nabla F^t & 0 \end{vmatrix} | \nabla F |^{-4}
\end{equation*}
and hence is continuous on the level sets as long as $F$ is smooth and $|\nabla F| \geq d > 0$.
This shows that $|K| \geq c/2 > 0$ on all level sets sufficiently close to $\{\xi\in\R^3: F(\xi) = 0 \}$,
where $|K| \geq c > 0$.
Furthermore, we have the following for the mean curvature of an implicitly defined surface (cf. \cite[Corollary~4.5,~p.~645]{Goldman2005}):
\begin{equation*}
  K_m = - \nabla \cdot \left( \frac{\nabla F}{|\nabla F|} \right).
\end{equation*}
Hence, again due to smoothness of $p$ and $|\nabla F| \geq d > 0$, along the curves on the level sets, where
the Gaussian curvature vanishes, we have one principal curvature bounded from below. Choosing the level sets
close to the original surface, we find one principal curvature bounded from below likewise on
all the layers.

\section{Generalized Bochner-Riesz estimates with negative index}
\label{section:BochnerRieszEstimates}

The purpose of this section is to show Theorem \ref{thm:GeneralizedBochnerRieszEstimates}. In the following
let $d \geq 2$ and $S = \{ (\xi',\psi(\xi')) : \, \xi' \in [-1,1]^{d-1} \}$ be a smooth surface with
$k \in \{1,\ldots,d-1 \}$ principal curvatures bounded from below. Let
\begin{equation*}
%\label{eq:GeneralizedBochnerRieszSection}
(T^\alpha f) \widehat (\xi) = \frac{(\xi_d - \psi(\xi'))_+^{-\alpha}}{\Gamma(1-\alpha)}  
\chi(\xi')\hat{f}(\xi), \; \chi \in C^\infty_c([-1,1]^{d-1}), \; 0 < \alpha < \frac{k+2}{2}.
\end{equation*}
We show strong estimates for a range of $p$ and $q$
\begin{equation*}
%\label{eq:BochnerRieszEstimates}
\| T^\alpha f \|_{L^q(\R^d)} \lesssim \| f \|_{L^p(\R^d)}
\end{equation*}
with weak endpoint estimates as stated in Theorem \ref{thm:GeneralizedBochnerRieszEstimates}.
We start with recapitulating Bochner-Riesz estimates in the elliptic case, which is understood best.
 
\subsection{Bochner-Riesz estimates with negative index for elliptic surfaces }
If $\psi$ is \textit{elliptic}, i.e.,the Hessian $\partial^2 \psi$ has eigenvalues of a fixed sign 
on $[-1,1]^{d-1}$, then $T^\alpha$ is a Bochner-Riesz operator of negative index. As explained above, we
shall show bounds also for possibly degenerate $\psi$, which will be useful in the next sections.
For solutions to time-harmonic Maxwell's equations we are interested in the case $d=3$, $\alpha = 1$,
corresponding to restriction--extension operator:
\begin{equation*}
T^1 f(x) = \int_{S_{\text{loc}}} e^{ix.\xi} \hat{f}(\xi) d \sigma(\xi).
\end{equation*}
We take a more general point of view to show that the considerations in the next section also apply in higher
dimensions and general $\alpha$. To put our results into context, we digress for a moment and recapitulate
results on the classical Bochner--Riesz problem.

\medskip
 
For $\alpha > 0$ recall
\begin{equation*}
%\label{eq:BochnerRieszRangeConjectured}
\mathcal{P}_\alpha(d-1) = 
\left\{ (x,y) \in [0,1]^2 \, : \, x-y \geq \frac{2\alpha}{d+1}, \; x > \frac{d-1}{2d} + \frac{\alpha}{d}, \; y
< \frac{d+1}{2d} - \frac{\alpha}{d} \right\}.
\end{equation*}
The Bochner--Riesz conjecture (for elliptic surfaces) with negative index states:
\begin{conjecture}
\label{conjecture:BochnerRiesz}
Let $d \geq 2$ and $0 < \alpha < \frac{d+1}{2}$. Then $T^\alpha$ is bounded from $L^p(\R^d)$ to $L^q(\R^d)$ if
and only if $(1/p,1/q) \in \mathcal{P}_\alpha(d-1)$.
\end{conjecture}
The necessity of these conditions was proved by B\"orjeson \cite{Boerjeson1986}. We refer to
\cite[Section~2.6]{KwonLee2020} for a survey, where the currently widest range is covered. 
In the special case $\alpha>\frac{1}{2}$ contributions  
are due to Bak--McMichael--Oberlin~\cite[Theorem~3]{BakMcMichaelOberlin}
and Guti\'errez \cite[Theorem~1]{Gutierrez2000}, see also~\cite{Sogge1986, Bak1997}. In \cite[Remark~2.3]{ChoKimLeeShim2005} was also pointed out that $T^\alpha: L^1(\R^d) \to L^\infty(\R^d)$ is bounded for $\alpha = \frac{d+1}{2}$. 
%We choose not to dwell on this endpoint further.

In the following we recall arguments from~\cite{ChoKimLeeShim2005,KwonLee2020}, which were needed for the
proofs and will be used in the next section for more general surfaces.
In the first step we decompose the multiplier distribution.
 
 \medskip
 
For $\alpha > 0$, let $D^\alpha \in \mathcal{S}'(\R^d)$ be defined by
\begin{equation*}
\langle D^{\alpha}, g \rangle_{(\mathcal{S}', \mathcal{S})} = 
\int_{\R^d} \frac{(\xi_d -
\psi(\xi^\prime))^{-\alpha}_+}{\Gamma(1-\alpha)} \chi(\xi^\prime) g(\xi) d\xi,
\end{equation*}
which is again extended by analytic continuation to the range $1 \leq \alpha<\frac{d+1}{2}$.
%\r{Would like to eliminate the symbol $\prescript{\,}{\mathcal{S}'}{\langle} D^\alpha, f \rangle_{\mathcal{S}}
%$. $f(\xi)$?}
We recall the following lemma to decompose the Fourier multiplier:
\begin{lemma}[{\cite[Lemma~2.1]{ChoKimLeeShim2005}}]
\label{lem:DecompositionLemma}
For $\alpha> 0$, there is a smooth function $\phi_\alpha$ satisfying $\text{supp } (\hat{\phi}_\alpha) \subseteq \{ t \, : \, |t| \sim 1 \}$ such that for all $g\in \mathcal{S}(\R^d)$,
\begin{equation*}
\langle D^{\alpha} , g \rangle_{(\mathcal{S}', \mathcal{S})}
 = \sum_{j \in \Z} 2^{\alpha j} \int_{\R^d} \phi_\alpha(2^j(\xi_d - \psi(\xi^\prime))) \chi(\xi^\prime) g(\xi) d\xi.
\end{equation*}
\end{lemma}

The importance for our analysis comes from $T^\alpha f(x)=\langle D^{\alpha} , g_x \rangle_{(\mathcal{S}',
\mathcal{S})}$ where $g_x(\xi)=e^{ix.\xi}\hat f(\xi)$. We are thus reduced to study the  operators 
\begin{equation*}
\widehat{T_\delta f}(\xi) = \phi_\alpha\left( \frac{\xi_d - \psi(\xi^\prime)}{\delta} \right) \chi(\xi^\prime) 
\hat{f}(\xi)
\end{equation*}
where $\phi_\alpha \in \mathcal{S}$ satisfies $\text{supp} \hat{\phi_\alpha} \subseteq \{ t : \, |t| \sim 1
\}$ and $\delta=2^{-j} > 0$.  
Fourier restriction estimates can be applied to $T_\delta$, and interpolation with a kernel estimate takes
advantage of the decomposition given by Lemma \ref{lem:DecompositionLemma}.
The Tomas--Stein restriction theorem (cf. \cite{Tomas1975,Stein1993}) suffices already for the sharp
estimates for the restriction--extension operator ($\alpha=1$) due to Guti\'errez \cite{Gutierrez2004}. Cho
\emph{et al.} \cite{ChoKimLeeShim2005} made further progress building on Tao's bilinear restriction theorem \cite{Tao2003}.
The most recent result is due to Kwon--Lee \cite{KwonLee2020} additionally using sharp oscillatory integral
estimates by Guth--Hickman--Iliopoulou \cite{GuthHickmanIliopoulou2019}.
 
%Note that for $d=3$ the results from \cite{ChoKimLeeShim2005,KwonLee2020} coincide and show that Conjecture
%\ref{conjecture:BochnerRiesz} is true for $\alpha > \frac{1}{5}$.
 
\subsection{Bochner-Riesz estimates with negative index for general non-flat surfaces}

In this section we extend the analysis to compact pieces of smooth regular hypersurfaces $S\subset
\R^d$ with $k$ non-vanishing principal curvatures and $k\in\{1,\ldots,d-1\}$, $d\in\N,d\geq 3$. 
Notice that the case $d=2$, $\alpha > 0$ was entirely settled by Bak \cite{Bak1997} and Guti\'errez~\cite[Theorem~1]{Gutierrez2000}.
Our argument is based on decompositions in Fourier space as in \cite{ChoKimLeeShim2005,KwonLee2020}. By
further localization in Fourier space we may suppose $S= \{(\xi^\prime, \psi(\xi^\prime)) \, : \,
\xi^\prime \in [-1,1]^{d-1} \}$.  
Notice that the case of $k=0$ corresponds to possibly flat surfaces, where no decay of the Fourier
transform can be expected. The case $k=d-1$ means that the Gaussian curvature is non-vanishing. In the
special case that all principal  curvatures have the same sign, the surface is elliptic and so is $\psi$. 

 In the following we show $L^p$-$L^q$ boundedness from Theorem~\ref{thm:GeneralizedBochnerRieszEstimates} of
the operator
\begin{equation*}
  (T^\alpha f) \widehat (\xi) = \frac{(\xi_d - \psi(\xi^\prime))_+^{-\alpha}}{\Gamma(1-\alpha)}  
  \chi(\xi^\prime)\hat{f}(\xi)
\end{equation*}
for $0<\alpha<\frac{k+2}{2}$ and $p,q\in [1,\infty]$ depending on the decay of
the Fourier transform of the surface measure and thus by the number of non-vanishing principal curvatures.
As above the operator $T^\alpha$ for $1\leq \alpha <\frac{k+2}{2}$ is
defined by analytic continuation (cf.~\cite{ChoKimLeeShim2005, KwonLee2020}). We comment on $\alpha = \frac{k+2}{2}$ after the proof of Lemma \ref{lem:KernelEstimate}.
 
\medskip
 
By Lemma \ref{lem:DecompositionLemma}, we decompose the operator $T^\alpha$ as distribution:
\begin{equation}
\label{eq:DecompositionTalpha}
T^\alpha f = \sum_{j \in \Z} 2^{\alpha j} \int e^{ix.\xi} \phi_\alpha(2^j(\xi_d - \psi(\xi')))  \chi(\xi')
\hat{f}(\xi) d\xi.
\end{equation}
where $\phi:=\phi_\alpha \in \mathcal{S}$ satisfies $\text{supp } \hat{\phi} \subseteq \{ t : \, |t| \sim 1
\}$.
In view of~\eqref{eq:DecompositionTalpha} we have
\begin{equation*}
%\label{eq:DecompositionTalphaII}
T^\alpha f = \sum_{j \in \Z} 2^{j \alpha} T_{2^{-j}} f
\end{equation*}
so that it suffices to consider the operators
\begin{equation*}
\widehat{T_\delta f}(\xi) = \phi( \frac{\xi_d - \psi(\xi^\prime)}{\delta} )  \chi(\xi^\prime)\hat{f}(\xi).
\end{equation*}
for $\delta > 0$. The contribution away from the surface corresponding to $\delta \gtrsim 1$ or $j \leq 0$ in
the above display can be estimated by Young's inequality, see below for a precise kernel estimate. This gives
summability for $j \leq 0$. We focus on the main contribution from $j \geq 0$.
 
\medskip
 
We start with using an $L^2$-restriction theorem for the surface $S$. To begin with, we recall the classical
result due to Littman \cite{Littman1963}; see also \cite[Section~VIII.5.8]{Stein1993}, which gives the
following decay of the Fourier transform of the surface measure $\mu$:
\begin{equation}
\label{eq:DecayFourierTransform}
|\hat{\mu}(\xi)| \lesssim \langle \xi \rangle^{- \frac{k}{2}}.
\end{equation}
By the $TT^*$-argument (cf. \cite{Tomas1975,GinibreVelo1979,KeelTao1998}) this can be recast into an
$L^2$-$L^q$ estimate as already recorded by Greenleaf \cite{Greenleaf1981}. 
The decay of the Fourier transform is  crucial for the verification of assumption (ii) in the following
special case of the abstract result from \cite[Theorem~1.2]{KeelTao1998}:

\begin{theorem}[Keel--Tao] \label{thm:KeelTao}
Let $(X,dx)$ be a measure space and $H$ a Hilbert space. Suppose that for each $t \in \R$ we have an
operator $U(t):H \to L^2(X)$ which satisfies the following assumptions for $\sigma>0$:
\begin{itemize}
\item[(i)] For all $t$ and $f \in H$ we have the energy estimate:
\begin{equation*}
%\label{eq:EnergyEstimate}
\| U(t) f \|_{L^2(X)} \lesssim \|f \|_H.
\end{equation*}
\item[(ii)] For all $t \neq s$ and $g \in L^1(X)$ we have the decay estimate
\begin{equation*}
\| U(s) (U(t))^* g \|_{L^\infty(X)} \lesssim (1+|t-s|)^{-\sigma} \| g \|_{L^1(X)}.
\end{equation*}
\end{itemize}
Then, for $q \geq \frac{2(1+ \sigma)}{\sigma}$, the estimate
\begin{equation*}
%\label{eq:StrichartzEstimate}
\| U(t) f \|_{L^q_{t,x}(\R \times X)} \lesssim \| f \|_H
\end{equation*}
holds.
\end{theorem}

The following two lemmas are the key ingredients in the proof of Theorem~\ref{thm:GeneralizedBochnerRieszEstimates}. 
Both rely on \eqref{eq:DecayFourierTransform}, which in turn depends on the lower bounds of the $k$
non-vanishing curvatures and $\| \psi \|_{C^N}$, $\| \chi \|_{C^N}$ for some large enough $N\in\N$. This leads
to the claimed stability in Theorem \ref{thm:GeneralizedBochnerRieszEstimates} of the estimates on $\psi$ and $\chi$.

In the following lemma we apply Theorem \ref{thm:KeelTao} to $T_\delta$ and $\sigma=\frac{k}{2}$:
 
\begin{lemma}
\label{lem:MultiplierBound}
Let $q \geq \frac{2(2+k)}{k}$. Then we have
\begin{equation}
\label{eq:MultiplierBound}
\| T_\delta f \|_{L^q(\R^d)} \lesssim \delta^{\frac{1}{2}} \| f \|_{L^2(\R^d)}.
\end{equation}
\end{lemma}
\begin{proof}
We perform a linear change of variables to rewrite
\begin{equation}
\label{eq:MultiplierII}
\begin{split}
(2 \pi)^d T_\delta f(x)
&= \int_{\R^d} e^{ix.\xi} \chi(\xi^\prime) \phi\big( \frac{\xi_d - \psi(\xi^\prime)}{\delta} \big)
\hat{f}(\xi) d\xi
\\
&= \int_{\R^d} e^{i(x^\prime.\xi^\prime) + x_d (\xi_d + \psi(\xi^\prime))} \chi(\xi^\prime)
\phi(\frac{\xi_d}{\delta}) \hat{f}(\xi^\prime,\xi_d+ \psi(\xi^\prime)) d\xi^\prime d\xi_d \\
&= \int_{\R} e^{i x_d \xi_d} \phi( \frac{\xi_d}{\delta} ) \int_{\R^{d-1}} e^{i(x^\prime. \xi^\prime + x_d
\psi(\xi^\prime))} \chi(\xi^\prime) \hat{f}(\xi^\prime,\xi_d + \psi(\xi^\prime)) d\xi^\prime d\xi_d.
\end{split}
\end{equation}
For the kernel in the inner integral we find by the assumptions on $\psi$
\begin{equation}
\label{eq:InnerKernelEstimate}
\left| \int_{\R^{d-1}} e^{i(x^\prime.\xi^\prime + x_d \psi(\xi^\prime))} \chi(\xi^\prime) d\xi^\prime \right|
\lesssim (1+|x_d|)^{- \frac{k}{2}}.
\end{equation}
From this and Theorem~\ref{thm:KeelTao}, applied
to $U(t)g(x')=\int_{\R^{d-1}}e^{ix'.\xi'+t\psi(\xi')}\chi(\xi')g(\xi')\,d\xi'$, 
 we infer
\begin{equation*}
\left\| \int_{\R^{d-1}} e^{i(x^\prime.\xi^\prime + x_d \psi(\xi^\prime)} \chi(\xi^\prime)
\hat{f}(\xi^\prime,\xi_d + \psi(\xi^\prime)) d\xi^\prime \right\|_{L^q(\R^d)}
\lesssim \|\hat{f}(\cdot,\xi_d + \psi(\cdot)) \|_{L^2(\R^{d-1})}.
\end{equation*}
By \eqref{eq:MultiplierII} and Minkowski's inequality, we find
\begin{align*}
\| T_\delta f \|_{L^q(\R^d)}
&\lesssim \int_\R |\phi(\frac{\xi_d}{\delta})|
\left\| \int_{\R^{d-1}} e^{i(x^\prime.\xi^\prime + x_d \psi(\xi^\prime))}
\chi(\xi^\prime) \hat{f}(\xi^\prime,\xi_d + \psi(\xi^\prime)) d\xi^\prime \right\|_{L^q(\R^d)} d\xi_d \\
&\lesssim \int_\R |\phi(\frac{\xi_d}{\delta})| \| \hat{f}(\cdot,\xi_d+\psi(\cdot)) \|_{L^2(\R^{d-1})} d\xi_d
\\
&\lesssim \delta^{\frac{1}{2}} \| f \|_{L^2(\R^d)}.
\end{align*}
The ultimate estimate follows from the Cauchy-Schwarz inequality, Plancherel's theorem, and inverting the change of variables.
\end{proof}
 
Further estimates for $T_\delta$ are derived from $T_\delta f = K_\delta \ast f$ where 
\begin{equation*}
%\label{eq:Kernel}
K_\delta(x) = \frac{1}{(2\pi)^{d}} \int_{\R^d} e^{ix.\xi} \chi(\xi^\prime) \phi(\frac{\xi_d -
\psi(\xi^\prime)}{\delta})d\xi.
\end{equation*}
Integration by parts leads to the following kernel estimate:
 
\begin{lemma}
\label{lem:KernelEstimate}
The function $K_\delta$ is supported in $\{(x^\prime,x_d)\in\R^d : |x_d| \sim \delta^{-1} \}$ and the
following estimates hold:
\begin{align}
\label{eq:KernelEstimate}
\begin{aligned}
|K_\delta(x)| &\lesssim_N \delta^{N+1} (1+\delta |x|)^{-N} &&, \text{ if } |x^\prime| \geq c |x_d|, \\
|K_\delta(x)| &\lesssim \delta^{\frac{k}{2}+1} &&, \text{ if } |x^\prime| \leq c|x_d|.
\end{aligned}
\end{align}
\end{lemma}
\begin{proof}
Changing variables $\xi_d \to \xi_d + \psi(\xi^\prime)$ and integrating in $\xi_d$, we have
\begin{equation*}
(2 \pi)^{d-1} K_\delta(x) = \delta \check{\phi}(\delta x_d) \int_{\R^{d-1}} e^{i(x^\prime.\xi^\prime + x_d \psi(\xi^\prime))}
\chi(\xi^\prime) d\xi^\prime.
\end{equation*}
Since $\check{\phi}$ is supported in $\{t : \, |t| \sim 1\}$, $K_\delta$ is supported in $\{(x^\prime,x_d): |x_d| \sim \delta^{-1} \}$. For the phase function $\Phi(\xi') = x'.\xi' + x_d \psi(\xi')$, we find
\begin{equation*}
|\nabla_{\xi^\prime} \Phi | \geq c_1|x|, \text{ if } |x^\prime| \geq c|x_d|.
\end{equation*}
So the method of non-stationary phase gives for $|x^\prime| \gtrsim |x_d|$
$$
|K_\delta(x)|
\lesssim_N \delta \|\check\phi\|_\infty (1+|x|)^{-N}
\lesssim_N \delta^{N+1} (1+\delta|x|)^{-N}.
$$
Notice that we used $\delta |x|\geq \delta |x_d| \gtrsim 1$ holds in this case.
On the other hand, \eqref{eq:InnerKernelEstimate} implies for $|x_d| \gtrsim |x^\prime|$
\begin{equation*}
|K_\delta(x)|
\lesssim \delta \left| \int_{\R^{d-1}} e^{i (x^\prime.\xi^\prime + x_d \psi(\xi^\prime))} \chi(\xi^\prime)
d\xi^\prime \right|
\lesssim \delta (1+ |x_d |)^{-\frac{k}{2}}
\lesssim \delta^{\frac{k+2}{2}}.
\end{equation*}
\end{proof}
We remark that the kernel estimate shows that $T^\alpha: L^1(\R^d) \to L^\infty(\R^d)$ also for $\alpha =
\frac{k+2}{2}$ by the same argument as in \cite[Remark~2.3]{ChoKimLeeShim2005}.
 
With Lemma~\ref{lem:KernelEstimate} at hand, we may now localize $f$ to cubes of size
$\delta^{-1}$ by the following argument, originally due to Fefferman \cite{Fefferman1970}; see also \cite[p.~422--423]{Stein1993}, and \cite{Lee2003,ChoKimLeeShim2005}:
Let $(Q_j)_{j \in \mathbb{Z}^d}$ denote a finitely overlapping covering of $\R^d$ with cubes of sidelength
$2\delta^{-1}$ centered at $j \delta^{-1}$ and aligned parallel to the coordinate axes.
Let $C_d>0$ be such that $|j-k|>C_d$ implies $\text{dist}(Q_j,Q_k) \gtrsim \delta^{-1}  |j-k|$ uniformly
with respect to $j,k,\delta$. Then, we obtain
\begin{equation*}
\begin{split}
\| T_\delta f \|_{L^q(\R^d)} \lesssim \big( \sum_{j \in \mathbb{Z}^d} \| T_\delta f \|_{L^q(Q_j)}^q \big)^{\frac{1}{q}} 
&\lesssim \big( \sum_{j\in\Z^d} \big( \sum_{|k-j| \leq C_d} \| T_\delta f_k \|_{L^q(Q_j)} \big)^q
\big)^{\frac{1}{q}} \\
&\qquad + \big( \sum_{j\in\Z^d} \big( \sum_{|k-j|> C_d} \|T_\delta f_k \|_{L^q(Q_j)} \big)^q
\big)^{\frac{1}{q}}.
\end{split}
\end{equation*}
If $|k-j|> C_d$, we use the first kernel estimate in ~\eqref{eq:KernelEstimate} and obtain
for all $N\in\N$
\begin{align*}
\| T_\delta f_k \|_{L^q(Q_j)}
&\lesssim \left(\int_{Q_j} \left|\int_{\R^d} K_\delta(x-y)f_k(y)\,dy\right|^q\,dx \right)^{\frac{1}{q}} \\
&\lesssim_N \delta^{N+1}(1+\delta\dist(Q_j,Q_j))^{-N} \left(\int_{Q_j} \left(\int_{Q_k} |f_k(y)|
\,dy\right)^q\,dx \right)^{\frac{1}{q}} \\
&\lesssim_N \delta^{N+1}(1+|j-k|)^{-N}\left(\int_{Q_j} \|f_k\|_p^q |Q_k|^{\frac{q}{p'}}\,dx
\right)^{\frac{1}{q}} \\
&\lesssim_N \delta^{N+1} (1+|j-k|)^{-N} \delta^{-\frac{d}{q}-\frac{d}{p'}} \|f_k\|_{L^p(\R^d)} \\
&\lesssim_N \delta^{N+1-\frac{d}{q}-\frac{d}{p'}} (1+|j-k|)^{-N} \|f_k\|_{L^p(\R^d)}.
\end{align*}
Hence, choosing $N\in\N$ large enough, these terms allow for summation by Young's inequality for series:
\begin{equation*}
\begin{split}
&\quad \big( \sum_{j\in\Z^d} \big( \sum_{|k-j| > C_d} \| T_\delta f_k \|_{L^q(Q_j)} \big)^q \big)^{\frac{1}{q}} \\
&\lesssim_N \delta^{N+1-\frac{d}{q}-\frac{d}{p'}} \big( \sum_{j\in\Z^d} \big( \sum_{k\in\Z^d}
(1+|j-k|)^{-N} \| f_k \|_{L^p(\R^d)} \big)^q \big)^{\frac{1}{q}} \\
&\lesssim_N \delta^{N+1-\frac{d}{q}-\frac{d}{p'}} \big( \sum_{k\in\Z^d} \| f_k \|_{L^p(\R^d)}^q
\big)^{\frac{1}{q}}
\\
&\lesssim \delta^{N+1-\frac{d}{q}-\frac{d}{p'}} \big( \sum_{k\in\Z^d} \| f_k \|_{L^p(\R^d)}^p
\big)^{\frac{1}{p}}
\\
&\lesssim \delta^{N+1-\frac{d}{q}-\frac{d}{p'}} \| f \|_{L^p(\R^d)}.
\end{split}
\end{equation*}
The penultimate estimate follows from the embedding $\ell^p \hookrightarrow \ell^q$, $p \leq q$, and the last
line from the finite overlapping property.
For the ``diagonal'' set, $|k-j| \leq C_d$, we use~\eqref{eq:MultiplierBound} as well as H\"older's inequality:
\begin{equation*}
\| T_\delta f_k \|_{L^q(Q_j)} \lesssim \delta^{\frac{1}{2}} \| f_k \|_{L^2(\R^d)} \lesssim \delta^{\frac{d}{p}
- \frac{d-1}{2}} \| f_k \|_{L^p(\R^d)}.
\end{equation*}
Here we have used that the support of $f_k$ has measure $\sim \delta^{-d}$ and $p\geq 2$. We conclude
\begin{equation*}
\begin{split}
\big( \sum_{j\in\Z^d} \big( \sum_{|k-j| \leq C_d} \| T_\delta f_k \|_{L^q(Q_j)} \big)^q \big)^{\frac{1}{q}}
&\lesssim \delta^{\frac{d}{p} - \frac{d-1}{2}} \big( \sum_{j\in\Z^d} \big( \sum_{|k-j| \leq C_d} \| f_k \|_{L^p(\R^d)} \big)^q \big)^{\frac{1}{q}} \\
&\lesssim \delta^{\frac{d}{p} - \frac{d-1}{2}} \big( \sum_{j\in\Z^d} \| f_j\|_{L^p(\R^d)}^q \big)^{\frac{1}{q}} \\
&\lesssim \delta^{\frac{d}{p} - \frac{d-1}{2}} \| f \|_{L^p(\R^d)},
\end{split}
\end{equation*}
like above due to the embedding $\ell^p \hookrightarrow \ell^q$ for $p \leq q$ and the finite overlapping property.
Combining the off-diagonal and the diagonal estimates for large enough~$N$, we get 
\begin{equation}
\label{eq:RestrictionBound}
2^{j \alpha} \| T_{2^{-j}} f \|_{L^q(\R^d)} 
\lesssim 2^{j \big( \frac{d-1}{2} - \frac{d}{p} +\alpha \big)} \| f\|_{L^p(\R^d)}
\end{equation}
for $q \geq \frac{2(2+k)}{k} $ and $2 \leq p \leq q$.
 
\medskip
 
By the kernel estimate~\eqref{eq:KernelEstimate}, we find $|K_\delta(x)|\les \delta^{\frac{k+2}{2}}$
for all $x\in\R^d$ and thus
\begin{equation}
\label{eq:LinftyEstimate}
2^{j \alpha} \| T_{2^{-j}} f \|_{L^\infty(\R^d)} \lesssim 2^{j(\alpha - \frac{k+2}{2})} \| f \|_{L^1(\R^d)}.
\end{equation}
%By disjointness of the kernels of $T_{2^{-j}}$ for $j \geq 1$ and better estimates for $j \leq 0$, we see
%that \b{$T^\alpha: L^1(\R^d) \to L^\infty(\R^d)$ for all $\alpha <\frac{k+2}{2}$.
 %We refer to~\cite[Remark~2.2]{ChoKimLeeShim2005} for details.
% }
Next we interpolate~\eqref{eq:RestrictionBound} and~\eqref{eq:LinftyEstimate} to prove our bounds.
To this end we distinguish the cases  $\frac{1}{2}<\alpha < \frac{k+2}{2}$ and $0<\alpha\leq \frac{1}{2}$. We obtain weak endpoint estimates using a special case of Bourgain's summation argument (cf.
\cite{Bourgain1985,CarberySeegerWaingerWright1999}). The present version is taken
from~\cite[Lemma~2.5]{ChoKimLeeShim2005}, see also~\cite[Lemma~2.3]{Lee2003} for an elementary proof:
 
\begin{lemma}
\label{lem:SummationLemma} 
Let $\varepsilon_1,\varepsilon_2 > 0$, $1 \leq p_1, \, p_2 <\infty$, $1 \leq q_1,q_2 < \infty$. For
every $j \in \mathbb{Z}$ let $\mathcal T_j$ be a linear operator, which satisfies
\begin{align*}
\| \mathcal T_jf \|_{q_1} &\leq M_1 2^{\varepsilon_1 j}  \| f \|_{p_1}, \\
\| \mathcal T_jf \|_{q_2} &\leq M_2 2^{-\varepsilon_2 j} \| f \|_{p_2}.
\end{align*}
Then, for $\theta,p,q$ defined by $\theta = \frac{\varepsilon_2}{\varepsilon_1
+\varepsilon_2}$, $\frac{1}{q} = \frac{\theta}{q_1} + \frac{1-\theta}{q_2}$ and $\frac{1}{p} =
\frac{\theta}{p_1} + \frac{1-\theta}{p_2}$, the following hold:
\begin{align}
\label{eq:SummationI}
\| \sum_{j\in\Z} \mathcal T_jf \|_{q,\infty} &\leq C M_1^\theta M_2^{1-\theta} \| f \|_{p,1}, \\
\label{eq:SummationII} \| \sum_{j\in\Z} \mathcal T_jf \|_q &\leq C M_1^\theta M_2^{1-\theta} \| f \|_{p,1} 
&&\hspace{-1cm}\text{if } q_1 = q_2 = q, \\
\label{eq:SummationIII} \| \sum_{j\in\Z} \mathcal T_jf \|_{q,\infty} &\leq C M_1^\theta M_2^{1-\theta}  \| f \|_{p}
&&\hspace{-1cm}\text{ if } p_1 = p_2.
\end{align}
\end{lemma}

\noindent
\textit{Proof of Theorem~\ref{thm:GeneralizedBochnerRieszEstimates}~(i): $\frac{1}{2} < \alpha <
\frac{k+2}{2}$.}\; Interpolating the estimates at the points $(\frac{1}{2},\frac{1}{q_1})$, $\frac{1}{q_1} \in \big[
0,\frac{k}{2(k+2)} \big]$ from~\eqref{eq:RestrictionBound} and $A:=(1,0)$ from~\eqref{eq:LinftyEstimate} gives
\begin{equation*}
  2^{j\alpha}\| T_{2^{-j}} f \|_{L^q(\R^d)} \lesssim 2^{j(\alpha+\frac{k}{2} - \frac{k+1 }{p})} \| f\|_{L^p(\R^d)}
\end{equation*}
for $\frac{1}{p} \in [\frac{1}{2},1]$ and $\frac{1}{q} \leq \frac{k}{k+2} \big( 1 - \frac{1}{p} \big)$. 
We use this bound for $p_1,p_2,q_1,q_2$ given by
$$
  \alpha+\frac{k}{2} - \frac{k+1 }{p_1}=\eps,\quad \alpha+\frac{k}{2} - \frac{k+1 }{p_2}=-\eps,\quad
  \frac{1}{q_i} = \frac{k}{k+2} \big( 1 - \frac{1}{p_i} \big) \quad (i=1,2).
$$ 
Here, $\eps>0$ is chosen so small that $\frac{1}{p_1},\frac{1}{p_2}\in [\frac{1}{2},1]$ holds, which is  
possible thanks to our assumption $\frac{1}{2}<\alpha<\frac{k+2}{2}$. So~\eqref{eq:SummationI} from
Lemma~\ref{lem:SummationLemma} gives  
\begin{equation}\label{eq:BoundIB}
  \| T^\alpha f \|_{L^{q,\infty}(\R^d)} 
  \lesssim  \|f\|_{L^{p,1}(\R^d)} \;\text{where } (\frac{1}{p},\frac{1}{q})
  = \left( \frac{k+2\alpha}{2(k+1)}, \frac{k(k+2-2\alpha)}{2(k+1)(k+2)} \right) =:B_{\alpha,k}.   
\end{equation} 
Furthermore, since $T_{2^{-j}}$ coincides with its dual, we have under the same conditions on $p,q$ as above:
\begin{equation*}
  2^{j\alpha}\| (T_{2^{-j}})^* g \|_{L^{p'}(\R^d)} \lesssim 2^{j(\alpha+\frac{k}{2} - \frac{k+1 }{p})} \|
  g\|_{L^{q'}(\R^d)}.
\end{equation*}
So~\eqref{eq:SummationIII} gives for $p_1=p_2=1,q_1=q_2=
\frac{2(k+1)}{k+2\alpha}$ the estimate 
$$
  \| (T^\alpha)^* g \|_{L^{ (\frac{2(k+1)}{k+2\alpha})',\infty}(\R^d)} 
  \lesssim  \|g\|_{L^1(\R^d)} 
$$
and hence, by duality,
\begin{equation}\label{eq:BoundIC}
  \| T^\alpha f \|_{L^q(\R^d)} 
  \lesssim  \|f\|_{L^{p,1}(\R^d)} \;\text{where } (\frac{1}{p},\frac{1}{q})
  =  \left( \frac{k+2\alpha}{2(k+1)}, 0 \right)=:C_{\alpha,k}.   
\end{equation} 
Since $T^\alpha$ coincides with its dual, the estimates \eqref{eq:BoundIB},\eqref{eq:BoundIC} imply 
\begin{align} \label{eq:BoundIBCprime}
  \begin{aligned}
  \| T^\alpha f \|_{L^{q,\infty}(\R^d)} 
  &\lesssim  \|f\|_{L^{p,1}(\R^d)} \quad\text{where } (\frac{1}{p},\frac{1}{q})
  =  \left( \frac{ k^2 + 2(2+\alpha)k +4}{2(k+1)(k+2)}, \frac{k+2-2\alpha}{2(k+1)} \right) =:B_{\alpha,k}', \\
  \| T^\alpha f \|_{L^{q,\infty}(\R^d)} 
  &\lesssim  \|f\|_{L^{p}(\R^d)} \quad\text{where } (\frac{1}{p},\frac{1}{q})
  =  \left(1, \frac{k+2-2\alpha}{2(k+1)}\right) =:C_{\alpha,k}'.
  \end{aligned}
\end{align} 
Finally, we have the trivial strong estimate 
\begin{equation}\label{eq:BoundIA}
  \|T^\alpha f\|_{L^q(\R^d)}\les \|f\|_{L^p(\R^d)} \quad\text{for }(\frac{1}{p},\frac{1}{q})=(1,0)=:A.
\end{equation} 
We refer to
Figure~\ref{fig:Rieszdiagram1} for a visualization of the situation. 
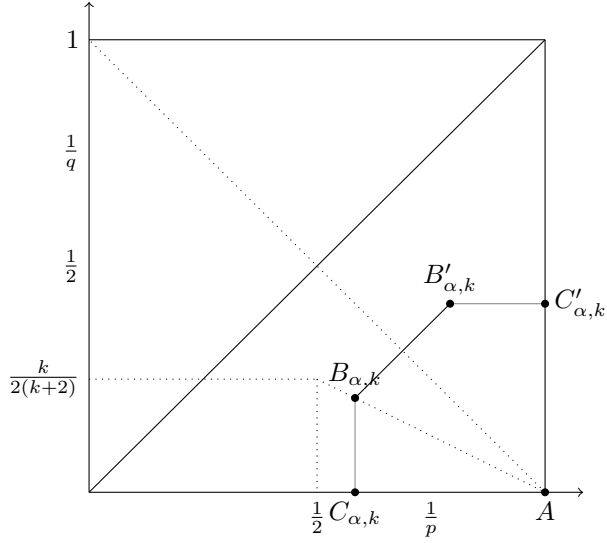
\begin{figure} 
  \centering
\begin{tikzpicture}[scale=0.5] 
\draw[->] (0,0) -- (13,0); \draw[->] (0,0) -- (0,13);
\draw (0,0) --(12,12); \draw(0,12) -- (12,12); \draw (12,12) -- (12,0);
\coordinate (E) at (6,3);
\coordinate [label=left:$\frac{k}{2(k+2)}$] (EX) at (0,3);
\coordinate [label=above:$B_{\alpha,k}$] (B) at (7,2.5);
\coordinate [label=above:$B_{\alpha,k}'$] (B') at (9.5,5);
\coordinate [label=below:$C_{\alpha,k}$] (C) at (7,0);
\coordinate [label=right:$C_{\alpha,k}'$] (C') at (12,5);
\coordinate [label=left:$\frac{1}{2}$] (Y) at (0,6);
\coordinate [label=left:$1$] (YY) at (0,12);
\coordinate [label=below:$\frac{1}{2}$] (X) at (6,0);
\coordinate [label=below:$A$] (XX) at (12,0);
\coordinate [label=left:$\frac{1}{q}$] (YC) at (0,9);
\coordinate [label=below:$\frac{1}{p}$] (XC) at (9,0);
\draw [dotted] (EX) -- (E); \draw [dotted] (E) -- (X); \draw [dotted] (E) -- (XX); \draw [dotted] (XX) -- (YY);
\draw [help lines] (C) -- (B); \draw (B) -- (B');
\draw [help lines] (B') -- (C');
\foreach \point in {(C),(C'),(B),(B'),(XX)}
\fill [black, opacity = 1] \point circle (3pt);
\end{tikzpicture} 
  \caption{Riesz diagram for $T^\alpha$ with $\frac{1}{2}<\alpha<\frac{k+2}{2}$.}
   \label{fig:Rieszdiagram1}
\end{figure}
 From the estimates~\eqref{eq:BoundIB}-\eqref{eq:BoundIA}
we now derive our claim using the real interpolation identity (cf. \cite[Theorem~5.3.1]{BerghLoefstrom1976})
\begin{equation*}
%\label{eq:RealInterpolation}
(L^{p_1,q_1}(\R^d),L^{p_2,q_2}(\R^d))_{\theta,q} = L^{p,q}(\R^d) \quad\text{ for } \frac{1}{p} =
\frac{\theta}{p_1} + \frac{1-\theta}{p_2}, \; \theta \in (0,1)
\end{equation*}
as well as the Lorentz space embeddings $L^{\tilde{p}}(\R^d) = L^{\tilde{p},\tilde p}(\R^d) \hookrightarrow
L^{\tilde{p},\tilde{q}}(\R^d)$ for $\tilde{q} \geq \tilde{p}$. In this way, 
we obtain strong estimates for the operator $T^\alpha$ in the interior of the pentagon 
$\text{conv}(A,C_{\alpha,k},B_{\alpha,k},{B'}_{\alpha,k},{C'}_{\alpha,k})$ as well on
$(B_{\alpha,k},{B'}_{\alpha,k})$:
Real interpolation with parameters $(\theta,\tilde{q})$ gives the estimate
\begin{equation*}
\| T^\alpha f \|_{L^{\tilde{q},\tilde{q}}(\R^d)} \lesssim \| f \|_{L^{\tilde{p},\tilde{q}}(\R^d)} \lesssim \|
f \|_{L^{\tilde{p}}(\R^d)}
\end{equation*}
for $(1/\tilde{p},1/\tilde{q}) \in (B_{\alpha,k},B'_{\alpha,k})$.

We have shown strong bounds for $p$, $q$ such that
\begin{equation*}
\frac{1}{p} > \frac{k+2\alpha}{2(k+1)}, \qquad \frac{1}{q} < \frac{k+2-2\alpha}{2(k+1)}, \qquad \frac{1}{p} -
\frac{1}{q} \geq \frac{2\alpha}{k+2}.
\end{equation*}
All these estimates are valid for $\alpha>\frac{1}{2}$. The strong bounds for $\alpha=\frac{1}{2}$ 
can be obtained using Stein's interpolation theorem for analytic families of operators and the estimates for
$\alpha>\frac{1}{2}$ just proved and for $\alpha<\frac{1}{2}$ that we prove below.

\medskip
 
\noindent 
\textit{Proof of Theorem~\ref{thm:GeneralizedBochnerRieszEstimates}~(ii): $0< \alpha <\frac{1}{2}$.}\;

We use the estimates from~\eqref{eq:RestrictionBound} and the same interpolation procedure as above to find
\begin{align*}%\label{eq:BoundIIBC}
  \begin{aligned}
  \| T^\alpha f \|_{L^{q,\infty}(\R^d)} 
  &\lesssim  \|f\|_{L^{p,1}(\R^d)}, \;\text{where } (\frac{1}{p},\frac{1}{q})
  = \left(  \frac{d-1+2\alpha}{2d}, \frac{k}{2(2+k)}   \right) =:B_{\alpha,k}, \\
  \| T^\alpha f \|_{L^q(\R^d)} 
  &\lesssim  \|f\|_{L^{p,1}(\R^d)}, \;\text{where } (\frac{1}{p},\frac{1}{q})
  =  \left( \frac{d-1+2\alpha}{2d}, 0 \right)=:C_{\alpha,k}.
  \end{aligned}   
\end{align*} 
By duality,
\begin{align*}%\label{eq:BoundIIBCprime}
  \begin{aligned}
  \| T^\alpha f \|_{L^{q,\infty}(\R^d)} 
  &\lesssim  \|f\|_{L^{p,1}(\R^d)}, \;\text{where } (\frac{1}{p},\frac{1}{q})
  = \left(  \frac{4+k}{2(2+k)}, \frac{d+1-2\alpha}{2d} \right) =:B_{\alpha,k}', \\
  \| T^\alpha f \|_{L^{q,\infty}(\R^d)} 
  &\lesssim  \|f\|_{L^p(\R^d)}, \;\text{where } (\frac{1}{p},\frac{1}{q})
  =  \left(1, \frac{d+1-2\alpha}{2d} \right)=:C_{\alpha,k}'.
  \end{aligned}   
\end{align*}

Again we have the trivial strong estimate~\eqref{eq:BoundIA}. Interpolating these estimates as above, we get
strong bounds precisely for $p,q$ such that 
 \begin{equation*}
\frac{1}{p} > \frac{d-1+2\alpha}{2d}, \qquad \frac{1}{q} < \frac{d+1-2 \alpha}{2d}, \qquad \frac{1}{p} -
\frac{1}{q} \geq \frac{2(d-1+2\alpha) + k(2\alpha -1)}{2d(2+k)}.
\end{equation*}
\begin{figure} 
  \centering
\begin{tikzpicture}[scale=0.5]
\draw[->] (0,0) -- (13,0); \draw[->] (0,0) -- (0,13);
\draw (0,0) --(12,12); \draw(0,12) -- (12,12); \draw (12,12) -- (12,0);
\coordinate (E) at (6,3);
\coordinate [label=left:$\frac{k}{2(k+2)}$] (EX) at (0,3);
\coordinate [label=above:$B_{\alpha,k}$] (B) at (5,3);
\coordinate [label=above:$B_{\alpha,k}'$] (B') at (9,7);
\coordinate [label=below:$C_{\alpha,k}$] (C) at (5,0);
\coordinate [label=right:$C_{\alpha,k}'$] (C') at (12,7);
\coordinate [label=left:$\frac{1}{2}$] (Y) at (0,6);
\coordinate [label=left:$1$] (YY) at (0,12);
\coordinate [label=below:$\frac{1}{2}$] (X) at (6,0);
\coordinate [label=below:$A$] (XX) at (12,0);
\coordinate [label=left:$\frac{1}{q}$] (YC) at (0,9);
\coordinate [label=below:$\frac{1}{p}$] (XC) at (9,0);
\draw [dotted] (EX) -- (E); \draw [dotted] (E) -- (X); \draw [dotted] (XX) -- (YY);
\draw [help lines] (C) -- (B); \draw (B) -- (B');
\draw [help lines] (B') -- (C');
\foreach \point in {(C),(C'),(B),(B'),(XX)}
\fill [black, opacity = 1] \point circle (3pt);
\end{tikzpicture}
 \caption{Riesz diagram for $T^\alpha$ with $0<\alpha<\frac{1}{2}$.}
  \label{fig:Rieszdiagram2}
\end{figure}
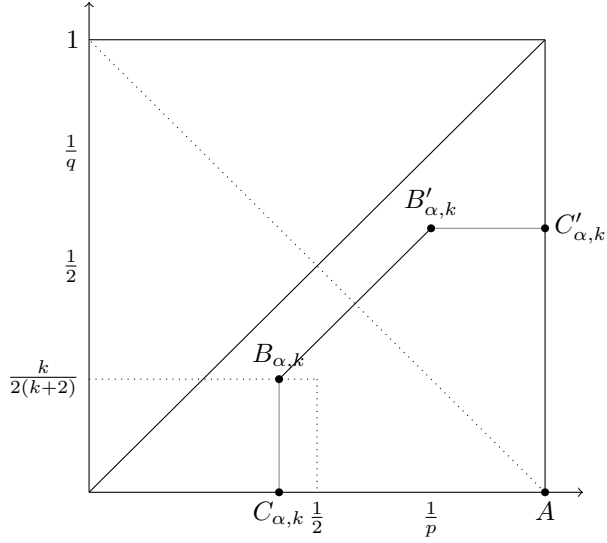
%In the limiting case $\alpha = 1/2$ we still find favorable bounds in the region described above. To find bounds in the complement to apply \eqref{eq:SummationI}, we can interpolate with the trivial estimate at $(1/2,1/2)$ from Plancherel's theorem.

This finishes the proof of Theorem~\ref{thm:GeneralizedBochnerRieszEstimates}. \qed

\subsection{Necessary conditions for generalized Bochner-Riesz estimates with negative index}
In this subsection we discuss necessary conditions for estimates
\begin{equation}
\label{eq:GeneralizedBochnerRieszNecessary}
\| T^\alpha f \|_{L^q(\R^d)} \lesssim \| f \|_{L^p(\R^d)}.
\end{equation}
We shall see that for $\alpha \geq 1/2$ the established strong estimates are sharp, but for $0<\alpha<1/2$
these are in general not. For this purpose, we compare to the estimates for elliptic surfaces in lower
dimensions where the bounds are known to be sharp, see~\cite[p.1419]{KwonLee2020}.
 
Suppose that for $d \geq 3$, there is $1 \leq k \leq d-1$ and $(\tilde{p},\tilde{q})$ such that
\eqref{eq:GeneralizedBochnerRieszNecessary} holds true for all regular hypersurfaces with $k$ non-vanishing
principal curvatures.
Then, let $d_1:=k+1$ and let $S= \{(\xi', \psi(\xi')) \in \R^{d_1} \, : \xi' \in B(0,c) \}$ be an elliptic
surface with $k= d_1-1$ positive principal curvatures. This can be trivially embedded into $\R^d$ considering $S' = \{
(\xi',\xi'', \psi(\xi')) \in \R^{d_1+d_2} \, : \xi' \in B(0,c) \}$.
We consider the operator
\begin{equation*}
(T^\alpha f) \widehat (\xi) = \frac{1}{\Gamma(1-\alpha)} \frac{\chi(\xi')}{(\xi_d - \psi(\xi'))_+^\alpha} \hat{f}(\xi).
\end{equation*}
Apparently,
\begin{equation*}
K^\alpha(x) = \frac{1}{(2 \pi)^d} \int_{\R^d} e^{ix.\xi} \frac{1}{\Gamma(1-\alpha)} \frac{\chi(\xi')}{(\xi_d - \psi(\xi'))_+^\alpha} d\xi = L^\alpha(x') \delta(x''),
\end{equation*}
where $x'=(x_1,\ldots,x_{d_1-1},x_{d_1+d_2})$, $x'' = (x_{d_1}, \ldots, x_{d_1+d_2-1})$, and
\begin{equation*}
L^\alpha(x) = \frac{1}{(2 \pi)^{d_1}} \int_{\R^{d_1}} e^{ix'.(\xi',\xi_{d})} \frac{1}{\Gamma(1-\alpha)} \frac{\chi(\xi')}{(\xi_d - \psi(\xi'))_+^\alpha} d\xi' d\xi_d. 
\end{equation*}
As $L^\alpha$ is the kernel of a Bochner-Riesz operator with negative index for an elliptic surface in $\R^{k+1}$, we know that for $\frac{1}{2} \leq \alpha < \frac{k+2}{2}$ the corresponding operator $R^\alpha f = L^\alpha * f: L^p(\R^{k+1}) \to L^q(\R^{k+1})$ is bounded if and only if $(1/p,1/q) \in \mathcal{P}_\alpha(k)$.
%\begin{equation}
%\label{eq:NecessaryConditions}
%\frac{1}{p} > \frac{k+2\alpha}{2(k+1)}, \qquad \frac{1}{q} < \frac{k+2-2\alpha}{2(k+1)}, \qquad \frac{1}{p} -
%\frac{1}{q} \geq \frac{2\alpha}{k+2}.
%\end{equation}
For $f \in L^p(\R^{k+1})$ consider $\tilde{f}(x) = f(x') \phi(x'')$ with $\phi \in C^\infty_c$. Using that
$T^\alpha:L^p(\R^d)\to L^q(\R^d)$ is bounded, we find
\begin{equation*}
\| R^\alpha f \|_{L^{\tilde{q}}(\R^{d_1})} \| \phi \|_{L^{\tilde{q}}(\R^{d_2})} = \| T^\alpha \tilde{f}
\|_{L^{\tilde{q}}(\R^d)} \lesssim \| \tilde{f} \|_{L^{\tilde{p}}(\R^d)} \lesssim \| f
\|_{L^{\tilde{p}}(\R^{d_1})} \| \phi \|_{L^{\tilde{p}}(\R^{d_2})}.
\end{equation*}
Hence, $R^\alpha:L^{\tilde p}(\R^{d_1})\to L^{\tilde q}(\R^{d_1})$ is bounded. By the sharpness of our
conditions for elliptic surfaces we infer $(1/\tilde{p},1/\tilde{q}) \in \mathcal
P_{\alpha}(d_1-1)=\mathcal{P}_\alpha(k)$, which is all we had to show. 

\medskip

On the other hand, we see that the estimates proved in Theorem \ref{thm:GeneralizedBochnerRieszEstimates} are not sharp for $ 0 <\alpha < 1/2$ as in the elliptic case better estimates are known to hold. Apparently, for $0<\alpha<1/2$ the geometry of the surface becomes more important. We believe that the optimal estimates will also depend on the difference between positive and negative curvatures as for oscillatory integral operators (cf. \cite{Wisewell2005,BourgainGuth2011,GuthHickmanIliopoulou2019}).

\section{Estimates for the regular part}
\label{section:CurvatureBoundedBelow}

In this section we estimate the contribution of $(E,H)$  with Fourier support close to 
smooth and regular component of the Fresnel surface by proving
Proposition~\ref{prop:PositiveCurvatureEstimate} and Proposition~\ref{prop:HamiltonEstimates}.
We recall that the first proposition deals with those parts where two principal curvatures are non-zero,
whereas the latter proposition deals with frequencies close to the Hamiltonian circles where only one
principal curvature is bounded away from zero. As explained in
the Introduction, our estimates result from uniform estimates for the Fourier multipliers
$(P(\xi)+i\delta)^{-1}$ as $\delta\to 0$ with $P(\xi)=p(\omega,\xi)$. We stress
that $\omega\in\R\sm\{0\}$ is fixed from now on.

\medskip

We first use our estimates for the Bochner-Riesz operator $T^\alpha$ from the previous section to
prove a Fourier restriction-extension estimate related to the two parts of the Fresnel surface mentioned
above. To carry out the estimates for both parts, we change from implicit to graph representation and apply
Theorem~\ref{thm:GeneralizedBochnerRieszEstimates} for $(\alpha,k)=(1,2)$, respectively $(\alpha,k)=(1,1)$.
The $L^p$-$L^q$-estimates are not affected by this change of representation, see
Corollary~\ref{cor:RestrictionEstimate}.  Then we use this result to prove uniform estimates for
$(P(D)+i\delta)^{-1}$ by a foliation with level sets of $P$ and
the Fourier restriction-extension theorem for the single layer.

\subsection{Parametric representation}
\label{subsection:ParametricRepresentation}
Already in \cite[p.~152]{ChoKimLeeShim2005} it was stated that a compact convex surface with curvature bounded
from below can be written locally as the graph  of an elliptic function. Moreover, it was stated that these
parametrizations do not affect Bochner-Riesz estimates. To see that this is also true in the non-elliptic
case, we explain this in a nutshell. 
%Furthermore, it is
% important to note that even in the smooth part of the surface with curvature bounded from below the surface is not necessarily convex.
%As pointed out at the end of Section \ref{section:FresnelSurface}, there is a part of the surface on the
% outer sheet ''between" the singular points and the Hamiltonian circles with principal curvatures of different
% sign. Hence, the usual Bochner-Riesz estimates for elliptic surfaces are not applicable.

\medskip

So let $M\subset\R^d$ be a compact part of a smooth regular hypersurface with
$k$ non-vanishing curvatures where $k\in\{1,\ldots,d-1\}$. After finite decompositions and rigid motions, which leave the $L^p-L^q$-estimates invariant, we find finitely many local graph representations of
$M$ of the form
\begin{equation*}
M_{loc} 
  = \{\xi=(\xi',\xi_d):\, p^{loc}(\xi) =0, \xi'\in B(0,c)\} 
  = \{(\xi', \psi(\xi')) : \, \xi' \in B(0,c) \},
\end{equation*}
where at least $k$ eigenvalues of the
Hessian matrices $\partial^2 \psi(x), x\in B(0,c)$ are bounded away from zero.
Taylor's formula gives for $\Delta:=\xi_d - \psi(\xi')$ 
\begin{equation*}
\begin{split}
p^{loc}(\xi) &= p^{loc}(\xi',\psi(\xi') + \Delta) \\
&= \int_0^1 \partial_d p^{loc}(\xi',\psi(\xi') + t \Delta) dt \cdot (\xi_d - \psi(\xi')) \\
&= m(\xi) (\xi_d - \psi(\xi')) \quad \text{ for } \xi \in B(0,c) \times (-c',c') =: B'.
\end{split}
\end{equation*}
By the properties of $p^{loc}$, we find $m \in C^\infty(B')$ with the properties
\begin{equation*}
0< c_1 \leq m \leq c_2 \;\text{ and }\; | \partial^\gamma m| \lesssim_\gamma 1 \text{ for } \gamma \in
\mathbb{N}_0^2.
\end{equation*}
The Fourier multiplier $\mathfrak{m}_\alpha$ defined by
\begin{equation*}
(\mathfrak{m}_\alpha f) \widehat (\xi) = \beta(\xi) m^{-\alpha}(\xi) \hat{f}(\xi), \; \alpha \in \R,
\end{equation*}
for a suitable cutoff $\beta \in C^\infty_c(B')$, defines a bounded mapping $L^p(\R^d)\to L^p(\R^d)$, $1\leq
p\leq \infty$ via Young's convolution inequality. Real interpolation of these estimates also yields the
boundedness $L^{p,r}(\R^d) \to L^{p,r}(\R^d)$ for $1<p<\infty,1 \leq r \leq \infty$. 
Accordingly, choosing a suitable finite partition of unity 
%and applying Theorem~\ref{thm:GeneralizedBochnerRieszEstimates} 
we find that the operators 
$$
  (\mathcal T^\alpha f) \widehat (\xi) := \frac{P(\xi)^{-\alpha}}{\Gamma(1-\alpha)} \hat f(\xi)
$$ 
are well-defined for $0<\alpha<\frac{k+2}{2}$ through analytic continuation and satisfy the same (weak)
$L^p-L^q$-estimates as the Bochner-Riesz operators that we analyzed in
Theorem~\ref{thm:GeneralizedBochnerRieszEstimates}. For $\alpha = 1$ this gives the following:
  
\begin{corollary} \label{cor:RestrictionEstimate}
  Let $K\subset\R^d$ be compact, $P\in C^\infty(K)$ such that  $\nabla P\neq 0$ on the hypersurface
  $M:=\{\xi\in K: P(\xi)=0\}$. Assume that in each point of $M$ at least $k$ principal curvatures are
  non-zero where $k\in\{1,\ldots,d-1\}$. Then, there is $t_0>0$ such that  
  $$
    \sup_{|t|<t_0} \left\|\int_{M_t} e^{ix.\xi} \hat g(\xi)\,d\sigma_t(\xi)
    \right\|_{L^{q,\infty}(\R^d)} \les \|g\|_{L^{p,1}(\R^d)}
  $$
  for $M_t:=\{\xi\in K:P(\xi)=t\}$ and $(\frac{1}{p},\frac{1}{q})\in\{B_{1,k},B_{1,k}'\}$.
  We have $(L^{p,1}(\R^d),L^q(\R^d))$-bounds for $(\frac{1}{p},\frac{1}{q})\in (B_{1,k},C_{1,k}]$,
  $(L^p(\R^d),L^{q,\infty}(\R^d))$-bounds for $(\frac{1}{p},\frac{1}{q})\in (B_{1,k}',C_{1,k}']$ and strong
  $(L^p(\R^d),L^{q}(\R^d))$-bounds for $(\frac{1}{p},\frac{1}{q})\in \mathcal P_1(k)$. 
\end{corollary}

As described at the end of Section \ref{section:FresnelSurface}, the principal curvatures of $M_t$ vary
continuously with respect to $t$ so that the curvature properties of $M_t$ for small $|t|$ are inherited from
those for $t=0$. %Indeed, $\psi_t$ varies smoothly in $t$ as $M_t$ is its graph.
%The continuous dependence is a consequence of the explicit definition of the principal curvatures as the eigenvalues of the
%Hessian of $\psi_t$ whenever $M_t$ is the graph of $\psi_t$. Since the Hessian varies continuously
%(even smoothly) with respect to $t$, its eigenvalues vary continuously. In our application $d=3$ we can deduce
%this fact from Corollary~4.2 and Theorem~4.3 by Goldman \r{TODO} where it is shown that Gaussian curvature and
%mean curvature only depend on the gradient $\nabla (P-t)=\nabla P$ and the Hessian $D^2 (P-t)=D^2P$.
The estimates leading to the proof of 
Proposition~\ref{prop:PositiveCurvatureEstimate} will result from 
an application of Corollary~\ref{cor:RestrictionEstimate}
for $d=3,K=\text{supp}(\beta_{11}),k=2$ whereas Proposition~\ref{prop:HamiltonEstimates} corresponds to the
choice $d=3,K=\text{supp}(\beta_{12}),k=1$.
To prove both results simultaneously, we therefore assume that $K\subset\R^3$ and $k\in\{1,2\}$ satisfy the
conditions of the corollary.

\subsection{Uniform estimates for the singular multiplier}
\label{subsection:UniformEstimatesSingularMultiplierI}
To prove the desired uniform resolvent estimates for $(P(\xi)+i\delta)^{-1}$, we consider
\begin{equation*}
\label{eq:Operator}
A_\delta f(x) = \int_{\R^d} \frac{\hat{f}(\xi) \beta(\xi)}{P(\xi) + i \delta} e^{ix.\xi} d\xi.
\end{equation*}
It is actually enough to show the restricted weak type bound
\begin{equation}
\label{eq:RestrictedBound}
\| A_\delta f \|_{L^{q_0,\infty}(\R^d)} \lesssim \| f \|_{L^{p_0,1}(\R^d)}
\end{equation}
for $(1/p_0,1/q_0) = (\frac{2(k+1)(k+2)}{k^2+6k+4},\frac{k}{2(k+1)})=B'$ and   
\begin{equation*}
\label{eq:RestrictedBoundII}
\| A_\delta f \|_{L^{q}(\R^d)} \lesssim \| f \|_{L^{p,1}(\R^d)}
\end{equation*}
for the remaining tuples $(1/p,1/q)\in (B',C']$ where  $C'=(1,\frac{k}{2(k+1)})$.
\begin{center}
\begin{figure}
\begin{tikzpicture}[scale=0.5]
\draw[->] (0,0) -- (13,0); \draw[->] (0,0) -- (0,13);
\draw (0,0) --(12,12); \draw(0,12) -- (12,12); \draw (12,12) -- (12,0);
\coordinate [label=above:$B$] (B) at (7,2.5);
\coordinate [label=above:$B'$] (B') at (9.5,5);
\coordinate [label=below:$C$] (C) at (7,0);
\coordinate [label=right:$C'$] (C') at (12,5);
\coordinate [label=left:$\frac{1}{2}$] (Y) at (0,6);
\coordinate [label=left:$1$] (YY) at (0,12);
\coordinate [label=below:$\frac{1}{2}$] (X) at (6,0);
\coordinate [label=below:$1$] (XX) at (12,0);
\coordinate [label=left:$\frac{1}{q}$] (YC) at (0,9);
\coordinate [label=below:$\frac{1}{p}$] (XC) at (9,0);
 
\draw [dotted] (XX) -- (YY);
\draw [help lines] (C) -- (B); \draw (B) -- (B');
\draw [help lines] (B') -- (C'); 
 
\foreach \point in {(C),(C'),(B),(B')}
\fill [black, opacity = 1] \point circle (3pt);
 
\end{tikzpicture}
\caption{All other claimed estimates result
from real interpolation with the corresponding dual estimates or with the trivial bound for
$(\frac{1}{p},\frac{1}{q})=(1,0)$.}
\end{figure}
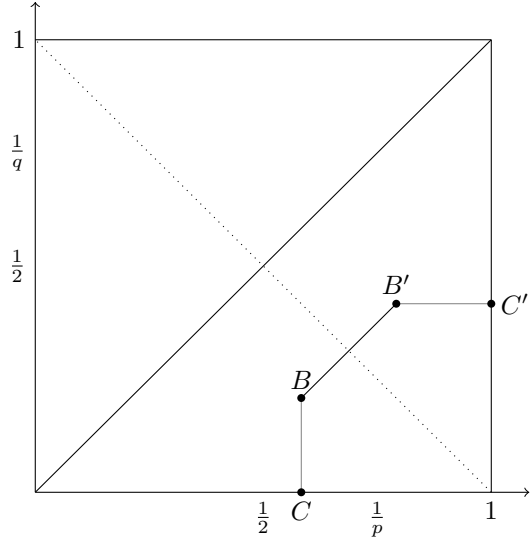
\end{center}

% We can suppose that in the support of $\beta$, we have $\nabla p(\omega,\xi) \neq 0$. We change to
%  generalized polar coordinates. Let $\xi = \xi(p,q)$, where $p$ and $q$ are complementary coordinates. Let
% $\chi \in C^\infty_c$ denote a smooth cutoff such that for $\xi \in \text{supp}(\chi)$, $\xi=\xi(p,q)$ is a
% smooth change of coordinates and suppose that $\chi = 1$ for $\text{dist}(\xi,S) < \delta'$ and $\xi \in \text{supp } \beta$.

% Write
% \begin{equation*}
% A_\delta f(x) = \int \frac{e^{ix.\xi} \chi(\xi) \beta(\xi)}{p(\omega,\xi) + i \delta} \hat{f}(\xi) d\xi + \int \frac{e^{ix.\xi} (1-\chi(\xi)) \beta(\xi) \hat{f}(\xi)}{p(\omega,\xi) + i \delta} d\xi.
% \end{equation*}
% Clearly, the Fourier multiplier in the latter expression is Schwartz and hence, by Young's inequality as in the proof of Proposition \ref{prop:GlobalEstimate}
% \begin{equation*}
% \left\| \int \frac{e^{ix.\xi} (1-\chi(\xi)) \beta(\xi) \hat{f}(\xi)}{p(\omega,\xi) + i \delta} d\xi \right\|_{L^q} \lesssim \| f \|_{L^p}.
% \end{equation*}

\medskip

We focus on \eqref{eq:RestrictedBound} in the following. To reduce our analysis to the region
$\{\xi\in K:|P(\xi)|<t_0\}$ for $t_0$ as in Corollary~\ref{cor:RestrictionEstimate}, 
we introduce a cut-off function $\chi\in C_c^\infty(\R^d)$ such that
$|P(\xi)|<t_0$ for $\chi(\xi)\neq 0$ and 
$P(\xi)>t_0/2$ for $\chi(\xi)\neq 1$. 
We then have 
\begin{equation*}
A_\delta f(x) = \int_{\R^d} e^{ix.\xi}  \frac{\chi(\xi) \beta(\xi)}{P(\xi) + i \delta} \hat{f}(\xi) \,d\xi +
\int_{\R^d} e^{ix.\xi} \frac{ (1-\chi(\xi)) \beta(\xi) }{P(\xi) + i \delta}\hat{f}(\xi)\, d\xi.
\end{equation*}
Since $P$ is smooth and bounded away from zero on $\text{supp}(1-\chi)$, the Fourier multiplier in the latter
expression is Schwartz and the claimed estimates (in fact even much stronger ones) hold for this second part.
For this reason we may from now on concentrate on the first part. We change to
generalized polar coordinates via the coarea formula:
\begin{align*}
\int_{\R^d} \frac{e^{ix.\xi}  \chi(\xi)\beta(\xi) \hat{f}(\xi)}{P(\xi) + i \delta} d\xi 
 &= (\mathfrak R(D)f)(x)+ i (\mathfrak I(D)f)(x) \\
 &= \int_{-t_0}^{t_0} 
  \frac{t}{t^2+ \delta^2}
 \left( \int_{M_t} e^{ix.\xi}\chi(\xi)\beta(\xi)|\nabla P(\xi)|^{-1} \hat{f}(\xi) \,d\sigma_t(\xi)\right) dt\\
 &\quad + i \int_{-t_0}^{t_0} 
  \frac{\delta}{t^2+ \delta^2}
 \left( \int_{M_t} e^{ix.\xi}\chi(\xi)\beta(\xi)|\nabla P(\xi)|^{-1} \hat{f}(\xi) \,d\sigma_t(\xi)\right) dt,
\end{align*}
%For $(p,q) \in \text{supp } \chi(\xi(p,q)) \cap \beta(\xi(p,q))$ we can suppose that $|\partial^\alpha
% h(p,q)| \lesssim_\alpha 1$.
where
\begin{equation*}
\frac{\chi(\xi)\beta(\xi)}{P(\xi) + i \delta} = \frac{\chi(\xi)\beta(\xi) P(\xi)}{P(\xi)^2 + \delta^2} + i
\frac{\chi(\xi)\beta(\xi) \delta}{P(\xi)^2+\delta^2} =: \mathfrak{R}(\xi) + i \mathfrak{I}(\xi).
\end{equation*}
 In the following we estimate this expression with the aid of Corollary~\ref{cor:RestrictionEstimate} 
 by decomposition in Fourier space as in~\cite[p.346]{JeongKwonLee2016}. 
 
 \medskip

 The estimate  for $\mathfrak{I}(D)$ is based on the coarea formula, Corollary~\ref{cor:RestrictionEstimate},
	 and Young's inequality in Lorentz spaces.   
\begin{equation} \label{eq:estimatesI(D)}
\begin{split}
  \|\mathfrak{I}(D)f\|_{L^{q_0,\infty}(\R^d)}
  &\les \int_{\R} \frac{\delta}{t^2 + \delta^2} \left\| \int_{M_t} e^{ix.\xi}
 \chi(\xi)\beta(\xi)|\nabla P(\xi)|^{-1} \hat{f}(\xi) d\sigma_t(\xi) \right\|_{L^{q_0,\infty}(\R^d)} \,dt \\
  &\lesssim \int_{\R} \frac{\delta}{t^2 + \delta^2}\| \mathcal
  F^{-1}(\chi(\xi)\beta(\xi)|\nabla P(\xi)|^{-1}\hat f(\xi)) \|_{L^{p_0,1}(\R^d)}\,dt \\
  &\lesssim \int_{\R} \frac{\delta}{t^2 + \delta^2} \|f \|_{L^{p_0,1}(\R^d)}\,dt \\
  &\lesssim \| f \|_{L^{p_0,1}(\R^d)}.
\end{split}
\end{equation}

\medskip

We turn to the estimate of $\mathfrak{R}(D)$, which requires an additional decomposition: Let $\phi \in
\mathcal{S}(\R)$ be such that $ \text{supp} (\hat\phi)  \subseteq [-2,-1/2] \cup [1/2,2]$ with $\tilde\phi(t):=t\phi(t)$ and
\begin{equation*}
\sum_{j=-\infty}^\infty  \tilde\phi(2^{-j} t) = 1 \qquad (t\in\R\sm\{0\}).
\end{equation*}
For the existence of $\phi$ we refer to the proof of~\cite[Lemma~2.2]{JeongKwonLee2016}, where it is denoted
by $\psi$.
%Sketch: Let $\chi \in C^\infty_c([1,2])$ with $\int_{\R} \chi dx = 1/2$. Let $\hat{\phi}(\xi) = \chi(\xi) + \chi(-\xi)$ and $\varphi(x) = \phi(x/2) - \phi(x)$. From this it is immediate that
%\begin{equation*}
%\sum_{j=-\infty}^\infty \varphi(2^{-j} x) = 1
%\end{equation*}
%whenever $x \neq 0$. Let $\psi(x) = \frac{\varphi(x)}{x}$. It remains to show that $\text{supp } \hat{\psi} \subseteq [-2,-1/2] \cup [1/2,2]$. With $\varphi \in \mathcal{S}$ and $\varphi/x \in L^1$, we have
%\begin{equation*}
%\hat{\psi}(t) = \int e^{-itx} \frac{\varphi(x)}{x} dx, \qquad \frac{d \hat{\psi}}{dt} = - i \hat{\varphi}(t).
%\end{equation*}
%Hence, 
%\begin{equation*}
%\hat{\psi}(t) = - i \int_{-\infty}^t \hat{\varphi}(s) ds = -i \int_{-\infty}^t [ 2 \chi(2s) - \chi(s) + 2 \chi(-2s) - \chi(-s)] ds
%\end{equation*}
%and it is easy to check that $\text{supp } \psi \subseteq [-2,-1/2] \cup [1/2,2]$. Hence, $\psi \in \mathcal{S}$.
 We split
\begin{align*}
A_j(\xi) &= \mathfrak{R}(\xi) \tilde{\phi}(2^{-j} P(\xi)) &&(2^j < |\delta|), \\ 
B_j(\xi) &= \left(\mathfrak{R}(\xi) - \frac{\chi(\xi)\beta(\xi)}{P(\xi)} \right) \tilde{\phi}(2^{-j}
P(\xi)) &&(2^j \geq \delta|), \\
C_j(\xi) &=  \frac{\chi(\xi)\beta(\xi)}{P(\xi)} \tilde{\phi}(2^{-j} P(\xi)) &&(2^j \geq |\delta|).
\end{align*}
The coarea formula, Minkowski's inequality, and Corollary~\ref{cor:RestrictionEstimate} yield as above
\begin{align} \label{eq:EstimatesAj}
  \begin{aligned}
&\quad \left\| \mathcal{F}^{-1} \left( \sum_{2^j < |\delta|} A_j(\xi) \hat{f}(\xi)
\right)\right\|_{L^{q_0,\infty}(\R^d)}\\
&\leq \sum_{2^j < |\delta|} \left\| \int_{-t_0}^{t_0}  \frac{t\tilde{\phi}(2^{-j} t)}{t^2 +
\delta^2} \left(\int_{M_t} e^{ix.\xi}  \chi(\xi)\beta(\xi) |\nabla P(\xi)|^{-1}
\hat{f}(\xi) \,d\sigma_t(\xi)\right)  \,dt \right\|_{L^{q_0,\infty}(\R^d)} \\
&\lesssim  \sum_{2^j < |\delta|}  \int_{\R} \frac{|t\tilde\phi(2^{-j}t)|}{t^2 + \delta^2}
\left\| \int_{M_t} e^{ix.\xi}  \chi(\xi)\beta(\xi) |\nabla P(\xi)|^{-1} \hat{f}(\xi)
\,d\sigma_t(\xi)  \right\|_{L^{q_0,\infty}(\R^d)} \,dt \\
&\lesssim  \sum_{2^j < |\delta|}  \int_{\R} \frac{2^j}{t^2 + \delta^2}  \|f\|_{L^{p_0,1}(\R^d)}
\,dt \\
&\lesssim    \int_{\R} \frac{\delta}{t^2 + \delta^2}
\,dt \, \|f\|_{L^{p_0,1}(\R^d)}  \\
 &\lesssim \|f\|_{L^{p_0,1}(\R^d)}. 
  \end{aligned}
\end{align}
Here we used the estimate $|\tilde \phi(s)|\les s^{-1}$, which holds because $\tilde\phi$ is a Schwartz
function. By similar means, we find
\begin{align} \label{eq:EstimateBj}
  \begin{aligned}
&\quad \left\| \mathcal{F}^{-1} \left( \sum_{2^j \geq |\delta|} B_j(\xi) \hat{f}(\xi)
\right)\right\|_{L^{q_0,\infty}(\R^d)}\\
&= \left\| \mathcal{F}^{-1} \left( \sum_{2^j \geq |\delta|}
\frac{\delta^2\tilde{\phi}(2^{-j}
P(\xi))}{P(\xi)(P(\xi)^2+\delta^2)} \chi(\xi)\beta(\xi) \hat{f}(\xi)
\right)\right\|_{L^{q_0,\infty}(\R^d)}\\
&\lesssim \sum_{2^j \geq |\delta|} \int_{\R} \frac{\delta^2 |\tilde{\phi}(2^{-j}t)|}{|t|(t^2
+ \delta^2)} \left\| \int_{M_t} e^{ix.\xi}  \chi(\xi)\beta(\xi) 
|\nabla P(\xi)|^{-1} \hat{f}(\xi) \,d\sigma_t(\xi) \right\|_{L^{q_0,\infty}(\R^d)} \,dt  \\
&\lesssim \sum_{2^j \geq |\delta|} \int_{\R}  \frac{\delta^2 2^{-j}}{t^2 + \delta^2} \|
f \|_{L^{p_0,1}(\R^d)} \,dt\\
&\lesssim   \int_{\R}  \frac{\delta}{t^2 + \delta^2} \|f \|_{L^{p_0,1}(\R^d)} \,dt\\
&\lesssim \| f \|_{L^{p_0,1}(\R^d)}.
\end{aligned}
\end{align}
Here, the estimate from the third to the fourth line uses
$|\tilde\phi(2^{-j}t)|
  = |\phi(2^{-j}t)| 2^{-j}t
  \les 2^{-j}t$.  
For the most involved estimate of $C_j$, we need the following lemma:
  
\begin{lemma}
\label{lem:CSum}
 Let  $\chi \in C^\infty_c(\R^d)$. Suppose $\phi \in \mathcal{S}(\R)$ with $\text{supp}
(\hat{\phi}) \subseteq [-2,-\frac{1}{2}] \cup [ \frac{1}{2}, 2]$ and that 
the level sets $\{\xi\in \text{supp}(\chi):P(\xi)=t\}$ have $k$ principal curvatures uniformly bounded from
below in modulus for all $|t|\leq t_0$.
  Then, for $1 \leq p,q \leq \infty$ with $q \geq 2$ and $\frac{1}{q} \geq \frac{k+2}{k} \big(1-\frac{1}{p}
\big)$, we find the following estimate to hold for all $\lambda>0$:
\begin{equation*}
\| \mathcal{F}^{-1} \big( \phi(\lambda^{-1} P(\xi)) \chi(\xi) \hat{f}(\xi) \big) \|_{L^{q}(\R^d)} 
\lesssim \lambda^{\frac{k+2}{2} - \frac{k+1}{q} } \| f \|_{L^{p}(\R^d)}.
\end{equation*}
\end{lemma}
\begin{proof}
 By interpolation, it suffices to prove the endpoint estimates for 
$(p,q) = (\frac{2(k+2)}{k+4},2)$ and $(p,q)=(1,\infty),(p,q)=(1,2)$.  
Since the multiplier is regular for $\lambda \geq 1$,
we may henceforth suppose $\lambda \leq 1$. For $q=2$ we use Plancherel's theorem,
the coarea formula and the $L^{\frac{2(k+2)}{k+4}}$-$L^2$ restriction-extension estimate from
Corollary~\ref{cor:RestrictionEstimate}:
\begin{align*}
&\quad \| \mathcal{F}^{-1} \big( \phi(\lambda^{-1} P(\xi)) \chi(\xi) \hat{f}(\xi) \big)\|_{L^2(\R^d)}^2 \\
&= \|  \phi(\lambda^{-1} P(\xi)) \chi(\xi) \hat{f}(\xi) \|_{L^2(\R^d)}^2 \\  
&= \int_{-t_0}^{t_0} | \phi(\lambda^{-1} t)|^2 \left(\int_{M_t} |\hat{f}(\xi)|^2
|\chi(\xi)|^2 |\nabla P(\xi)|^{-1} \,d\sigma_t(\xi) \right) \,dt \\
&\lesssim  \int_{-t_0}^{t_0} | \phi(\lambda^{-1} t)|^2
\|f\|_{L^{\frac{2(k+2)}{k+4}}(\R^d)}^2 \,dt \\
&\lesssim \lambda \| f \|^2_{L^{\frac{2(k+2)}{k+4}}(\R^d)}.
\end{align*}
Using the trivial estimate $|\hat f(\xi)|\leq \|f\|_{L^1(\R^d)}$ instead (from the third to the fourth line),
we find the endpoint estimate for $(p,q)=(1,2)$.

\medskip

For the endpoint $(p,q) = (1,\infty)$ it suffices to show the kernel estimate
\begin{equation*}
|K(x)| \lesssim \lambda^{\frac{k+2}{2}}
\end{equation*}
for
\begin{align*}
K(x)
 &:=  \mathcal{F}^{-1} \big( \phi(\lambda^{-1} P(\xi)) \chi(\xi)\big)(x) \\
 &= \frac{1}{(2\pi)^{3/2}} \int_{\R^d} e^{ix.\xi} \phi(\lambda^{-1} P(\xi)) \chi(\xi) \,d\xi \\   
 &= \frac{1}{2\pi} \int_{\R^d} e^{ix.\xi} \chi(\xi) \int_{\frac{1}{2}}^2 e^{ir \lambda^{-1} P(\xi)}
 \hat{\phi}(r)\,dr\,d\xi \\
  &= \frac{1}{2\pi} \int_{\frac{1}{2}}^2 \hat{\phi}(r) \big( \int_{-t_0}^{t_0} 
  e^{ir\lambda^{-1}t} \underbrace{\big( \int_{M_t} e^{ix.\xi}
  \chi(\xi)|\nabla P(\xi)|^{-1} \,d\sigma_t(\xi)\big)}_{=:a(t,x)} \,dt \big) \,dr  
\end{align*}
The function $a$ is smooth, all its derivatives are bounded functions and its support is bounded
with respect to $t$. So the principle of non-stationary phase yields for $|x| \ll \lambda^{-1}$ and all
$M\in\N$
\begin{align*}
  |K(x)| 
  \les_M \int_{\frac{1}{2}}^2 |\hat{\phi}(r)| |r\lambda^{-1}|^{-M}   \,dr 
  \les_M \lambda^M.
\end{align*}
In particular, this holds for $M=\frac{k+2}{2}$.
For $|x| \gtrsim \lambda^{-1}$ we can use the dispersive estimate  $|a(t,x)|\les (1+|x|)^{-k/2}$, which holds
due to method of stationary phase and the presence of $k$ non-vanishing principal curvatures. We thus get for
$|x| \gtrsim \lambda^{-1}$
\begin{align*}
|K(x)| 
&= \frac{1}{2\pi} \left| \int_{-t_0}^{t_0} \phi(\lambda^{-1}t)  a(t,x) \,dt  \right| \\ 
&\lesssim  \int_{-t_0}^{t_0} |\phi(\lambda^{-1}t)| (1+|x|)^{-\frac{k}{2}} \,dt \\ 
&\lesssim \lambda (1+|x|)^{-\frac{k}{2}}  \\
&\lesssim \lambda^{\frac{k+2}{2}}.
\end{align*}
The proof is complete. 
\end{proof}

The lemma allows to bound the $C_j$-terms as follows:
\begin{align*} 
\| C_j(D) f \|_{L^\sigma(\R^d)} 
 &=  \| \mathcal F^{-1}\left(\frac{\chi(\xi)\beta(\xi)}{P(\xi)}  \tilde{\phi}(2^{-j} P(\xi)) \hat
  f(\xi)\right)\|_{L^\sigma(\R^d)} \\
 &=  2^{-j} \| \mathcal F^{-1}\left(\chi(\xi)\beta(\xi) \phi(2^{-j} P(\xi)) \hat f(\xi) \right)\|_{L^\sigma(\R^d)} \\ 
 & \lesssim 2^{j \big( \frac{k}{2} - \frac{k+1}{\sigma} \big)} \| f \|_{L^r(\R^d)}
\end{align*}
for $2 \leq \sigma \leq \infty$, $\frac{1}{\sigma} \geq \frac{k+2}{k}\big( 1- \frac{1}{r} \big)$.
Using Lemma~\ref{lem:SummationLemma}, ~\eqref{eq:SummationI} for $q_1,q_2,p_1,p_2$ defined as
$$
  \frac{k}{2}-\frac{k+1}{q_1} = \eps,\quad
  \frac{k}{2}-\frac{k+1}{q_2} = -\eps,\quad
  \frac{1}{q_i} =: \frac{k+2}{k}\big( 1- \frac{1}{p_i} \big)
$$  
for small $\eps>0$, we finally get due to $\frac{1}{q_0}=\frac{k}{2(k+1)}=\frac{1}{2q_1}+\frac{1}{2q_2} =
\frac{k+2}{k}\big( 1- \frac{1}{p_0}\big)$
\begin{align} \label{eq:EstimatesCj}
   \left\| \mathcal{F}^{-1} \left( \sum_{2^j \geq  |\delta|} C_j(\xi) \hat{f}(\xi)
   \right)\right\|_{L^{q_0,\infty}(\R^d)}
   %= \left\|  \sum_{2^j \geq  |\delta|} C_j(D)f \right\|_{L^{q_0,\infty}(\R^d)}
   \les \|f\|_{L^{p_0,1}(\R^d)}. 
\end{align}
 
\medskip

Combining the
estimates~\eqref{eq:estimatesI(D)}-\eqref{eq:EstimatesCj}, 
we get the claimed estimate 
$$
  \| A_\delta f \|_{L^{q_0,\infty}(\R^d)} \lesssim \| f \|_{L^{p_0,1}(\R^d)}.
$$
This proves Proposition~\ref{prop:PositiveCurvatureEstimate} ($k=2$) and
Proposition~\ref{prop:HamiltonEstimates} ($k=1$). \hfill $\Box$

\subsection{An improved Fourier restriction--extension estimate for the Fresnel surface close to Hamiltonian
circles}

The purpose of this section is to point out how the special degeneracy along the Hamiltonian circles might
allow for improved estimates in Proposition~\ref{prop:HamiltonEstimates}. In our proof in the previous
section we exploited that one principal curvature is bounded away from zero close to these circles. But actually we have
more: The other principal curvature does not vanish identically in that region, but only vanishes at the
Hamiltonian circle, which is a curve on the Fresnel surface. We refer to Figure~\ref{fig:Fresnel} for an
illustration of the situation. 

%In the previous section we have seen how decay of the Fourier transform of the surface
%measure yields Proposition \ref{prop:HamiltonEstimates} when we have one principle curvature bounded
%from below. 
For surfaces with vanishing Gaussian curvature, but no flat points, improved results were
established in special cases. For in a sense generic surfaces in $\R^3$ with Gaussian curvature vanishing along a
one-dimensional sub-manifold, the decay
\begin{equation*}
\label{eq:ImprovedDecay}
|\hat{\mu}(\xi)| \lesssim \langle \xi \rangle^{-\frac{3}{4}}
\end{equation*}
was shown by Erd\H{o}s--Salmhofer \cite{ErdosSalmhofer2007}. 
(In our proof we used Bochner-Riesz estimates resulting from the weaker
bound $|\hat{\mu}(\xi)| \lesssim \langle \xi \rangle^{-\frac{1}{2}}$, which
is~\eqref{eq:DecayFourierTransform} for $k=1$.) We shall show the corresponding $L^p$--$L^q$ estimates for
these surfaces in future work. However, these results are not applicable in our case. Indeed, one can still
show that the gradient of the curvature $\nabla_{\Sigma} K(p) \neq 0$ does not vanish along the Hamiltonian
circles (cf. \cite[Assumption~2]{ErdosSalmhofer2007}) and any unit vector has only finitely many preimages
under the normal $\nu: \Sigma \to \mathbb{S}^{2}$ (cf. \cite[Assumption~3]{ErdosSalmhofer2007}).
It turns out that the transversality assumption \cite[Assumption~4]{ErdosSalmhofer2007} regarding the
Hamiltonian circle and the direction of the non-vanishing principal curvature fails: On the curves $\Gamma =
\{ K = 0 \}$, exactly one of the principal curvatures vanishes. We define a (local) unit vectorfield $Z \in T
\Sigma$ along $\Gamma$ in the tangent plane of $\Sigma$. On $\Gamma$, $Z$ is supposed to point into the direction of the vanishing principal curvature, and $Z$ can be extended to a neighbourhood of $\Gamma$
as the direction of the principal curvature that is small and vanishes on $\Gamma$. To apply the arguments
from \cite{ErdosSalmhofer2007}, it is required that $Z$ is transversal to $\Gamma$ up to finitely many
points, and the angle between $Z$ and $\Gamma$ increases linearly.

 But along the Hamiltonian circles, we find, assuming w.l.o.g. $\eps_1<\eps_2<\eps_3$,
\begin{equation*}
\alpha(s,t) = \varepsilon_2.
\end{equation*}
This allows to solve for $t=t(s)$ (cf. \eqref{eq:ParametrizationHamiltonianCircles}):
\begin{equation*}
t = \frac{\varepsilon_1 \varepsilon_3 (\varepsilon_2 -s) }{s^2 - (\varepsilon_1 + \varepsilon_2 + \varepsilon_3) s + (\varepsilon_1 \varepsilon_2 + \varepsilon_1 \varepsilon_3 + \varepsilon_2 \varepsilon_3) - \varepsilon_1 \varepsilon_3} = \frac{\varepsilon_1 \varepsilon_3}{\varepsilon_1 + \varepsilon_3 -s}.
\end{equation*}
and further,
\begin{equation*}
\frac{\partial t}{\partial s} = \frac{\varepsilon_1 \varepsilon_3}{(\varepsilon_1 + \varepsilon_3 -s)^2}.
\end{equation*}
In the following we shall see that the direction of the vanishing principal curvature is tangential to the
Hamiltonian circles. This violates the transversality assumption.\\
For this purpose, consider $\xi = \xi(s,t)$ with $t=\frac{\varepsilon_1 \varepsilon_3}{\varepsilon_1 +
\varepsilon_3 -s}$. This yields a parametrization of the Hamiltonian circles. For a tangent vector we find
\begin{equation*}
\frac{d \xi}{d s} = \frac{\partial \xi}{\partial s} + \frac{\partial \xi}{\partial t} \frac{\partial t}{\partial s} = e_s + \frac{\partial t}{\partial s} e_t.
\end{equation*}
A straight-forward computation shows that this is in the kernel of the second fundamental form II, which was computed in $(s,t)$ coordinates in Proposition \ref{prop:SecondFundamentalForm}:
\begin{equation*}
\text{II} (1, \frac{\partial t}{\partial s}) = 0.
\end{equation*}
Note that
\begin{align*}
t-\varepsilon_1 = \frac{t (s-\varepsilon_1)}{\varepsilon_3}, \quad t- \varepsilon_3 = \frac{t(s-\varepsilon_3)}{\varepsilon_1}.
\end{align*}
Consequently,
\begin{equation*}
\frac{t^2}{\varepsilon_1 \varepsilon_3} (s-\varepsilon_1) (s-\varepsilon_2) (s- \varepsilon_3) = (t-\varepsilon_1) (s-\varepsilon_2) (t-\varepsilon_3).
\end{equation*}
From this follows
\begin{equation*}
\frac{P_L}{(s-\varepsilon_1)(s-\varepsilon_2)(s-\varepsilon_3)} = - \frac{t^2}{\varepsilon_1 \varepsilon_3},
\end{equation*}
and thus,
\begin{equation*}
\frac{\partial t}{\partial s} + \frac{P_L(s,t)}{(s-\varepsilon_1)(s-\varepsilon_2)(s-\varepsilon_3)} = 0,
\end{equation*}
with $P_L$ defined in Section \ref{section:FresnelSurface}.\\
Furthermore,
\begin{equation*}
2t-s = t \big( 1 + \frac{(s-\varepsilon_1)(s-\varepsilon_3) }{\varepsilon_1 \varepsilon_3} \big),
\end{equation*}
and therefore, by plugging the definition of $\alpha$ into $P_N$,
\begin{equation*}
\frac{P_N}{t} = (s-t) \varepsilon_1 \varepsilon_3 + (s-\varepsilon_1) (s-\varepsilon_3) \varepsilon_2 = -(s-\varepsilon_1) (t-\varepsilon_2) (s-\varepsilon_3).
\end{equation*}
Thus,
\begin{equation*}
\frac{P_N}{t(t-\varepsilon_1)(t-\varepsilon_2)(t-\varepsilon_3)} = - \frac{(s-\varepsilon_1)(s-\varepsilon_3)}{(t-\varepsilon_1)(t-\varepsilon_3)} = - \frac{\varepsilon_1 \varepsilon_3}{t^2}.
\end{equation*}

Still, there is hope that one can show better decay
\begin{equation*}
|\hat{\mu}(\xi)| \lesssim \langle \xi \rangle^{-(\frac{1}{2} + \delta)}
\end{equation*}
for some $\delta > 0$ using stationary phase estimates for functions with degenerate
Hessian as in Ikromov--M\"uller~\cite[Corollary~1.6]{IkromovMueller2011} applied by
Greenblatt~\cite{Greenblatt2016}; see also~\cite{Varchenko1976, PhongStein1997, Greenblatt2009,
IkromovKempeMueller2010, IkromovMueller2016} and references therein.  Since the singular points of our
Fresnel surface (to be discussed in the following section) give rise to the worse total decay 
$|\hat{\mu}(\xi)| \lesssim \langle \xi \rangle^{-\frac{1}{2}}$ of the Fourier
transform, the analysis is not detailed here.

\section{Estimates for neighbourhoods of the singular points}
\label{section:SingularPoints}
The purpose of this section is to prove the estimate
\begin{equation*}
\| \beta_{13}(D) (E,H) \|_{L^q(\R^3)} \lesssim \| \beta_{13}(D) (J_e,J_m) \|_{L^p(\R^3)}
\end{equation*}
with $\beta_{13}$ defined in Section \ref{section:Reductions} as smooth cutoff localizing to a neighbourhood
of the singular points. We shall also take the opportunity to derive estimates for perturbed cone multipliers
in $\R^d$. These naturally arise for surfaces $S= \{ \xi \in \R^d : p(\xi) = 0 \}$ at singular points $\xi \in
S$ with $\nabla p(\xi) = 0$, and $\partial^2 p$ with signature $(1,d-1)$. \\
%The arguments generalize to higher dimensions and seem to be of relevance when analyzing Fourier multipliers
%associated with hypersurfaces containing singular, non-degenerate points. \\
In the first step, to clarify the nature of $S$, we shall change to parametric representation in Section
\ref{subsection:ParametricRepresentationSingularPoints}. We will see that it suffices to analyze two perturbed half-cones
\begin{equation*}
\{ \xi_d = \pm |\xi'| + O(|\xi'|^2) \}, \quad i=1,2.
\end{equation*}
This yields that for a small, but fixed distance from the origin, we have the curvature properties of the
cone and can apply Theorem \ref{thm:GeneralizedBochnerRieszEstimates} with $\alpha =1$, $k=d-2$ to derive
Fourier restriction-extension estimates for the layers. Then, the arguments of Section
\ref{subsection:UniformEstimatesSingularMultiplierI} apply again.
We derive the estimates for the generalized cone multiplier and \eqref{eq:UniformEstimateSingular} by an
additional Littlewood-Paley decomposition and a scaling argument in
Subsection~\ref{subsection:PerturbedConeMultiplier}.

\medskip

Coming back to Fresnel's surface, we first prove that $S$ looks like a cone around the singular points. We
recall that we assumed without loss of generality $\mu_1=\mu_2=\mu_3=\omega=1$ so that the results from
Section~\ref{section:FresnelSurface} apply for $S=S^*$.
\begin{proposition}
  Set $\zeta\in S$ be one of the four singular points given by Proposition~\ref{prop:singularpoints}. Then 
  $$
    p(\omega,\xi) =  \frac{1}{2}(\xi-\zeta)^T  D^2p(\omega,\zeta)(\xi-\zeta) +
    O(|\xi-\zeta|^3) \quad\text{as }\xi\to\zeta
  $$
  and $D^2p(\omega,\zeta)$ has two positive and one negative eigenvalue.
\end{proposition}
\begin{proof}
  By Taylor's theorem, $p(\omega,\zeta)=0$ (because $\zeta\in S$), and $\nabla p(\omega,\zeta)=0$ (because
  $\zeta$ is singular), it suffices to prove that $D^2p(\omega,\zeta)$ has two positive  and one
  negative eigenvalue.
   For notational convenience we assume $\eps_1<\eps_2<\eps_3$ and concentrate on the singular point 
  $\zeta=(\zeta_1,\zeta_3,\zeta_3)\in S$ given by 
  $$
    \zeta_1 = \sqrt{\frac{\varepsilon_3 (\varepsilon_1-\varepsilon_2)}{\varepsilon_1 -
    \varepsilon_3}}, 
    \qquad \zeta_2=0,
    \qquad 
    \zeta_3 = \sqrt{\frac{\varepsilon_1(\varepsilon_3 -\varepsilon_2)}{\varepsilon_3
    -\varepsilon_1}},
  $$
  Then we find  
  $$
    D^2p(\omega,\zeta) = \matIII{D_{11}}{0}{D_{13}}{0}{D_{22}}{0}{D_{13}}{0}{D_{33}},
$$
where (cf. \cite[pp.~74-75]{Liess1991})
$$
  D_{22} 
  = \frac{2}{\eps_1}+\frac{2}{\eps_3}-\frac{2(\eps_1+\eps_2)}{\eps_1\eps_2\eps_3}\zeta_1^2
  -\frac{2(\eps_2+\eps_3)}{\eps_1\eps_2\eps_3}\zeta_3^2 
  = \frac{2}{\eps_1\eps_2\eps_3} (\eps_1 - \eps_2)( \eps_2- \eps_3 )>0 
$$
and
\begin{align*}
  D_{11}
  &= \frac{2}{\eps_2}+\frac{2}{\eps_3}-\frac{12}{\eps_2\eps_3}\zeta_1^2 -
  \frac{2(\eps_1+\eps_3)}{\eps_1\eps_2\eps_3}\zeta_3^2
  =  -\frac{8(\eps_2 - \eps_1) }{\eps_2(\eps_3 - \eps_1)}<0,  \\ 
  D_{33} 
  &= -\frac{8(\eps_2 - \eps_3)}{\eps_2(\eps_1- \eps_3)},   \\
  D_{13}
  &=  -\frac{4(\eps_1+\eps_3)}{\eps_1\eps_2\eps_3} \zeta_1\zeta_3
   =  -\frac{4(\eps_1+\eps_3)\sqrt{(\eps_2-\eps_1)(\eps_3-\eps_2)}}{\eps_2\sqrt{\eps_1\eps_3}(\eps_3-\eps_1)},
   \\
  D_{11}D_{33} - D_{13}^2
  &= - \frac{16}{\eps_1\eps_2^2\eps_3}
  (\eps_2 - \eps_1)( \eps_3 - \eps_2) < 0.
\end{align*}
  So the symmetric $2\times 2$-submatrix with entries $D_{11},D_{13},D_{13},D_{33}$
  is indefinite and hence posseses one positive and one negative eigenvalue. This yields the claim.
\end{proof}

Accordingly, after suitable rotations, translations and multiplication by $-1$, we may suppose in the
following that the analyzed singular point lies in the origin and that the Taylor expansion of the Fourier
symbol around the singular point is given by
\begin{equation*}
  \tilde{p}(\xi) =  \xi_3^2 -  |\xi'|^2 + g(\xi), \quad |\partial^\alpha g(\xi)| \les_\alpha
  |\xi|^{3-|\alpha|} \quad (\alpha\in\N_0^3).
\end{equation*}
We will discuss the corresponding Fourier multiplier given by
\begin{equation*}
A_\delta f(x) =  \int_{\R^3} \frac{e^{ix.\xi} \beta(\xi) \hat{f}(\xi) }{\tilde{p}(\xi) + i \delta} d\xi,
\end{equation*}
where $\beta \in C^\infty_c(\R^3)$. The support of $\beta$ will later be assumed to be close to zero so that
the mapping properties of $A_\delta$ are determined by the Taylor expansion of $\tilde p$ around zero.
%Due to $\varepsilon$-regularization, we can suppose that $\text{dist}(\text{supp} \beta, 0) \gtrsim \varepsilon$ and $\text{supp}(\beta) \subseteq B(0,c)$, $0<c \ll 1$. 
The aim is to show estimates
\begin{equation}
\label{eq:UniformEstimateSingular}
\| A_\delta f \|_{L^q(\R^3)} \lesssim \| f \|_{L^p(\R^3)}
\end{equation} 
for $p$,$q$ as in Proposition \ref{prop:SingularEstimate} as previously independent of $\delta$. 
This will be proved in Subsection~\ref{subsection:ApproximateSolutionsSingularPoints}.

\subsection{Parametric representation around the singular points}
\label{subsection:ParametricRepresentationSingularPoints} 
In this subsection we change to a parametric representation.  This requires additional arguments as $\tilde p$
vanishes of second order at the origin. We find the following:

\begin{proposition}\label{prop:ParametrizedCone}
Let $\tilde{p}: \R^d \to \R$ be a smooth function with $\tilde{p}(0) = 0$ and $\nabla \tilde{p}(0)=0$,
$\partial^2 \tilde{p} = \text{diag}(-1,\ldots,-1,1)$. Then, there is $c>0$ such that  
\begin{equation}
\label{eq:Factorization}
\tilde{p}(\xi) = (\xi_d - |\xi'|+r_1(\xi'))(\xi_d + |\xi'|+r_2(\xi')) m(\xi) \text{ for } \xi = (\xi',\xi_d)\in B(0,c)
\end{equation}
with $m \in C^\infty(\R^d \backslash \{0\})$, $|m| \gtrsim 1$, $|\partial^\alpha m(\xi)| \lesssim_\alpha
|\xi|^{-|\alpha|}$ for $\alpha \in \mathbb{N}_0^d$ and $r_i \in C^\infty(\R^{d-1} \backslash \{0\})$, 
$|\partial^\alpha r_i(\xi')| \lesssim |\xi'|^{2-|\alpha|}$. 
\end{proposition}
\begin{proof}
From the Taylor expansion we get $\tilde p(\xi)=\xi_d^2-|\xi'|^2+g(\xi)$ with 
$|\partial^\alpha g(\xi)| \leq C_\alpha |\xi|^{3-|\alpha|}$ for
$\alpha\in\N_0^d$.  Choose $c:=\min\{\frac{1}{10C_0},\frac{1}{5C_1}\}$  and we consider $|\xi|\leq c$ from
now on.
In the first step, we find zeros for fixed $\xi'$ by monotonicity with respect to $\xi_d$. 
For $|\xi_d| \leq \frac{|\xi'|}{2}$, we find $\xi_d^2 - |\xi'|^2 \leq -\frac{3|\xi'|^2}{4} \leq
-\frac{3|\xi|^2}{5}$ and hence  $\tilde{p}(\xi) \leq -\frac{|\xi|^2}{2}$. 
In the latter estimate we used $|\xi|\leq c\leq \frac{1}{10C_0}$. 
Similarly, we find 
$\xi_d^2 - |\xi'|^2 \geq |\xi'|^2 \geq \frac{3|\xi|^2}{5}$ and hence 
$\tilde{p}(\xi)
\geq \frac{|\xi|^2}{2}$ whenever $|\xi_d| \geq 2|\xi'|$. Both inequalities together imply that for $|\xi|\leq
c$ we have $\tilde{p}(0,\xi_d)=0$ if and only if $\xi_d=0$ as well as 
\begin{equation*}
\tilde{p}(\xi',z_1) < 0 < \tilde{p}(\xi',z_2) \quad\text{if }\xi'\neq 0,\;|z_1|\leq \frac{|\xi'|}{2},\;
|z_2|\geq 2|\xi'|.
\end{equation*} 
Furthermore, 
$$
   |\partial_d\tilde{p}(\xi',\xi_d)| 
   \geq 2|\xi_d|  - C_1|\xi|^2
   \geq |\xi_d|
   \qquad\text{if } \frac{|\xi'|}{2}\leq |\xi_d|\leq 2|\xi'|.
$$ 
In the last inequality we used $|\xi_d|\leq |\xi|\leq c\leq \frac{1}{5C_1}$.
Hence, by strict monotonicity, all solutions of $\tilde p(\xi)=0$ in $B_c(0)$ are given by
$\xi_d=\psi_1(\xi')$ or $\xi_d=-\psi_2(\xi')$ for positive functions $\psi_1,\psi_2$ that, by the implicit function theorem, are
even smooth away from the origin. Taking the gradients on each part of the equations $\tilde
p(\xi',\psi_1(\xi'))=0$ and $\tilde p(\xi',-\psi_2(\xi'))=0$  we find the claimed properties
\begin{equation*}
  \psi_1(\xi')=|\xi'|-r_1(\xi'),\quad
  \psi_2(\xi')=|\xi'|+r_2(\xi')\quad\text{with }
| \partial^\alpha r_i(\xi')| \lesssim_\alpha |\xi'|^{2-|\alpha|}.
\end{equation*}
 It remains to check the validity of \eqref{eq:Factorization}. This is straight-forward for
 $|\xi_d|\leq \frac{|\xi'|}{2}$ or  $|\xi_d|\geq 2|\xi'|$ where the factor 
  $(\xi_d - |\xi'|+r_1(\xi'))(\xi_d + |\xi'|+r_2(\xi'))$ does not vanish.   In the case $\frac{|\xi'|}{2}\leq
  |\xi_d| \leq  2|\xi'|$ and $\xi_d>0$ we obtain by the same arguments as in Subsection
  \ref{subsection:ParametricRepresentation}  
\begin{equation*}
\frac{\tilde{p}(\xi)}{(\xi_d - \psi_1(\xi'))(\xi_d + \psi_2(\xi'))} 
= \frac{\int_0^1 \partial_d \tilde{p}(\xi',\psi_1(\xi') + t(\xi_d -\psi_1(\xi')) dt}{\xi_d + \psi_2(\xi')}
\end{equation*}
 with $\int_0^1 \partial_d \tilde{p}(\xi',\psi_1(\xi') + t(\xi_d
-\psi_1(\xi'))\,dt = O(|\xi_d|)$, and the claim follows by $|\xi_d - \psi_2(\xi')| \gtrsim |\xi|$. The claim
for the derivatives follows from the above display by induction.
 The case $\xi_d < 0$ is treated analogously. 
\end{proof}

\subsection{Estimates for perturbed cone multiplier}
\label{subsection:PerturbedConeMultiplier}
With $(\mathfrak{m}_\alpha f) \widehat (\xi) = m^{-\alpha}(\xi) \hat{f}(\xi)$ a Fourier multiplier in $L^p(\R^d)$ for $1 < p < \infty$ by Mikhlin's theorem, the above parametric representation suggests to analyze the generalized cone multiplier
\begin{equation*}
\label{eq:ConeMultiplier}
(\mathcal C^\alpha f) \widehat (\xi) = \frac{1}{\Gamma(1-\alpha)} \frac{\beta(\xi) \hat{f}(\xi)}{((\xi_d -
|\xi'| + r_1(\xi'))(\xi_d + |\xi'| + r_2(\xi'))^{\alpha}_+},
\end{equation*}
which is again defined by analytic continuation for $\alpha \geq 1$. As provided in
Subsection~\ref{subsection:ParametricRepresentationSingularPoints} for singular non-degenerate points, we suppose that
\begin{equation*}
r_i \in C^\infty(\R^{d-1} \backslash \{0\}), \quad |\partial^\alpha r_i(\xi')| \lesssim_\alpha
|\xi'|^{2-|\alpha|} \qquad (i=1,2, \alpha\in\N_0^d)
\end{equation*}
  and $\beta \in C^\infty_c(B(0,c))$ satisfies $\beta(\xi) = 1$ for $|\xi| \leq \frac{c}{2}$ for $c$ as in
  Proposition~\ref{prop:ParametrizedCone}.  
   We suppose that $c=1$ to lighten the notation.
The aim of this section is to show that $\mathcal C^\alpha: L^p(\R^d) \to L^q(\R^d)$ is bounded for exponents
$p,q$ as described below.
To explain $\mathcal C^{\alpha}$ \emph{a priori} in the distributional sense, we suppose that $f \in
\mathcal{S}$ with $0 \notin \text{supp} (\hat{f})$. As we prove estimates independent of the Fourier support, $\mathcal C^\alpha$ extends by density.
% 
%  We consider estimates
% \begin{equation}
% \label{eq:StrongEstimatesGeneralizedConeMultiplier}
% \| \mathcal C^\alpha f \|_{L^q(\R^d)} \lesssim \| f \|_{L^p(\R^d)},
% \end{equation}
% also weak estimates
% \begin{align}
% \label{eq:WeakTypeBoundGenCone}
% \| \mathcal C^\alpha f \|_{L^q(\R^d)} &\lesssim \| f \|_{L^{p,1}(\R^d)}, \\
% \label{eq:WeakTypeBoundIIGenCone}
% \| \mathcal C^\alpha f \|_{L^{q,\infty}(\R^d)} &\lesssim \| f \|_{L^p(\R^d)},
% \end{align}
% and restricted weak bounds
% \begin{equation}
% \label{eq:WeakTypeBoundIIIGenCone}
% \| \mathcal C^\alpha f \|_{L^{q,\infty}(\R^d)} \lesssim \| f \|_{L^{p,1}(\R^d)}.
% \end{equation}

\begin{proposition}
\label{prop:GeneralizedConeMultiplier}
Let $1/2 < \alpha < d/2$. Then $\mathcal C^\alpha$ has the same mapping properties as the Bochner-Riesz
operator $T^\alpha$ from Theorem~\ref{thm:GeneralizedBochnerRieszEstimates}~(i) for $k=d-2$.
%  and $1 \leq p \leq q \leq \infty$. Then, we find
% \eqref{eq:StrongEstimatesGeneralizedConeMultiplier} to hold for
% \begin{equation*}
% %\frac{1}{p} > \frac{1}{4}+ \frac{\alpha}{2}, \qquad \frac{1}{q} < \frac{3}{4}- \frac{\alpha}{2}, \qquad
% % \frac{1}{p} - \frac{1}{q} \geq \frac{2\alpha}{3}.
%   \frac{d-2+2\alpha}{2(d-1)}, \; y < \frac{d-2\alpha}{2(d-1)}, \; x - y \geq \frac{2\alpha}{d} 
% \end{equation*}
% Furthermore, let
% \begin{align*}
% B &= \big( \frac{d-2+2\alpha}{2(d-1)}, \frac{1}{4}-\frac{\alpha}{6} \big), &C  = \big(
% \frac{1}{4}+\frac{\alpha}{2}, 0 \big), \\
% B' &= \left( \frac{3}{4} + \frac{\alpha}{6}, \frac{3}{4}-\frac{\alpha}{2} \right), &C' = (1, \frac{3}{4}-\frac{\alpha}{2}).
% \end{align*}
% We find estimates \eqref{eq:WeakTypeBoundGenCone} to hold for $(\frac{1}{p},\frac{1}{q}) \in (B, C)$, \eqref{eq:WeakTypeBoundIIGenCone} for $(\frac{1}{p}, \frac{1}{q}) \in (B',C')$, and \eqref{eq:WeakTypeBoundIIIGenCone} for $(\frac{1}{p},\frac{1}{q}) \in \{ B, B' \}$.\\
\end{proposition}

The proposition generalizes Lee's result~\cite[Theorem~1.1]{Lee2003} for $\alpha > 1/2$: Fourier supports and
perturbations of the cone including the singular point are covered and the space dimension is not restricted
to $d=3$. As we obtain the same conditions on $(p,q)$ as Lee, which he showed to be sharp in the case $d=3$,
the conditions in Proposition~\ref{prop:GeneralizedConeMultiplier} are clearly sharp. It seems likely that by
bilinear restriction the result can be improved as in~\cite{Lee2003} for $\alpha < \frac{1}{2}$.

To reduce the estimates to Theorem \ref{thm:GeneralizedBochnerRieszEstimates}, we apply a Littlewood-Paley decomposition. Let 
$\beta_l(\xi) = \beta_0(2^l \xi)$ with $\text{supp} (\beta_0) \subseteq B(0,2) \backslash B(0,1/2)$ and
\begin{equation*}
\sum_{l \geq 0} \beta_l \cdot \beta = \beta.
\end{equation*}
We define
\begin{equation*}
(\mathcal C_l^{\alpha} f) \widehat(\xi)= \frac{1}{\Gamma(1-\alpha)}\frac{ \beta_l(\xi) \beta(\xi) 
\hat{f}(\xi)}{((\xi_d - |\xi'| + r_1(\xi'))(\xi_d + |\xi'| + r_2(\xi'))^{\alpha}_+}.
\end{equation*}

We have the following consequence of Littlewood-Paley theory:
\begin{lemma}
\label{lem:LittlewoodPaley}
Assume that there are $1< p < 2< q < \infty$ and $r_1 \in \{1,p \}$ and $r_2 \in \{q,\infty\}$ such that 
\begin{equation*}
\label{eq:DyadicEstimate}
\| C_l^\alpha f \|_{L^{q,r_2}(\R^d)} \leq C \| f \|_{L^{p,r_1}(\R^d)}
\end{equation*}
holds for all $l\in\N_0$. Then 
\begin{equation*}
\| \mathcal C^\alpha f \|_{L^{q,r_2}(\R^d)} \lesssim C \| f \|_{L^{p,r_1}(\R^d)}.
\end{equation*}
\end{lemma}
\begin{proof}
We write by the square function estimate, which also holds in Lorentz spaces, see, e.g. \cite[Lemma~3.2]{JeongKwonLee2016}, and Minkowski's inequality (note that $L^{\frac{q}{2},\infty}$ is normable because $q>2$)
\begin{align*}
\| \mathcal C^\alpha f \|_{L^{q,r_2}(\R^d)} 
\lesssim \big\| \big( \sum_{l \geq 0} |\mathcal C^\alpha_l f |^2 \big)^{\frac{1}{2}} \big\|_{L^{q,r_2}(\R^d)} 
%&\r{\lesssim \big\| \big( \sum_{k \geq 0} |\mathcal C^\alpha_k f |^q \big)^{\frac{1}{q}}
%\big\|_{L^{q,r_2}(\R^d)}} \\
\lesssim \big( \sum_{l \geq 0} \| \mathcal C_l^\alpha f \|_{L^{q,r_2}(\R^d)}^2 \big)^{\frac{1}{2}}.
\end{align*}
By hypothesis and noting that $\mathcal C_l^\alpha f = \mathcal C_l^\alpha \big( \sum_{|l'-l| \leq 2}
\beta_{l'}(D) f \big)$, we find
\begin{align*}
\big( \sum_{l \geq 0} \| \mathcal C_l^\alpha f \|^2_{L^{q,r_2}(\R^d)} \big)^{\frac{1}{2}} 
\lesssim \big( \sum_{l' \geq 0} \| \beta_{l'}(D) f \|^2_{L^{p,r_1}(\R^d)} \big)^{\frac{1}{2}} 
\lesssim  \| f \|_{L^{p,r_1}(\R^d)}.
\end{align*}
Notice that the ultimate estimate is dual to the previous display.
\end{proof}
We are ready for the proof of Proposition \ref{prop:GeneralizedConeMultiplier}.
\begin{proof}[Proof of Proposition \ref{prop:GeneralizedConeMultiplier}]
We use scaling to reduce to unit frequencies:
\begin{align*}
\mathcal C_l^\alpha f(x) 
&= \frac{1}{\Gamma(1-\alpha)} \int_{\R^d}\frac{e^{ix.\xi} \beta_l(\xi) \hat{f}(\xi)}{((\xi_d -|\xi'| +
r_1(\xi'))(\xi_d+|\xi'| +r_2(\xi'))^\alpha_+} \,d\xi \\
&= \frac{2^{-dl}}{\Gamma(1-\alpha)} \int_{\R^d}\frac{e^{ix.2^{-l} \zeta} \beta_0(\zeta) \hat{f}(2^{-l}
\zeta)}{((2^{-l} \zeta_d - 2^{-l} |\zeta'| + r_1(2^{-l}(\zeta'))(2^{-l} \zeta_d + 2^{-l} |\zeta'| + r_2(2^{-l} \zeta'))^\alpha_+} d\zeta \\
&= \frac{2^{2\alpha l - dl}}{\Gamma(1-\alpha)} \int_{\R^d}\frac{e^{i 2^{-l} x.\zeta} \beta_0(\zeta) \hat{f}_l(\zeta)}{((\zeta_d - |\zeta'| + r_{1,l}(\zeta'))(\zeta_d + |\zeta'| + r_{2,l}(\zeta'))^\alpha_+} d\zeta,
\end{align*}
where $\hat{f}_l(\zeta) = \hat{f}(2^{-l} \zeta)$, $r_{i,l}(\zeta') = 2^{l} r_{i}(2^{-l} \zeta')$, $\xi = 2^{-l}
\zeta$. We therefore consider the operator
\begin{equation*}
S_l^\alpha g(y) = \frac{1}{\Gamma(1-\alpha)} \int_{\R^d}\frac{e^{iy.\xi} \beta_0(\xi) \hat{g}(\xi)}{((\xi_d - |\xi'|
+ r_{1,l}(\xi'))(\xi_d + |\xi'|  + r_{2,l}(\xi'))^\alpha_+} d\xi.
\end{equation*}
With $\text{supp} (\beta_0) \subseteq B(0,2) \backslash B(0,1/2)$, the subsets of $\text{supp}(\beta_0)$ where
the factors $\xi_d -|\xi'| + r_{1,l}(\xi')$ and $\xi_d + |\xi'| + r_{2,l}(\xi')$ vanish are separated.   We write
\begin{equation*}
\beta_0(\xi) = \beta_0(\xi) ( \gamma_0(\xi) + \gamma_1(\xi) + \gamma_2(\xi))
\end{equation*}
 with $\gamma_i \in C^\infty_c(\R^d)$ and $\text{supp}(\gamma_0) \subseteq \{ \xi \in \R^d : |\xi_d| \not\sim
 |\xi'| \}$, $\text{supp} (\gamma_i) \subseteq \{ \xi \in \R^d : (-1)^{i+1} \xi_d \sim |\xi'| \}$ for $i=1,2$. Correspondingly, we consider the operators $S_{l,i}^\alpha$ with
\begin{equation*}
\big( S^\alpha_{l,i} h\big) \widehat (\xi) = \frac{1}{\Gamma(1-\alpha)} \frac{\gamma_i(\xi) \beta_0(\xi)
\hat{h}(\xi)}{((\xi_d -|\xi'| + r_{1,l}(\xi'))(\xi_d+|\xi'| + r_{2,l}(\xi'))^\alpha_+}.
\end{equation*}
Clearly, $S_{l,0}^\alpha$ is bounded from $L^p(\R^d) \to L^q(\R^d)$ for $1 \leq p \leq q \leq \infty$ as the kernel is a Schwartz function. We shall only estimate $S_{l,1}^\alpha$ as $S_{l,2}^\alpha$ is treated \emph{mutatis mutandis}:
\begin{equation*}
(S_{l,1}^\alpha h) \widehat (\xi) 
= \frac{1}{\Gamma(1-\alpha)} \frac{\gamma_1(\xi) \beta_0(\xi) \hat{h}(\xi)}{((\xi_d - |\xi'| + r_{1,l}(\xi'))(\xi_d + |\xi'| + r_{2,l}(\xi'))^\alpha_+}.
\end{equation*}
With $m(\xi) = \xi_d + |\xi'| + r_{2,l}(\xi') \gtrsim \xi_d$  for $\xi \in \text{supp} (\beta_1) \cap \text{supp} (\beta)$ and $|\partial^\alpha m(\xi)| \lesssim 1$, by Young's inequality it is enough to consider $\tilde{S}^\alpha_{l,1}$ given by
\begin{equation*}
( \tilde{S}^\alpha_{l,1} g) \widehat (\xi) = \frac{1}{\Gamma(1-\alpha)} \frac{\gamma_1(\xi) \beta_0(\xi) \hat{g}(\xi)}{(\xi_d - |\xi'| + r_{1,l}(\xi'))^\alpha_+}.
\end{equation*}
To this operator, we can apply the estimates of Theorem~\ref{thm:GeneralizedBochnerRieszEstimates} for $k=d-2$
since in each point of the perturbed cone $d-2$ principal curvatures are bounded from below in modulus
uniformly with respect to $k$. Moreover, the rescaled surfaces $\{ \zeta_d = \mp |\zeta'| + r_{i,l}(\zeta') \}$ can be approximated with the cone in any $C^N$-norm. As a consequence, $S_k^\alpha$ has the mapping properties
described in Theorem~\ref{thm:GeneralizedBochnerRieszEstimates}~(i) for $\frac{1}{2}<\alpha<\frac{d}{2}$ with
a uniform mapping constant. From
\begin{equation*}
\mathcal C_l^\alpha f(x) = 2^{2 \alpha l - 3l} S_l^\alpha f_l(2^{-l} x)
\end{equation*}
we   conclude
\begin{align*}
\| \mathcal C_l^\alpha f \|_{L^q(\R^d)} 
&= 2^{2 \alpha l - dl} \| (S_l^\alpha f_l)(2^{-l} \cdot) \|_{L^q(\R^d)} \\
&\lesssim 2^{2 \alpha l - dl} 2^{\frac{dl}{q}} \| S_l^\alpha f_l \|_{L^q(\R^d)} \\
&\lesssim 2^{2 \alpha l -dl} 2^{\frac{dl}{q}} \| f_l \|_{L^p(\R^d)} \\
&= 2^{2\alpha l + \frac{dl}{q} -\frac{dl}{p}} \| f \|_{L^p(\R^d)}.
\end{align*}
Given that the conditions on $p,q$ imply $2 \alpha + \frac{d}{q} - \frac{d}{p}\leq 0$, we obtain the desired
uniform estimates for any fixed $l\in\N_0$. 
Hence, an application of Lemma \ref{lem:LittlewoodPaley} finishes the proof for $p \neq 1$, $q \neq \infty$
because of $p<2<q$. If $p>1,q= \infty$, we can find $q^*<\infty$ such that the conditions hold for
$(p,q^*)$. This is true because $\frac{1}{p}-\frac{1}{q}= 1-\frac{1}{q*} >\frac{2\alpha}{d}$ for large enough
$q^*$. Take $\chi$ a cut-off function with
$\chi=1$ on $\text{supp}(\beta)$. Then, by bounded frequencies and Young's inequality,
\begin{align*}
\| \mathcal C^\alpha f \|_{L^{\infty}(\R^d)} 
&\les \| \langle \mathcal C^\alpha f \|_{L^{q*}(\R^d)} \\
&=\| \mathcal C^\alpha ( \chi(D)f) \|_{L^{q*}(\R^d)}   \\
&\les\|  \chi(D) f \|_{L^p(\R^d)} \\
&\lesssim \|f \|_{L^p(\R^d)}
\end{align*}
The case $p=1,q<\infty$ is dual and thus proved as well. The proof is complete.

\end{proof}

\subsection{Estimates for approximate solutions close to the singular points}
\label{subsection:ApproximateSolutionsSingularPoints}
In this section we prove Proposition \ref{prop:SingularEstimate} by showing the corresponding $L^p$-$L^q$-bounds for
\begin{equation*}
\label{eq:MultiplierSingularPoints}
A_\delta f(x) = \int_{\R^3} \frac{e^{ix.\xi} \beta(\xi)}{\tilde p(\xi) + i\delta} \hat{f}(\xi)\,d\xi, 
\end{equation*}
where $\tilde p$, after some translation and dilation, has the form
\begin{equation}
\label{eq:ApproximateCone}
  \tilde p(\xi) = \xi_3^2 - \xi_1^2 - \xi_2^2 + g(\xi) \text{ with } |\partial^\alpha g(\xi)| \leq C_\alpha  |\xi|^{3-|\alpha|}
\end{equation} 
 and  $\text{supp}(\beta)\subset B(0,c)$ with $c$ as in Proposition~\ref{prop:ParametrizedCone}. Roughly
 speaking, this guarantees that the surface $\{ \tilde p(\xi) = 0 \}$ looks like a cone in $B(0,c)$. Due to
 the singularity at the origin, this seems problematic, but can be remedied by Littlewood-Paley
 decomposition.

We proceed similar as above.
  Let $\beta_0 \in C^\infty_c(\R^3)$ with $\text{supp }(\beta_0) \subseteq \{ c/2 \leq |\xi| \leq 2c \}$ and $\beta_{\ell}(\xi) = \beta_0(2^{\ell} \xi)$, $\ell \geq 1$, such that
\begin{equation*}
	\sum_{\ell \geq 0} \beta_\ell \cdot \beta_{13}(\xi) = \beta_{13}(\xi) \quad (\xi \neq 0).  
\end{equation*}
  We further set $\tilde{\beta}_\ell(\xi) = \beta_{\ell-1}(\xi) + \beta_{\ell}(\xi) + \beta_{\ell+1}(\xi)$.
As in the previous section, we have the following lemma by Littlewood-Paley theory and Minkowski's inequality:
\begin{lemma}  
\label{lem:LittlewoodPaleySingularMultiplier}
Let $1<p \leq 2 \leq q < \infty$, $r_1 \in \{1,p\}$, and $r_2 \in \{q,\infty\}$. Suppose that
\begin{equation}
\label{eq:UniformEstimateMultiplier}
 \left\| \int_{\R^3}\frac{\hat{f}(\xi) \beta_\ell(\xi) e^{ix.\xi}}{p(\xi) + i \delta} d \xi
 \right\|_{L^{q,r_2}(\R^3)} \leq C \| \tilde{\beta}_\ell(D) f \|_{L^{p,r_1}(\R^3)}
\end{equation}
holds for $C$ independent of $\ell$ and $\delta \neq 0$. Then we have
\begin{equation*}
\| A_\delta f \|_{L^{q,r_2}(\R^3)} \lesssim \| f \|_{L^{p,r_1}(\R^3)}.
\end{equation*}
 \end{lemma}
  
 We prove \eqref{eq:UniformEstimateMultiplier} for $\ell=0$ and see how the remaining estimates follow by
 rescaling as in the previous subsection. In the first step we localize to the singular set: Let
\begin{equation*}
\beta_0 = \beta_0 ( \beta_{01} + \beta_{02}), \quad \beta_{0i} \in C^\infty_c(\R^3)
\end{equation*}
with 
\begin{align*}
\text{supp }(\beta_{01}) &\subseteq \{ \xi \in \R^3 : \, |\xi_3| \sim |\xi'| \}, \\
\text{supp }(\beta_{02}) &\subseteq \{ \xi \in \R^3 : \, |\xi_3| \ll |\xi'|, \, |\xi'| \ll |\xi_3| \}.  
\end{align*}
We start with noting that in the support of $\beta_0 \beta_{02}$ we find $|p(\xi)| \gtrsim c^2$ and uniform boundedness of
\begin{equation*}
\left\| \int_{\R^3}\frac{\hat{f}(\xi) e^{ix.\xi} \beta_0 \beta_{02} }{p(\xi) + i \delta} d\xi \right\|_{L^q(\R^3)} \lesssim \| f \|_{L^p(\R^3)}
\end{equation*}  
is immediate from Young's inequality as the kernel is a Schwartz function.\\
We turn to the estimate of the contribution close to the vanishing set of $p$: Let $\chi(\xi) = \beta_0(\xi) \beta_{01}(\xi)$. We follow the arguments of Section \ref{subsection:UniformEstimatesSingularMultiplierI}: We decompose
\begin{equation*}
\frac{\chi(\xi)}{p(\xi) + i \delta} = \mathfrak{R}(\xi) + i \mathfrak{I}(\xi).
\end{equation*}
These multipliers will be estimated by Fourier restriction-extension estimates for the level sets of $p$
given by Theorem~\ref{thm:GeneralizedBochnerRieszEstimates} for $k=d-2=1$, $\alpha=1$. To carry out the
program of Section \ref{subsection:UniformEstimatesSingularMultiplierI}, we need to change to generalized polar coordinates $\xi = \xi(p,q)$ in $\text{ supp }(\chi)$. We can suppose that this is possible as $|\nabla p(\xi)| \gtrsim c > 0$ for $p(\xi) = 0$, $|\xi| \sim c$, after making the support of $\beta_{01}$ closer to the characteristic set, if necessary.
Furthermore, with graph parametrizations $(\xi',\psi(\xi'))$ of $\{ \xi \in \text{supp}(\chi) \, : \, p(\xi) =
t\} $ uniform in $t \in (-t_0,t_0)$, $t_0$ chosen small enough, Theorem \ref{thm:GeneralizedBochnerRieszEstimates} yields uniform bounds. Also note that Lemma \ref{lem:CSum} applies with $k=1$. This finishes the proof of \eqref{eq:UniformEstimateSingular} for $\ell=0$. We show the bounds for $\ell \geq 1$ by rescaling. A change of variables gives
\begin{equation}
\label{eq:RescalingMultiplier}
\begin{split}
&\quad  \int_{\R^3}\frac{e^{ix.\xi} \beta_\ell(\xi) \hat{f}(\xi)}{p(\xi) + i \delta} d\xi \\
&= 2^{-3 \ell} \int_{\R^3}\frac{e^{i2^{-\ell} x.\zeta} \beta_0(\zeta) \hat{f}(2^{-\ell} \zeta)}{2^{-2\ell} \zeta_3^2 - 2^{-2\ell} \zeta_1^2 - 2^{-2\ell} \zeta_2^2 + g(2^{-\ell} \zeta) + i\delta} d\zeta \qquad (\zeta = 2^\ell \xi) \\
&= 2^{-\ell} \int_{\R^3}\frac{e^{i 2^{-\ell} x.\zeta} \beta_0(\zeta) \hat{f}_\ell(\zeta)}{\zeta_3^2 - \zeta_1^2 -\zeta_2^2 +2^{2\ell} g(2^{-\ell} \zeta) + i 2^{2\ell} \delta} d\zeta \qquad (\hat{f}_\ell(\zeta) = \hat{f}(2^{-\ell} \zeta)).
\end{split}
\end{equation}
Let $p_\ell(\zeta ) = \zeta_3^2 - \zeta_1^2 -\zeta_2^2 +2^{2\ell} g(2^{-\ell} \zeta)$, $\delta_{\ell} = 2^{2 \ell} \delta$. Recall that $| \partial^\alpha g(\xi) | \lesssim |\xi|^{3-|\alpha|}$, which previously allowed to carry out the proof for $\ell=0$ for $c$ chosen small enough depending on finitely many $C_\alpha$ in \eqref{eq:ApproximateCone}. Furthermore, we find
\begin{equation}
\label{eq:EstimateUniformK}
\| \int_{\R^3}\frac{e^{ix.\xi} \beta_0(\zeta) \hat{h}(\zeta)}{p_\ell(\zeta) + i \delta_\ell } d\zeta \|_{L^q(\R^3)} \lesssim \| h \|_{L^p(\R^3)}
\end{equation}  
with implicit constant independent of $\ell \geq 1$ choosing $c$ small enough depending only on finitely many $C_\alpha$. Hence, taking \eqref{eq:RescalingMultiplier} and \eqref{eq:EstimateUniformK} together, gives
\begin{equation*}
\begin{split}
\big\| \int_{\R^3}\frac{e^{ix.\xi} \beta_\ell(\xi) \hat{f}(\xi)}{p(\xi) + i \delta} d\xi \big\|_{L^q(\R^3)} 
&\lesssim 2^{\frac{3\ell}{q} -\ell} \big\| \int_{\R^3}\frac{e^{i x.\zeta} \beta_0(\zeta) \hat{f}_\ell(\zeta)}{p_\ell(\zeta) + i \delta_\ell} d\zeta \big\|_{L^q(\R^3)} \\
&\lesssim 2^{\frac{3\ell}{q} -\ell} \| f_\ell \|_{L^p(\R^3)} \\
&\lesssim 2^{2\ell - \frac{3\ell}{p} + \frac{3\ell}{q}} \| f \|_{L^p(\R^3)}.
\end{split}
\end{equation*}
Hence, Lemma \ref{lem:LittlewoodPaleySingularMultiplier} applies for $p \neq 1$ and $q \neq \infty$ because for our choice of $p$ and $q$ we find $\frac{2}{3} \leq \frac{1}{p} - \frac{1}{q}$. For $q = \infty$ or $p=1$, we use that frequencies are compactly supported to reduce to $p \neq 1$ and $q \neq \infty$ like at the end of the proof of Proposition \ref{prop:GeneralizedConeMultiplier}. The proof of Proposition \ref{prop:SingularEstimate} is complete. \qed

\section*{Acknowledgements}
Funded by the Deutsche Forschungsgemeinschaft (DFG, German Research Foundation) - Project-ID 258734477 - SFB 1173. We would like to thank Roland Schnaubelt (KIT) for useful remarks on the invariance of Maxwell's equations in media under change of basis. The involved computations for the Fresnel surface and its visualization in Figure \ref{fig:Fresnel} were performed by using MAPLE\textsuperscript{TM}.

\section*{Note on published version}

This version of the article has been accepted for publication, after peer review (when applicable)
but is not the Version of Record and does not reflect post-acceptance improvements, or any
corrections. The Version of Record is available online at: http://dx.doi.org/10.1007/s00023-021-01144-y.

\end{document}